\documentclass[11pt,a4paper,reqno]{amsart}

 \usepackage{cases,hyperref,tikz} \usepackage{xcolor,appendix,geometry}
\usepackage{mathrsfs}
\usepackage{epsfig}
\usepackage{fp,ifthen}
\usepackage{float}
\usepackage{graphicx}
\usepackage{enumitem}
\usepackage{esint}
\usepackage{accents}
\usepackage{mathtools}
\usepackage{cite}
\usepackage{amsmath,amssymb,amsthm,esint,empheq,etoolbox}

\title[$\Gamma$-convergence for free discontinuity functionals in ${\rm BD}$] {$\mathbf{\Gamma}$-convergence and homogenisation for free discontinuity functionals with linear growth in the space of functions with bounded deformation}
\author{Gianni Dal Maso}
\address{SISSA, via Bonomea 265, Trieste, Italy}
\email{dalmaso@sissa.it}
\author{Davide Donati}
\address{SISSA, via Bonomea 265, Trieste, Italy}
\email{ddonati@sissa.it}
\numberwithin{equation}{section}
\theoremstyle{plain} 
\newtheorem{theorem}{Theorem}[section] \newtheorem{lemma}[theorem]{Lemma} 
\newtheorem{proposition}[theorem]{Proposition}
\newtheorem{corollary}[theorem]{Corollary}
\theoremstyle{definition} 
\newtheorem{definition}[theorem]{Definition}
\newtheorem{example}[theorem]{Example} \newtheorem{remark}[theorem]{Remark}

\newcommand{\dx}{{\rm d}}

\newcommand\e{\varepsilon}
\newcommand{\N}{\mathbb{N}}
\newcommand{\Z}{\mathbb{Z}}
\newcommand\R{\mathbb{R}}
\newcommand{\B}{\mathcal{B}}
\newcommand{\E}{\mathcal{E}}
\newcommand\Rd{\mathbb{R}^d}
\newcommand\Sd{\mathbb{S}^{d-1}}

\newcommand{\Rdsym}{\mathbb{R}^{d\times d}_{\rm sym}}
\newcommand{\Rdskew}{\R^{d\times d}_{\rm skew}}
\newcommand{\Rdd}{\R^{d\times d}}
\newcommand{\hd}{\mathcal{H}^{d-1}}
\newcommand\Hd[1]{\mathcal{H}^{d-1}(#1)}

\newcommand{\Lb}{\mathcal{L}}
\newcommand{\Ld}{\mathcal{L}^d}

\newcommand{\mres}
{\mathbin{\vrule height 1.6ex depth 0pt width
0.13ex\vrule height 0.13ex depth 0pt width 1.3ex}}
\mathchardef\emptyset="001F
\newcommand{\U}{\mathcal{U}}
\newcommand{\m}{\mathfrak{m}}
\makeatletter
\newcommand{\leqnomode}{\tagsleft@true\let\veqno\@@leqno}
\makeatother

\theoremstyle{remark}
\begingroup
\endgroup

\begin{document}

\thanks{Preprint SISSA 01/2026/MATE}
\begin{abstract} 
\smallskip
 We study the $\Gamma$-convergence of sequences of free discontinuity functionals with linear growth defined in the space ${\rm BD}$ of functions with bounded deformation. We prove a  compactness result with respect to $\Gamma$-convergence and outline the main properties of the $\Gamma$-limits, which lead to an integral representation result. The corresponding integrands are obtained by taking limits of suitable minimisation problems on small cubes.
 These results are then  used to study the deterministic and stochastic homogenisation problem for a large class of free discontinuity functionals defined in ${\rm BD}$.
\end{abstract}

\maketitle
\vspace{-0.6 cm}


\section{Introduction}

    Given a bounded open set $U\subset \Rd$, the space ${\rm BD}(U)$ of  {\it functions of bounded deformation} is defined as the space of vector fields $u\in L^1(U;\Rd)$ such that the symmetric part $Eu:=\frac{1}{2}(Du+Du^T)$ of the distributional gradient $Du$ is a bounded Radon measure with values in the space $\Rdsym$ of $(d\times d)$ symmetric matrices.  It was introduced in \cite{Matthies,suquet, Temam}  to provide a variational framework for the study  of elasto-plasticity in the small strain regime and has since been object of a large number of contributions, where the authors investigated trace and fine properties \cite{Babdjian,AmbCoscDalM,DePhilRindler,Hajlasz,kohn1979new,ContiFocardiIurlano,arroyo2020slicing}, approximation by means of more regular functions \cite{BabadjianIurlano,CrismaleSIAM,ChambolleJMPA,ChambolleJMPA2, ChambolleConti2}, and rigidity estimates \cite{ChambolleContiFranfort,DiFrattaSolombrino,Friedrich,FriedrichM3MA,ChambPonsGiacom}.  

In addition to its applications to elasto-plasticity,
this space is particularly useful  in the mathematical modelling of fracture mechanics in the variational formulation of Francfort and Marigo introduced in \cite{FrancfortMarigo}. For this application, it is convenient to consider the decomposition of the measure $Eu$ introduced in \cite{AmbCoscDalM} and given by  
\begin{equation}\label{intro decomposition}
    Eu=E^au+E^cu+E^ju,
\end{equation}
where $E^au$ is the absolutely continuous part of $Eu$ with respect to the  Lebesgue measure $\Ld$, whose density is denoted by $\E u$, $E^ju$ is the jump part of $Eu$, defined as the restriction of $Eu$ to the jump set $J_u$ of $u$, and the Cantor part $E^cu$ is the restriction of the singular part of the measure $Eu$ to the complement of $J_u$ in $U$.
It is known that $E^ju:=([u]\odot\nu_u)\hd\mres J_u$, where $[u]:=u^+-u^-$ is the difference of  the unilateral traces $u^+$ and $u^-$ of $u$ on $J_u$, $\nu_u$ is the unit normal to $J_u$, $\odot$ denotes the symmetrised tensor product, $\hd$ is the $(d-1)$-dimensional Hausdorff measure, and $\mres$ denotes the restriction of a Borel measure to a Borel set. Moreover, it is known that $E^cu$ is singular with respect to the Lebesgue measure and vanishes on all Borel sets with finite $\hd$-measure.

In some {\it cohesive} fracture models, the approach of \cite{FrancfortMarigo} naturally leads to the minimisation of energies of the form (see, for instance, \cite[Sections 4.2 and 5.2]{Bourdin})
\begin{equation}\label{intro cohesive energies}
   F(u,U):= \int_{U}f(x,\E u)\,{\rm d}x+\int_{U}f^\infty\Big(x,\frac{{\rm d} E^cu}{{\rm d}|E^cu|}\Big)\,{\rm d}|E^cu|+\int_{J_u\cap U}g(x,[u],\nu_u)\,{\rm d}\hd,
\end{equation}
where  the {\it bulk} energy density $f\colon \Rd\times \Rdsym\to [0,+\infty)$ is a Borel function with 
\begin{equation}\label{intro bound f}
    c_1|A|-c_2\leq f(x,A)\leq c_3|A|+c_4 \quad \text{ for every }x\in\Rd \text{ and }A\in\Rdsym,
\end{equation}
for some constants $0<c_1\leq c_3$, $c_2\geq 0$, and $c_4\geq 0$,  
$f^\infty\colon \Rd\times \Rdsym\to [0,+\infty)$ denotes the {\it recession function} of $f$ with respect to $A$, defined by 
\begin{equation*}
    f^\infty(x,A):=\limsup_{t\to+\infty}\frac{f(x,tA)}{t}\quad \text{ for every }x\in \Rd \text{ and }A\in\Rdsym,
\end{equation*}
and the {\it surface} energy density $g\colon \Rd\times \Rd\times \Sd\to[0,+\infty)$ is a Borel function satisfying 
\begin{equation}\label{intro bound g} 
    c_1|\zeta\odot \nu|\leq g(x,\zeta,\nu)\leq c_3|\zeta\odot \nu|\quad \text{ for every }x\in\Rd, \, \zeta\in\Rd, \text{ and }\nu\in\Sd,
\end{equation}
where $\Sd:=\{\nu\in \Rd:|\nu|=1\}$.

In this paper we study the $\Gamma$-limits of sequences of the form \eqref{intro cohesive energies} with respect to the topology induced by $L^1_{\rm loc}$. The first result in this direction is Theorem \ref{compactness}, which states that for every sequence $\{F_n\}_{n\in\N}$, with $c_1,\dots,  c_4$ independent of $n$, there exists a subsequence, not relabelled, such that for every bounded open set $U \subset \Rd$ the sequence $\{F_n(\cdot,U)\}_{n\in\N}$ $\Gamma$-converges to a functional $F(\cdot, U)$ defined on ${\rm BD}(U)$. 

This limit functional $F$ satisfies suitable regularity properties (see Definition \ref{abstract functionals}), the most important being that for every bounded open set $U\subset \Rd$ and $u\in {\rm BD}(U)$ the set function defined for every Borel set $B\subset U$ by
\begin{equation}\notag
    B\mapsto F(u,B):=\inf\{F(u,V): V \text{ open},\,  B\subset V\subset U\}
\end{equation}
is a bounded Radon measure. Let $F^a(u,\cdot)$ and $F^s(u,\cdot)$ be the absolutely continuous and singular parts of  the measure $F(u,\cdot)$ with respect to $\Ld$, respectively.  In analogy to \eqref{intro decomposition}, we  decompose the limit functional $F$ as 
\begin{equation*}
    F(u,B)=F^a(u,B)+F^c(u,B)+F^j(u,B),
\end{equation*}
where  $F^j(u,B):=F^s(u,B\cap J_u)$ and $F^c(u,B):=F^s(u,B\setminus J_u)$. It is possible to check (see Remark \ref{remark abs continuity}) that $F^c(u,\cdot)$ is the absolutely continuous part of $F(u,\cdot)$ with respect to $E^cu$ and that $F^j(u,\cdot)$ is the absolutely continuous part with respect to $E^ju$. 

By recent results of Caroccia, Focardi, and Van Goethem \cite{CaroccFocardiVan} (see also \cite{EbobisseToader}) it is possible to represent $F^a$ and $F^j$ by means of  integral functionals in the form 

\begin{gather}\label{intro bulk}
    F^{a}(u,B)=\int_{B}f(x,\E u)\,{\rm d}x,\\
    \label{intro surface}F^{j}(u,B)=\int_{J_u\cap B}g(x,[u],\nu_u)\,{\rm d}\hd,
\end{gather}
where  $f$ and $g$ are Borel functions satisfying \eqref{intro bound f} and \eqref{intro bound g}. As customary in integral representation results for free discontinuity functionals (see, for instance,  \cite{BraidesChiadoPiat,bouchitte1998global}), the functions $f$ and $g$ are defined by means of limits of auxiliary minimisation problems on small cubes. More precisely, for every bounded open set $W \subset \Rd$ with Lipschitz boundary, and $w\in {\rm BD}(W)$ we set 
\begin{equation*}
    \m^F(w,W):=\inf\{F(u,W): u\in {\rm BD}(W) \text{ and }u=w \text{ on }\partial W\}.
\end{equation*}
The integrands $f$ and $g$ that appear in \eqref{intro bulk} and \eqref{intro surface} are given by
\begin{gather}\label{intro def bulk} 
    f(x,A):=\limsup_{\rho \to 0^+}\frac{\m^F(\ell_A,Q(x,\rho))}{\rho^d},\\\notag 
     g(x,\zeta,\nu):=\limsup_{\rho \to 0^+}\frac{\m^F(u_{x,\zeta,\nu},Q_\nu(x,\rho))}{\rho^{d-1}},
\end{gather}
where  $\ell_A$ is the linear function defined by  $\ell_A(y)=Ay$, $Q(x,\rho)$ is the cube with centre $x$, side-length $\rho>0$, and sides parallel to the axes, 
   $u_{x,\zeta,\nu}$ is the pure jump function equal to $\zeta$ on $\{y\in\Rd:(y-x)\cdot\nu\geq 0\}$ and to $0$ on $\{y\in\Rd:(y-x)\cdot\nu< 0\}$,  while $Q_\nu(x,\rho)$ is a cube with centre $x$, side-length $\rho>0$, and  two faces orthogonal to $\nu$.

Under suitable additional hypotheses (see Definition \ref{def Falpha}) on the sequence $\{F_n\}_{n\in\N}$, and assuming that \eqref{intro def bulk} holds with $f$ independent of $x$ and $\limsup$ replaced by $\lim$, we show in Theorem \ref{theorem cantor} that the complete limit functional $F$, including its Cantor part $F^c$, can be represented in the form
\begin{equation}\label{intro representation}
   F(u,B)= \int_{B}f(\E u)\,{\rm d}x+\int_{B}f^\infty\Big(\frac{{\rm d} E^cu}{{\rm d}|E^cu|}\Big)\,{\rm d}|E^cu|+\int_{J_u\cap B}g(x,[u],\nu_u)\,{\rm d}\hd,
\end{equation}
for every $u\in {\rm BD}(U)$, with $U\subset \Rd$ bounded open set, and every Borel set $B\subset U$. Thanks to \eqref{intro bulk} and \eqref{intro surface} the crucial point is to show that 
\begin{equation}\notag 
   F^c(u,B)=\int_{B}f^\infty\Big(\frac{{\rm d} E^cu}{{\rm d}|E^cu|}\Big)\,{\rm d}|E^cu|.
\end{equation}

In Caroccia, Focardi, and Van Goethem \cite{CaroccFocardiVan}, this is obtained by exploiting the uniform  continuity of $F$ with respect to horizontal translations, which  is one of their key assumptions. However, this hypothesis is not natural for our approach to stochastic homogenisation, so we are forced to use a different technique.
First of all, we use the characterisation of
\begin{equation*}
    \frac{{\rm d}F^c(u,\cdot)}{{\rm d}|E^cu|}(x)
\end{equation*} 
obtained in \cite[Lemma 5.3]{CaroccFocardiVan}, which holds even if  $\{F_n\}_{n\in\N}$ is not uniformly continuous, and then, using the property introduced in Definition \ref{def Falpha}, we  adapt to the ${\rm BD}$ setting  some arguments used in \cite{DalToaHomo,DalMasoDonati2025} for the ${\rm BV}$ case.

Combining the compactness result  and the integral representation \eqref{intro representation} we are able to obtain in Theorem \ref{theorem sufficient} a characterisation of the integrands of the $\Gamma$-limits by means of limits  of minimum values on small cubes. More precisely, assuming that $F_n$ satisfy the property described in Definition \ref{def Falpha} uniformly with respect to $n$ and  that for some functions $\hat{f}$ and $\hat{g}$ we have 
 \begin{gather}\label{intro condition 1}
\hat{f}(A)=\lim_{\rho\to 0^+}\limsup_{n\to+\infty}\frac{\m^{F_n}(\ell_A,Q(x,\rho))}{\rho^d}=\lim_{\rho\to 0^+}\liminf_{n\to+\infty}\frac{\m^{F_n}(\ell_A,Q(x,\rho))}{\rho^d},\\\label{intro condition 2}
\hat{g}(x,\zeta,\nu)=\limsup_{\rho\to 0^+}\limsup_{n\to+\infty}\frac{\m^{F_n}(u_{x,\zeta,\nu},Q_\nu(x,\rho))}{\rho^{d-1}}=\limsup_{\rho\to 0^+}\liminf_{n\to+\infty}\frac{\m^{F_n}(u_{x,\zeta,\nu},Q_\nu(x,\rho))}{\rho^{d-1}}
    \end{gather}
for every $x\in\Rd$, $A\in\Rdsym$, $\zeta\in\Rd$, and $\nu\in\Sd$, then for every bounded open set  $U \subset \Rd$ we prove that the sequence $\{F_n(\cdot,U)\}_{n\in\N}$ $\Gamma$-converges to the functional $F(\cdot,U)$ in \eqref{intro representation} with  $f=\hat{f}$ and $g=\hat{g}$.

We conclude the paper by applying these results to a general class of integral functionals to deduce homogenisation results. More precisely, in Theorems \ref{homogeneous homogenization} and \ref{stochastic homo} we will consider 
 functionals $F_n$ given by 
\begin{equation}\label{intro to homogenise}
\hspace{-0.2 cm}F_n(u,U):= \!\!\int_{U}f\Big(\frac{x}{\e_n},\E u\Big)\,{\rm d}x+\int_{U}f^\infty\Big(\frac{x}{\e_n},\frac{{\rm d} E^cu}{{\rm d}|E^cu|}\Big)\,{\rm d}|E^cu|+\int_{J_u\cap U}\hspace{-0.2 cm}\e_n g\Big(\frac{x}{\e_n},\frac{1}{\e_n}[u],\nu_u\Big)\,{\rm d}\hd,    
\end{equation}
where $f$ and $g$ satisfy \eqref{intro bound f},  \eqref{intro bound g}, and some conditions related to Definition \ref{def Falpha} (see Definition \ref{defizione integrande rappresentabili}), while $\{\e_n\}_{n\in\N}\subset (0,1)$ is a sequence converging to $0$ as $n\to+\infty$.
We remark that, in general the functionals $\{F_n\}_{n\in\N}$ do not coincide with those normally considered in the homogenisation  of free discontinuity problems, unless one assumes that $[0,+\infty)\ni t\mapsto g(x,t\zeta,\nu)$ is positively homogeneous of degree one, in which case  $F_n(u,U)$ reads as  
\begin{equation}
\label{intro homo Bv}
\int_{U}f\Big(\frac{x}{\e_n},\E u\Big)\,{\rm d}x+\int_{U}f^\infty\Big(\frac{x}{\e_n},\frac{{\rm d} E^cu}{{\rm d}|E^cu|}\Big)\,{\rm d}|E^cu|+\int_{J_u\cap U} g\Big(\frac{x}{\e_n},[u],\nu_u\Big)\,{\rm d}\hd, \end{equation}
which is the standard functional considered in the homogenisation problem for free discontinuity functionals (see \cite{Braides1996,CagnettiStocFree,CagnettiAnnals,CagnettiDetFree}). 

So far, in the ${\rm BD}$ setting  our technique allows us to prove the $\Gamma$-convergence of \eqref{intro homo Bv} only when $g$ is positively $1$-homogeneous in the variable $\zeta$, while in the ${\rm BV}$ setting (see \cite{CagnettiAnnals,DalToaHomo,DalMasoDonati2025})  the $\Gamma$-convergence of the functionals corresponding to \eqref{intro homo Bv} has been proved assuming only that the function $g$ is sufficiently close to a $1$-homogeneous function for $|\zeta|$ small enough. Unfortunately, this approach  heavily relies on sophisticated truncation arguments that we are not able to extend to ${\rm BD}$. 

The particular choice of the scaling for the surface term
in \eqref{intro to homogenise}
will allow us to circumvent the use of vertical truncations and will allow us to obtain suitable change of variable formulas used in the proof, which avoids any truncation argument.

A similar problem for bulk energies of the form
\begin{equation*}
    \Phi_n(u,U):=\begin{cases}\displaystyle
        \int_{U}f\Big(\frac{x}{\e_n},\E u\Big)\,{\rm d}x &\text{ if }Eu<\!<\Ld,\\
        +\infty &\text{ otherwise,}
    \end{cases}
\end{equation*}
was studied in \cite{CaroccFocardiVan}, when $f$ is $1$-periodic with respect to $x$ and satisfies \eqref{intro bound f}, and in \cite{Mandallena} in the stochastically periodic case.

To obtain the $\Gamma$-limit of $\{F_n\}_{n\in\N}$ using \eqref{intro condition 1} and \eqref{intro condition 2} is convenient to rewrite these formulas using the change of variables $y=\frac{x}{\e_n}$. Setting $v(y):=\frac{1}{\e_n}u(\e_ny)$,  one then checks that 
\begin{gather*}
F_n(u,Q(x,\rho))=\e_n^d\Big(\int_{Q(\frac{x}{\e_n},\frac{\rho}{\e_n})}f(y,\E v)\,{\rm d}y+\int_{Q(\frac{x}{\e_n},\frac{\rho}{\e_n})}f^\infty\Big(y,\frac{{\rm d} E^cv}{{\rm d}|E^cv|}\Big)\,{\rm d}|E^cv|\\
+\int_{J_v}g(y,[v],\nu_u)\,{\rm d}\hd\Big),
\end{gather*}
where we have exploited the scaling chosen for the surface term in \eqref{intro to homogenise}. Setting $r_n:=\frac{\rho}{\e_n}$, we have
\begin{equation*}
 \frac{1}{\rho^d}\m^{F_n}(\ell_A,Q(x,\rho))=\frac{1}{r_n^d}\m^{F}(\ell_A,Q(r_n\tfrac{x}{\rho},r_n)),
\end{equation*}
where $F$ is given by \eqref{intro cohesive energies}.
Hence, if there exists the limit
\begin{equation}\label{intro stochastic bulk}
    f_{\rm lim}(A):=\lim_{n\to+\infty}\frac{1}{r_n^d}\m^{F}(\ell_A,Q(x,r_n))
\end{equation}
and is independent of $x$, condition \eqref{intro condition 1} is satisfied by $\hat{f}(x,A)=f_{\rm lim}(A)$. 

To deal with \eqref{intro condition 2} we consider $F^\infty$  the functional  obtained by replacing $f$ and $g$ by $f^\infty$  and $g^\infty$ in \eqref{intro cohesive energies}, where 
\begin{equation*}
g^\infty(x,\zeta,\nu):=\limsup_{t\to+\infty}\frac{g(x,t\zeta,\nu)}{t}
\end{equation*}
(see Definition \ref{defizione integrande rappresentabili}). Performing a change of variables similar to the one that led to \eqref{intro stochastic bulk},  we obtain that if the limit 
\begin{equation}\label{intro stocastic surface}
    g_{\rm lim}(\zeta,\nu):=\lim_{n\to+\infty}\frac{1}{r_n^{d-1}}\m^{F^\infty}(u_{r_nx,\zeta,\nu},Q(x,r_n))
\end{equation}
exists for every $\zeta\in\Rd$ and $\nu\in\Sd$ and is independent of $x$, then condition \eqref{intro condition 2} holds with $\hat{g}(x,\zeta,\nu)=g_{\rm lim}(\zeta,\nu)$. 

As $\eqref{intro condition 1}$ and \eqref{intro condition 2} are sufficient conditions for the $\Gamma$-convergence, we obtain the following result (see Theorem \ref{homogeneous homogenization}): suppose that the limits \eqref{intro stochastic bulk} and \eqref{intro stocastic surface} exist for every $A\in\Rdsym$, $\zeta\in\Rd$, $\nu\in\Sd$, and are independent of $x$, then for every bounded open set  $U \subset \Rd$ the sequence $\{F_n(\cdot,U)\}_{n\in\N}$ $\Gamma$-converges with respect to the $L^1_{\rm loc}$-convergence to the functional $F_{\rm lim}(\cdot,U)$ defined by
\begin{equation}\notag
    F_{\rm lim}(u,U):=\int_{U}f_{\rm lim}(\E u)\,{\rm d}x+\int_{U}f^\infty_{\rm lim}\Big(\frac{{\rm d} E^cu}{{\rm d}|E^cu|}\Big)\,{\rm d}|E^cu|+\int_{J_u}g_{\rm lim}([u],\nu_u)\,{\rm d}\hd.
\end{equation}

 We also prove a related result in the probabilistic setting. Under the standard assumptions of stochastic homogenisation (see Section \ref{subsection stochastic}), we will assume that the integrands $f$ and $g$ appearing in \eqref{intro to homogenise} are random functions, and we will prove that the limit \eqref{intro stochastic bulk} exists almost surely thanks to the Subadditive Ergodic Theorem of Akcoglu and Krengel \cite{AkcogluKrengel}, and show that the almost sure existence of the limit in  \eqref{intro stocastic surface} can be obtained by using the same theorem in dimension $d-1$, arguing as in \cite{CagnettiStocFree,CagnettiAnnals}. This will allow us to prove in Theorem \ref{stochastic homo} the almost sure $\Gamma$-convergence of $\{F_{n}\}_{n\in\N}$ to $F_{\rm lim}$. Of course, the stochastic result immediately implies a related result in the deterministic periodic case (see Corollary \ref{periodic corollary}).

\section{Notation}\label{sec: Notation}

In this section we present the notation used throughout the paper.
\begin{itemize}
    \item[(a)] The space $\Rd$ is endowed with the usual scalar product,  denoted by $\cdot$,  while the  Euclidean norm is denoted by $|\,\,|$. The unit sphere of $\Rd$ is denoted by $\Sd:=\{\nu\in\Rd:|\nu|=1\}$. We also set $\Sd_{\pm}:=\{\nu\in\mathbb{S}^{d-1}: \pm \nu_{i(\nu)}>0\}$, where $i(\nu)\in\{1,...,d\}$ is the largest index such that $\nu_{i(\nu)}\neq 0$.
\smallskip

    \item[(b)]
    
    We identify the vector space $\R^{d\times d}$ with the space of $d\times d$ matrices. The subspace of  $\Rdd$  of $d\times d$ symmetric matrices (resp.\ antisymmetric) is denoted by $\Rdsym$ (resp.\ $\R^{d\times d}_{\rm skew}$). If $A\in\Rdd$ its symmetric part is denoted by  $A^{\rm sym}:=\frac{1}{2}\big(A+A^T\big)$.  For $A\in\R^{d\times d}$ and  $x\in\Rd$, the vector $A x\in\Rd$ is given by the usual matrix by vector multiplication. Given a matrix $A=(A_{ij})\in\Rdd$, its Frobenius norm  is given by
    \begin{equation*}
        |A|:=\Big(\sum_{i,j=1}^d|A_{ij}|^2\Big)^{1/2}.
    \end{equation*}
    The identity matrix is denoted by $I$.

     \item[(c)] The $i$-th vector of the canonical basis of $\Rd$ is denoted by $e_i$. Given $x\in\Rd$ and $\rho>0$, we consider the cube $Q(x,\rho):=\big\{y\in\Rd:|(y-x)\cdot e_i|<\rho/2 \,\,\text{ for every }i\in \{1,...,d\}\big\}$.
     \item[(d)] $SO(d)$ is the space of  $d\times d$ orthonormal matrices $R$ with $\text{det}(R)=1$.  We fix once and for all a map $\Sd\ni\nu\mapsto R_\nu\in SO(d)$ satisfying the following properties: $R_\nu e_d=\nu$ and $R_\nu(Q(0,\rho))=R_{-\nu}(Q(0,\rho))$ for every $\nu\in\mathbb{S}^{d-1}$, $R_{e_d}=I$, and 
      the restrictions of $\nu\mapsto R_\nu$ to $\mathbb{S}^{d-1}_{\pm}$  are  continuous  (for an example of such map see, for instance,  \cite[Example A.1]{CagnettiStocFree}).
     \item[(e)] 
Given $x\in\Rd$, $\nu\in\Sd$, and $\rho>0$, we consider the open cube defined by 
\begin{equation*}
    Q_\nu(x,\rho):=x+R_\nu\big(Q(0,\rho) \big).
\end{equation*}
\item[(f)] For every $A\in\Rdd$, $x\in\Rd$, $\zeta\in\Rd$, and $\nu\in\Sd$ let $\ell_A\colon\Rd\to \Rd$ and $u_{x,\zeta,\nu}\colon\Rd\to \Rd$ be the functions defined for every $y\in\Rd$ by 
\begin{equation}\notag
    \ell_A(y):=Ay \quad\text{ and }\quad   u_{x,\zeta,\nu}(y):=\begin{cases}\zeta &\text{ if }(y-x)\cdot\nu>0,\\
    0 & \text{ if }(y-x)\cdot\nu\leq 0.
        
    \end{cases}
\end{equation}

\item[(g)] Given an open set $U\subset \Rd$, the collection of all open subsets (resp.\ Borel) of $U$ is denoted by $\mathcal{U}(U)$ (resp.\ $\mathcal{B}(U)$). If $V\in \U(\Rd)$,  we write $V\subset \subset U$ if the closure of $V$ is contained in $U$. The collection of all these open sets is denoted by $\U_c(U)$. 

\item[(h)]Given $U\in\U(\Rd)$ and a finite dimensional real normed vector space $X$, the space of all bounded Radon measures with values in $X$ is denoted by $\mathcal{M}_b(U;X)$. The indication of $X$ is omitted if $X=\R$.  Given a positive measure $\lambda\in \mathcal{M}_b(U)$, and a measure  $\mu \in\mathcal{M}_b(U;X)$ with $\mu<\!<\lambda$,  the function  ${\rm d}\mu/{\rm d}\lambda$ denotes the Radon-Nikod\'ym derivative of $\mu$ with respect to $\lambda$, defined as the density of $\mu$ with respect to $\lambda$. The total variation  measure of $\mu\in\mathcal{M}_b(U;X)$ with respect to the norm $|\,\,|$ on $X$, is denoted by $|\mu|$. Given a positive measure $\lambda\in \mathcal{M}_b(U)$ and an integrable function  $f\colon U\to X$,  we denote by $f\lambda$  the $X$-valued measure defined for every  $B\in\B(U)$ by 
 \begin{equation*}
     f\lambda(B):=\int_{B}f\,{\rm d}\lambda.
 \end{equation*}
 Given  a Borel measure $\mu$ on $U$ and $E\in\B(U)$, the Borel measure $\mu\mres E$ is defined by $\mu\mres E(B):=\mu(B\cap E)$ for every $B\in\B(U)$.

 \item [(i)] Given a function $u\in L^1_{\rm loc}(\Rd;\Rd)$, the \textit{jump set} $J_u$ of $u$ is the set of all points $x\in U$ with the following property: there exists a triple $(u^+(x),u^-(x),\nu_u(x))\in\R^d\times \R^d\times \Sd$, with $u^{+}(x)\neq u^-(x)$, such that 
\begin{gather}\notag 
\lim_{\rho\to 0^+}\frac{1}{\rho^d}\int_{B_{\rho}(x)\cap H^{\pm}(x)}|u(y)-u^{\pm}(x)|\,{\rm d}y=0,
     \end{gather}
  where  $H^{\pm}(x):=\{y\in\Rd:\pm(y-x)\cdot\nu_u(x)>0\}$. The triple $(u^+(x),u^-(x),\nu_u(x))$ is uniquely determined, up to changing the sign of $\nu(x)$ and swapping   $u^+(x)$ and $u^-(x)$. Having fixed $\nu_u$, we set $[u]:=u^+-u^-$ on  $J_u$. It can be shown that $J_u$ is a Borel set and  that there is a choice of $u^\pm$ and $\nu_u$ that makes them Borel measurable on $J_u$ (see \cite[Proposition 3.69]{Ambrosio}).
  Moreover, Del Nin has recently showed (see \cite{DelNin}) that $J_u$ is always countably $(\hd,d-1)$-rectifiable in the sense of \cite[3.2.14]{Federer}.
\item[(j)] Given $U\in\U(\Rd)$, a function $u\colon U\to \Rd$ is said to be of {\it bounded deformation} if  $u\in L^1(U;\Rd)$ and its {\it distributional symmetric gradient} $Eu:=\frac12(Du+Du^T)$ belongs to $\mathcal{M}_b(U;\Rdsym)$. The collection of all functions of bounded deformation on $U$ is denoted by ${\rm BD}(U)$, while ${\rm BD}_{\rm loc}(U)$ is the collection of all $\Ld$-measurable functions $u\colon U\to \Rd$ such that $u|_V\in {\rm BD}(V)$ for every $V\in \U_c(U)$.
 We refer the reader to \cite{Temam} for an introduction to this space and to \cite{AmbCoscDalM} for the fine properties of its functions.

 \smallskip
 
 \item[(k)] Given $U\in \U(\Rd)$ and  $u\in {\rm BD}(U)$, the measure $Eu$ can be decomposed as 
\begin{equation}\label{decomposition of gradients}
    Eu=E^au+E^cu+E^ju=\E u\Ld+E^cu+
    [u]\odot\nu_u\hd\mres J_u,
\end{equation}
where:
\
\smallskip

\noindent(i)  the {\it absolutely continuous part} $E^au$ is the absolutely continuous part of $Eu$ with respect to $\Ld$, whose density is denoted by $\E u\in L^1(U;\Rdsym)$,
    \smallskip
    
\noindent (ii)  the {\it Cantor part} $E^cu$ is defined as $E^cu:=E^su\mres(\Omega\setminus J_u)$, where $E^su$ is the singular part of $Eu$ with respect to $\Ld$; it is known that  $E^cu$ vanishes on all $B\in \B(U)$ which are $\sigma$-finite with respect to $\hd$,

\smallskip

\noindent
(iii) the {\it jump part} $E^ju$ is defined by $E^ju:=Eu\mres J_u=E^su\mres J_u$; it is known that $E^ju=[u]\odot \nu_u\hd\mres J_u$, where   $\odot$ denotes the symmetric tensor product defined by $(a\odot b)_{ij}=\frac12(a_ib_j+a_jb_i)$ for every $a,b\in\Rd$.

\smallskip

\item[(l)] A function $u\colon \Rd\to \Rd$ is said to be a {\it rigid motion} if there exist $A\in\Rdskew$ and $b\in\Rd$ such that $u(x)=Ax+b$ for every $x\in\Rd$. The collection of all rigid motions is denoted by $\mathcal{R}$.
\end{itemize}

\section{The integrands and the functionals}
In this section we introduce the collections of integrands and of functionals that are going to be our main objects of study.

We fix non-negative constants $c_1,...,c_5\in [0,+\infty)$,  with $0<c_1\leq c_3$,  and a constant $\alpha\in (0,1)$.
We also fix a continuous function $\sigma\colon[0,+\infty)\to [0,+\infty)$ with the property that $\sigma(0)=0$ and 
\begin{equation}\label{slope at infty for sigma}
\sigma^\infty:=\limsup_{t\to+\infty}\frac{\sigma(t )}{t}<+\infty.
\end{equation}

\begin{definition}\label{bulk and surface integrands}Let $\mathcal{F}$ be the collection of all integrands $f\colon \Rd\times \Rdsym\to [0,+\infty)$ satisfying the following conditions:
\begin{enumerate}
\item[(f1)] $f$ is Borel measurable on $\Rd\times \Rdsym$;
    \item[(f2)] for every $x\in \Rd$ and $A\in\Rdsym$ we have $$c_1|A|-c_2\leq f(x,A)\leq c_3|A|+c_4;$$ 
    \item[(f3)]for every $x\in\Rd$ and $A_1,A_2\in\R^{d\times d}_{\rm sym}$ we have 
    $$|f(x,A_1)-f(x,A_2)|\leq c_5|A_1-A_2|.$$ 
\end{enumerate}

Let $\mathcal{G}$ be the collection of all integrands $g\colon  \Rd\times\Rd\times \Sd\to [0,+\infty)$ satisfying the following conditions:
\begin{enumerate}
\item [(g1)] $g$ is Borel measurable on $\Rd\times \Rd\times \Sd$;
\item[(g2)]for every $x\in\Rd$, $\zeta\in\Rd$, and $\nu\in\Sd$  we have $$g(x,\zeta,\nu)=g(x,-\zeta,-\nu);$$
\item [(g3)] for every $x\in\Rd$, $\zeta\in\Rd$, and $\nu\in\Sd$ we have 
$$c_1|\zeta\odot\nu|\leq g(x,\zeta,\nu)\leq c_3|\zeta\odot\nu|;$$
\item [(g4)] for every $x\in\Rd$, $\zeta_1,\zeta_2\in\Rd$, and $\nu\in\Sd$ we have $$|g(x,\zeta_1,\nu)-g(x,\zeta_2,\nu)|\leq \sigma(|\zeta_1-\zeta_2|).$$ 
\end{enumerate}
\end{definition}

\medskip

To every function $f\colon\Rd\times \Rdsym\to [0,+\infty)$  we associate its {\it recession function} $f^\infty$ (with respect to the variable $A$), defined by
\begin{equation}\label{def recession}
f^\infty(x,A):=\limsup_{t\to+\infty}\frac{f(x,tA)}{t}\quad \text{ for every $x\in \Rd$ and $A\in\Rdsym$}.
\end{equation}

\begin{remark}
 Given $f\colon \Rd\times \Rdsym$ satisfying (f2) and (f3), one can check that for every $x\in\Rd$, $A,A_1,A_2\in\Rdsym$ the function $f^\infty$ satisfies
\begin{gather}\label{bound finfty}
c_1|A|\leq f^\infty(x,A)\leq c_3|A|,
\\   \notag |f^\infty(x,A_1)-f^\infty(x,A_2)|\leq c_5|A_1-A_2|.
\end{gather}
    \end{remark}

We now introduce a class of  integral functionals  associated with $\mathcal{F}$  and $\mathcal{G}$.
\begin{definition}\label{def:integral functionals}
   Given $f\in\mathcal{F}$ and 
   $g\in\mathcal{G}$, for every open set $U\subset \Rd$ and $u\in L^1_{\rm loc}(\Rd;\Rd)$ we set 
\begin{equation}\label{integral on open sets}
    F^{f,g}(u,U):=\int_Uf(x,\E u)\,{\rm d}x+\int_Uf^\infty\Big(x,\frac{{\rm d}E^cu}{{\rm d}|E^cu|}\Big)\,{\rm d}|E^cu|+\int_{J_u\cap U}g(x,[u],\nu_u)\,{\rm d}\hd
\end{equation}
if $u|_U\in {\rm BD}_{\rm loc}(U)$, while $ F^{f,g}(u,U)=+\infty$ otherwise. The definition is then extended to every Borel set $B\in\B(\Rd)$ by setting
\begin{equation}\label{3.4bis}
    F^{f,g}(u,B):=\inf\{F^{f,g}(u,U): U\subset\Rd \text{ open, with }U\supset B\}.
\end{equation}
 Note that, if   $U\subset \Rd$ is a bounded open set, $B
\in\B(U)$, $u|_U\in {\rm BD}_{\rm loc}(U)$, and $|Eu|(U)<+\infty$, then
\begin{equation}\label{def functional}
    F^{f,g}(u,B)=\int_Bf(x,\E u)\,{\rm d}x+\int_Bf^\infty\Big(x,\frac{{\rm d}E^cu}{{\rm d}|E^cu|}\Big)\,{\rm d}|E^cu|+\int_{J_u\cap B}g(x,[u],\nu_u)\,{\rm d}\hd.
\end{equation}
Indeed, in this case the upper bounds in (f2), \eqref{bound finfty}, and (g3) imply
\begin{gather*}
   \int_Uf(x,\E u)\,{\rm d}x+
   \int_Uf^\infty\Big(x,\frac{{\rm d}E^cu}{{\rm d}|E^cu|}\Big)\,{\rm d}|E^cu|+
   \int_{J_u\cap U}g(x,[u],\nu_u)\,{\rm d}\hd<+\infty,
\end{gather*}
so that \eqref{def functional} follows from  \eqref{integral on open sets} and \eqref{3.4bis}.
\end{definition}

We now introduce a collection of abstract functionals  related to the integral functionals considered above. 
\begin{definition}\label{abstract functionals}
     Let $\mathfrak{F}$ be the collection of functionals $F\colon L^1_{\rm loc}(\Rd;\Rd)\times \B(\Rd)\to [0,+\infty]$ with the following properties:
\begin{enumerate}
    \item [{\rm(a)}] {\it measure property}:   for every $u\in L^1_{\rm loc}(\Rd;\Rd)$ the set function $F(u,\cdot)$ is a Borel measure that for every $B\in\B(\Rd)$  satisfies 
    \begin{equation}\label{outer regularity}
        F(u,B):=\inf\{F(u,U):U\subset  \Rd \text{ open}, \,U\supset B\};
    \end{equation}
    \item [\rm {(b)}] {\it locality on open sets:} for every bounded open set $U\subset \Rd$ and  $u,v\in L^1_{\rm loc}(\Rd;\Rd)$ with $u=v$ $\Ld$-a.e\ on $U$ we have $F(u,U)=F(v,U)$;
    \item[{\rm (c)}]{\it upper and lower bounds:} if $u\in L^1_{\rm loc}(\Rd;\Rd)$ and   $U\subset \Rd$ is a bounded open set, then 
    \begin{equation*}
       c_1|Eu|(U)-c_2\Ld(U)\leq F(u,U)\leq c_3|Eu|(U)+c_4\Ld(U) \quad \text{if  $u|_U\in {\rm BD}(U)$},
    \end{equation*}
    while $F(u,U)=+\infty$ if $u|_U\notin {\rm BD}(U);$
    \item[{\rm (d)}]{\it invariance under rigid motions:} for every $u\in L^1_{\rm loc}(\Rd;\Rd)$, $w\in\mathcal{R}$, and $B\in \B(\Rd)$  we have
 $$F(u+w,B)=F(u,B).$$
 
    \item[{\rm (e)}]{\it bulk continuity estimate:} for every $u\in L^1_{\rm loc}(\Rd;\Rd)$, $A\in\Rdsym$, and $B\in \B(\Rd)$ we have 
    \begin{equation*}
        F(u+\ell_A,B)\leq F(u,B)+c_5|A|\Ld(B);
    \end{equation*}
    \item[{\rm (f)}]{\it surface continuity estimate:} for every  $u\in L^1_{\rm loc}(\Rd;\Rd)$,  $x\in\Rd$, $\zeta\in\Rd$, $\nu\in\Sd$, and $B\in \B(\Rd)$ we have 
    \begin{equation*}
        F(u+u_{x,\zeta,\nu},B)\leq F(u,B)+\sigma(|\zeta|)\hd(B\cap \Pi^\nu_x),
    \end{equation*}
    where $\Pi^\nu_x:=\{y\in\Rd:(y-x)\cdot\nu=0\}$ is the hyperplane orthogonal to $\nu$ containing $x$.
\end{enumerate}
We denote by $\mathfrak{F}_{\rm sc}$ the collection of functionals $F\in\mathfrak{F}$ such that 
for every open $U\subset \Rd$ the functional $F(\cdot,U)$ is $L^1_{\rm loc}(\Rd;\Rd)$-lower semicontinuous.
\end{definition}

\begin{remark}\label{strong locality}
    Combining properties (a) and (b) we obtain that the functionals in $\mathfrak{F}$ satisfy the following locality property: if $B\in\B(\Rd)$, $u,v\in L^1_{\rm loc}(\Rd;\Rd)$, and $u=v$ $\Ld$-a.e.\ on a neighbourhood of $B$, then $F(u,B)=F(v,B)$. Easy examples show that, in general, this equality does not hold if we have only $u=v$ $\Ld$-a.e.\ on $B$. 

    For this reason, if $U\subset \Rd$ is an open set and $u\in L^1_{\rm loc}(U;\Rd)$, for every $B\in\B(U)$ we can define  $F(u,B):=F(v,B)$, where $v\in L^1_{\rm loc}(\Rd;\Rd)$ is any function such that $v=u$ $\Ld$-a.e.\ on $U$. The locality property described above implies that the value of $F(u,B)$  does not depend on the chosen extension $v$ of $u$.
\end{remark}

\begin{remark}
From \eqref{outer regularity} and (c) we deduce the following inequalities for Borel sets. If $U\subset \Rd$ is a bounded open set and $u\in {\rm BD}(U)$, then for every $B\in\B(U)$ we have
\begin{equation*}
       c_1|Eu|(B)-c_2\Ld(B)\leq F(u,B)\leq c_3|Eu|(B)+c_4\Ld(B). 
    \end{equation*}
\end{remark}
\begin{remark}
   Thanks to property (a) a functional $F\in\mathfrak{F}$ belongs to $\mathfrak{F}_{\rm sc}$ if and only if for every open bounded set $U\subset \Rd$ the functional $F(\cdot,U)$ is $L^1_{\rm loc}(\Rd;\Rd)$-lower semicontinuous.
\end{remark}

We now show that the functional introduced in Definition \ref{def:integral functionals} belongs to $\mathfrak{F}$.

\begin{proposition}\label{prop integrals are abstract}
    Let $f\in\mathcal{F}$ and $g\in\mathcal{G}$. Then the functional $F^{f,g}$ introduced in Definition \ref{def:integral functionals} belongs to $\mathfrak{F}$.
\end{proposition}
\begin{proof}
   By \eqref{3.4bis} $F^{f,g}$ satisfies \eqref{outer regularity}, while the rest of property (a) is obvious when $u\in {\rm BD}(\Rd)$. The general case $u\in L^1_{\rm loc}(\Rd;\Rd)$ can be obtained by a straightforward argument; it can also be treated using the De Giorgi-Letta criterion for measures (see Lemma \ref{measure} below).  Property (b) is obvious. Property (c) follows from (f2), \eqref{bound finfty}, and (g3). To see that property (d) is satisfied by $F^{f,g}$ is enough to observe that for every $w\in\mathcal{R}$ we have $Ew=0$. 

   To see that $F^{f,g}$ satisfies property (e) we observe that by (f3) for every $B\in\B(\Rd)$ and $A\in\Rdsym$ we have
    \begin{equation*}
        |F^{f,g}(u+\ell_A,B)-F^{f,g}(u,B)|\leq c_5|A|\Ld(B).
    \end{equation*}
    A similar argument shows that $F^{f,g}$ satisfies property (f).
\end{proof}

 \section{A compactness result for  \texorpdfstring{$\mathfrak{F}$}{F}}\label{sec:compactness}
  In this section we  state and prove a compactness result for sequences of functionals in $\mathfrak{F}$. The proof is based on the classical localisation method for $\Gamma$-convergence (see \cite[Section 3.3]{BraHandbook} or \cite[Chapter 18]{DalBook}).

\begin{theorem}\label{compactness} Let $\{F_n\}_{n\in\N}\subset \mathfrak{F}$.  Then there exist a functional $F\in\mathfrak{F}_{\rm sc}$ and a subsequence, not relabelled, such that for every bounded open set $U\subset \Rd$ the sequence $\{F_n(\cdot,U)\}_{n\in\N}$ $\Gamma$-converges to $F(\cdot,U)$ in the topology of $L^1_{\rm loc
}(\Rd;\Rd)$.
\end{theorem}
\begin{proof}
For every open set $U\subset\Rd$  we consider the functionals defined on $L^1_{\rm loc}(\Rd;\Rd)$ by
\begin{gather}\label{def Gammaliming et limsup}
    F^\prime (\cdot,U)=\Gamma\text{-}\liminf_{n\to+\infty}F_n(\cdot,U)\quad \text{ and }\quad  F^{\prime\prime} (\cdot,U)=\Gamma\text{-}\limsup_{n\to+\infty}F_n(\cdot,U),\\\label{def inner}
    F^\prime_- (\cdot,U)=\sup_{U'\in\mathcal{U}_c(U)}F^\prime(\cdot,U')\quad \text{ and }\quad  F^{\prime\prime}_- (\cdot,U)=\sup_{U'\in\mathcal{U}_c(U)}F^{\prime\prime}(\cdot,U'),
\end{gather}
where the $\Gamma$-liminf and $\Gamma$-limsup are computed with respect to the $L^1_{\rm loc}(\Rd;\Rd)$-convergence and $\U_c(U)$ is the collection of open sets defined in (g) of Section \ref{sec: Notation}. 

It is immediate to check that for every $u\in L^1_{\rm loc}(\Rd;\Rd)$ the  set functions $F_n(u,\cdot)$, $F^\prime(u,\cdot)$, $F''(u,\cdot)$, $ F^\prime_- (u,\cdot)$, and  $ F''_- (u,\cdot)$ are  increasing. Thus, we can apply the compactness theorem for sequences of increasing set functionals \cite[Theorem 16.9]{DalBook} to obtain that there exist a subsequence of $\{F_n\}_{n\in\N}$, not relabelled, and a functional $F\colon L^1_{\rm loc}(\Rd;\Rd)\times \U(\Rd)$ such that   
\begin{equation}\label{def F}
    F(u,U)=F'_-(u,U)=F''_-(u,U)\quad \text{for every }u\in L^1_{\rm loc}(\Rd;\Rd)  \text{ and }U\in\mathcal{U}(\Rd).
\end{equation}

For every $u\in L^1_{\rm loc}(\Rd;\Rd)$ and $B\in\B(\Rd)$ we set 
\begin{equation}\label{Borel extension}
    F(u,B):=\inf\{F(u,U):U\in\U(\Rd)  \text{ with }U\supset B\}.
\end{equation}
We want to prove that $F\in\mathfrak{F}$. The approximation property \eqref{outer regularity} is obvious, while the measure property (a) will be proved in Lemma \ref{measure} below.  The locality property (b) and the lower semicontinuity of $F$ follow from general results about $\Gamma$-limits of local functionals (see \cite[Proposition 16.15 and Remark 16.3]{DalBook}).

We note that it is enough to prove properties (d)-(f) for $U\in \U_c(\Rd)$. Indeed, by the measure property (a) they can be extended to $U\in\U(\Rd)$ using the inner regularity of Borel measures. The further extension to Borel sets follows immediately from \eqref{Borel extension}. 

Let us fix  $u\in L^1_{\rm loc}(\Rd;\Rd)$ and $U\in \U_c(\Rd)$.
We begin by proving the upper and lower bounds in (c).  We consider first the case $u\in {\rm BD}(U)$. Under this hypothesis, the inequality 
\begin{equation}\label{upper bound F''}
        F''(u,U)\leq c_3|Eu|(U)+c_4\Ld(U)
    \end{equation} is a consequence of the corresponding inequality for $F_n$ and of the fact that the $\Gamma$-limsup is less than or equal to the pointwise limsup. The upper bound in (c) for $F$ follows immediately from \eqref{def inner},  \eqref{def F}, and \eqref{upper bound F''}. To prove the lower bound for $F''$, we observe that our assumption $u\in {\rm BD}(U)$ implies that $F''
(u,U)<+\infty$.  Given $U'\in\U_c(U)$, we fix a sequence $\{u_n\}_{n\in\N}\subset L^1_{\rm loc}(\Rd;\Rd)$ 
 converging to $u$ in $L^1_{\rm loc}(\Rd;\Rd)$ such that $\limsup_n F_n(u_n,U')= F''(u,U')<+\infty$.  Since for every $n\in\N$ the functional $F_n$ satisfies  (c) we have $u_n\in {\rm BD}(U')$ for $n\in\N$ large enough and 
\begin{gather}\notag 
c_1|E{u_n}|(U')-c_2\Ld(U')\leq F_n(u_n,U').
\end{gather}
From the lower semicontinuity of $v\mapsto|Ev|(U')$ with respect to the $L^1$-convergence and from the previous inequality we obtain 
\begin{equation*}
    c_1|E u|(U')-c_2\Ld(U')\leq F''(u,U').
\end{equation*} Using again  \eqref{def inner} and \eqref{def F}, we deduce the corresponding lower bound for $F(u,U)$. 

Assume now that $u\notin  {\rm BD}(U)$. We want to prove that $F(u,U)=+\infty$. Assume by contradiction that  $F(u,U)<+\infty$. Given   $U'\in \U_c(U)$, we consider a sequence $\{u_n\}_{n\in\N}\subset L^1_{\rm loc}(\Rd;\Rd)$ converging to $u$ in $L^1_{\rm loc}(\Rd;\Rd)$ such that $\limsup_nF_n(u_n,U')=F''(u,U')\leq F(u,U)<+\infty$. Arguing as before, we obtain that  $u_n\in {\rm BD}(U')$ for $n\in\N$ large enough and  $$\limsup_{n\to+\infty}|Eu_n|(U')\leq \frac{1}{c_1}\limsup_{n\to+\infty}F_n(u_n,U')+\frac{c_2}{c_1}\Ld(U')\leq \frac{F(u,U)}{c_1}+\frac{c_2}{c_1}\Ld(U),
$$ which implies $u\in {\rm BD}(U')$ and $|Eu|(U')\leq \frac{1}{c_1}{F(u,U)}+\frac{c_2}{c_1}\Ld(U)$. As this holds for every $U'\in\U_c(U)$, we obtain $u\in {\rm BD}(U)$, contradicting our hypothesis. This concludes the proof of the implication $u\notin {\rm BD}(U)\implies F(u,U)=+\infty$.

Properties (d)-(f) for $F$ follow immediately from the same properties for $F_n$. 

In Lemma \ref{lemma:equality comp} we will prove that for every $u\in L^1_{\rm loc}(\Rd;\Rd)$   and $U\in\mathcal{U}_c(\Rd)$ we have
\begin{equation}\label{equality comp}
      F(u,U)=F'(u,U)=F''(u,U), 
\end{equation}
hence for every $U\in\U_c(\Rd)$ the sequence  $\{F_n(\cdot,U)\}_{n\in\N}$ $\Gamma$-converges to $F(\cdot,U)$ in the topology of  $L^1_{\rm loc}(\Rd;\Rd)$. 
\end{proof}

In the rest of this section we assume that $F'$ and $F''$ are defined by \eqref{def Gammaliming et limsup} and that $F$ satisfies \eqref{def F}. Our aim is to  prove equality \eqref{equality comp} and the measure property (a) of Definition \ref{abstract functionals} for this functional. The proof of the measure property (a) is  based on the De Giorgi-Letta criterion for measures, which is based on the subadditivity of the function $U\mapsto F(u,U)$. For this reason,
we begin by proving that $F^{\prime \prime}$ satisfies the {\it nested} subadditivity condition. 

\begin{lemma}\label{fundamental estimate} Let $u\in L^1_{\rm loc}(\Rd;\Rd)$ and  $U',U,V\in \mathcal{U}_c(\Rd)$, with  $U'\subset \subset U$. Then  
\begin{gather}\label{nested subadditivity}
    F''(u,U'\cup V)\leq F''(u,U)+F''(u,V).
\end{gather}
\end{lemma}
\begin{proof}
Without loss of generality, we may assume that $F''(u,U)+F''(u,V)<+\infty$. Since $F''\geq F$ and $F$ satisfies property (c),  this implies that $u\in {\rm BD}(U\cup  V)$. 
 Recalling the definition of $F''$ (see \eqref{def Gammaliming et limsup}), we can find  two sequences $\{u_n\}_{n\in\N}$ and $\{v_n\}_{n\in\N}$ in $L^1_{\rm loc}(\Rd;\Rd)$, converging to $u$ in  $L^1_{\rm loc}(\Rd;\Rd)$ and  such that 
\begin{equation}\label{recoveries subadditivity}
F''(u,U)=\limsup_{n\to+\infty}F_n(u_n,U)\quad \text{ and }\quad F''(u,V)=\limsup_{n\to+\infty}F_n(v_n,V).
\end{equation}
Thanks to property (c) for $F_n$ and the finiteness of $F''(u,U)$ and $F''(u,V)$ we can assume that  $\{u_n\}_{n\in\N}\subset {\rm BD}(U)$, $\{v_n\}_{n\in\N}\subset {\rm BD}(V)$, and that 
\begin{gather}
M:=c_3\sup_{n\in\N} \Big(|E{u_n}|(U)+|E{v_n}|(V)\Big)+c_4\Ld(U\cup V)<+\infty
\label{bound M}.
\end{gather}

Let us fix $m\in\N$ and set $\eta:=\text{dist}(U',\partial U)>0$. For $j\in\{0,...,m\}$ we also consider the sets $U_j:=\{x\in U: \text{dist}(x,\partial U)>\eta-\frac{j\eta}{m}\}$. Clearly,  we have  $U'\subset U_0\subset \subset\dots\subset \subset U_{m-1}\subset U_m= U$.  

For every $j\in\{1,...,m\}$ we consider a function  $\varphi_j\in C^\infty_c(\Rd;[0,1])$  satisfying $\varphi_j=1$ on  a neighbourhood of $\overline{U}_{j-1}$, $\varphi_j=0$ on  a neighbourhood of $\Rd\setminus U_{j}$, and  \begin{equation}\label{4.11}|\nabla \varphi_j|\leq \frac{2m}{\eta}\quad  \text{on $\Rd$}.
\end{equation}We set 
\begin{equation}\label{convex combination}
    w^j_n:=\varphi_j {u}_n+(1-\varphi_j){v}_n\quad \text{and}\quad S_j:=(U_{j}\setminus \overline{U}_{j-1})\cap V, 
\end{equation}
 and observe  that $w_n^j\in {\rm BD}(U'\cup V)$ and 
\begin{equation}\label{liebnitz rule}
    E{w^j_n}=\varphi_j E{{u}_n}+(1-\varphi_j)E{{v}_n}+({u}_n-v_n)\odot \nabla \varphi_j\, \Ld
\end{equation}
as Radon measures on $U'\cup V$. Moreover, one checks immediately that for every choice of  $j_n\in\{1,...,m\}$ the sequence  $\{w^{j_n}_n\}_{n\in\N}$ converges to $u$  in $L^1_{\rm loc}(\Rd;\Rd)$ as $n\to+\infty$.

 Let us fix $j\in\{1,...,m\}$ and $n\in\N$. We observe that by the  measure property (a), by the properties of $\varphi_j$, and by the locality property of Remark  \ref{strong locality} we have
\begin{gather}\notag 
    F_n(w^j_n,U'\cup V)\leq     F_n(w^j_n,\overline{U}_{j-1})+F_n(w^j_n,S_j)+F_n(w^j_n,V\setminus {U}_{j})\\ \label{decomposition in shells}
   =F_n(u_n,\overline{U}_{j-1})+F_n(w^j_n,S_j)+F_n(v_n,V\setminus {U}_{j})
    \leq F_n(u_n,U)+F_n(w^j_n,S_j)+F_n(v_n,V).
\end{gather}

We  now estimate the term involving the strip $S_j$ in \eqref{decomposition in shells}. 
\begin{gather}\notag 
F_n(w_n^j,S_j)
\leq c_3|Ew^j_n|(S_j)+c_4\Ld(S_j),
\end{gather}
and by \eqref{4.11} and \eqref{liebnitz rule} we get 
\begin{align}\notag
    |E{w^j_n}|(S_j)\leq |E{u_n}|(S_j)+|E{v_n}|(S_j)+\frac{2m}{\eta}\int_{S_j}|u_n-v_n|\,{\rm d}x. 
\end{align}
 From these two inequalities, we deduce that  for every $n\in\N$ we can find $j_n\in\{1,..,m\}$  such that, setting $w_n:=w^{j_n}_n$, we have \bigskip {
 \begin{gather}
     \notag 
    F_n(w_n,S_{j_n})\leq\frac{1}{m} \sum_{j=1}^m \Big(c_3|E{u_n}|(S_j)+c_3|E{v_n}|(S_j) +\frac{2c_3m}{\eta}\int_{S_j}|u_n-v_n|\,{\rm d}x+c_4\Ld(S_j)\Big)\\ 
    \leq \notag 
    \frac{1}{m}\big(c_3|Eu_n|(U)+c_3|Ev_n|(V)+c_4\Ld(U\cup V)\big)+ \frac{2c_3}{\eta}\int_{U \cap V}|u_n-v_n|\,\dx x
    \end{gather}
    \begin{gather}
    \label{estimate on good stripbis}
     \leq \frac{M}{m}+\frac{2c_3}{\eta}\int_{U\cap V}|u_n-v_n|\,\dx x,
  \end{gather}
}
   where in the second inequality we used \eqref{bound M}.

Finally, combining \eqref{decomposition in shells} and \eqref{estimate on good stripbis} we deduce that
\begin{align}
    F_n(w_n,U'\cup V)
    &\leq  F_n(u_n,U)+F_n(v_n,V)\label{final subadditive estimate}
 +\frac{M}{m}+\frac{2c_3}{\eta}\int_{U \cap V}|u_n-v_n|\,\dx x
\end{align}
for every $n\in\N$.
Recalling that $\{u_n-v_n\}_{n\in\N}$  converges to $0$ strongly in $L^1(U\cap V;\Rd)$, that $\{w_n\}_{n\in\N}$ converges to $u$ in $L^1_{\rm loc}(\Rd;\Rd)$,  we may let $n\to+\infty$ in \eqref{final subadditive estimate} and from \eqref{recoveries subadditivity} we  deduce that
\begin{equation*}
    F''(u,U'\cup V)\leq\limsup_{n\to+\infty} F_n(w_n,U'\cup V)\leq  F''(u,U)+F''(u,V)+\frac{M}{m}.
\end{equation*}
 Letting $m\to+\infty$ in this inequality, we obtain \eqref{nested subadditivity}.
\end{proof}

We are now ready to prove the subadditivity of $F(u,\cdot)$ on $\U(\Rd)$.
\begin{lemma}\label{lemma: subadditivity}
    Let  $u\in L^1_{\rm loc}(\Rd;\Rd)$ and $U,V\in \U(\Rd)$. Then
    \begin{equation*}
        F(u,U\cup V)\leq F(u,U)+F(u,V).
    \end{equation*}
    \begin{proof}
        The result is a consequence of Lemma \ref{fundamental estimate} and of classical arguments concerning increasing set functions (see, for instance, \cite[Lemma 18.4]{DalBook}).
    \end{proof}    
\end{lemma}
We now prove that $F$ satisfies the measure property (a).
\begin{lemma}\label{measure}
    Let $u\in L^1_{\rm loc}(\Rd;\Rd)$. Then $F(u,\cdot)$ is a Borel measure.
\end{lemma}
\begin{proof}
    It is enough to use the De Giorgi-Letta criterion for measures (see \cite[Théorème 5.1]{DeGiorgiLetta} or\cite[Theorem 14.21]{DalBook}), which ensures that the claim is proved, provided that we show that the set function $U\mapsto F(u,U)$ is subadditive, superadditive, and inner regular on $\U(\Rd)$. Subadditivity follows from Lemma \ref{lemma: subadditivity}, superadditivity is a consequence of  classic results of $\Gamma$-convergence (see \cite[Proposition 16.12]{DalBook}), and  inner regularity follows from  \eqref{def F}, recalling  \eqref{def inner}.
\end{proof}

Finally, we conclude the section by showing that $\Gamma$-convergence takes place for every $U\in\U_c(\Rd)$.

\begin{lemma}\label{lemma:equality comp}
Let $u\in L^1_{\rm loc}(\Rd;\Rd)$ and $U\in \U_c(\Rd)$. Then
\begin{equation}\notag
      F(u,U)=F'(u,U)=F''(u,U).
\end{equation}
\end{lemma}
\begin{proof}
    Since the inequalities $F(u,U)\leq F'(u,U)\leq F''(u,U)$ are obvious, we only  have to show that $F''(u,U)\leq F(u,U)$, assuming that $F(u,U)<+\infty$. By property (c) for the functional $F$,  this implies that 
    $u\in {\rm BD}(U)$. 
    
    Let $\e>0$ and consider a compact set $K\subset U$ such that \begin{equation*}c_3|Eu|(U\setminus K)+c_4\Ld(U\setminus K)< \e.\end{equation*}
    Consider now $U'\in \U_c(\Rd)$ with $K\subset U'\subset \subset U$. Using Lemma \ref{fundamental estimate} with $V=U\setminus K$, from the previous inequality and \eqref{upper bound F''} we obtain 
    \begin{equation}\notag
         F''(u,U)\leq F''(u,U')+F''(u,U\setminus K)< F(u,U)+\e.
    \end{equation}
    The conclusion then follows from the arbitrariness of $\e$.
\end{proof}

\section{A partial integral  
representation for abstract functionals}

In this section we use the results proved in  \cite{CaroccFocardiVan} to provide a partial integral representation of every functional  $F\in\mathfrak{F}_{\rm sc}$.  More precisely, we  show that the ``absolutely continuous''  part and the ``jump'' part of every such $F$ can be written as integrals associated to suitable integrands $f\in\mathcal{F}$ and $g\in\mathcal{G}$. 

\medspace 

We begin by introducing a decomposition of functionals in $\mathfrak{F}$ that reflects the usual decomposition of the measure $Eu$ described in \eqref{decomposition of gradients}.

\begin{definition}\label{decomposition}
    Let $F\in\mathfrak{F}$, $U\in \U_c(\Rd)$, and $u\in {\rm BD}(U)$. By (a) and (c) of Definition \ref{abstract functionals} $F(u,\cdot)$ is a bounded Radon measure on $U$.  Its absolutely continuous and the singular part with respect to $\Ld$,  defined on every $B\in\B(U)$,  are denoted by $F^a(u,B)$ and $F^s(u,B)$, respectively.  We also introduce the bounded Radon measures $F^c(u,\cdot)$ and $F^j(u,\cdot)$ defined by 
    \begin{gather*}
F^c(u,B):=F^s(u,B\setminus J_u) \text{ and }F^j(u,B):=F(u, B\cap J_u)
    \end{gather*}
    for every $B\in\B(U)$.
Note that we have
\begin{equation}\label{somma facj}
F(u,B)=F^a(u,B)+F^c(u,B)+F^j(u,B)
\end{equation}
for every $B\in \B(U)$. 
\end{definition}

\begin{remark}\label{remark abs continuity}
    Let $F\in\mathfrak{F}$,  $U\in\U_c(\Rd)$, and $u\in {\rm BD}(U)$. It follows directly from the bounds in (c) of Definition \ref{abstract functionals} that $F^c(u,\cdot)$ and  $F^j(u,\cdot)$, considered as measures defined on $\B(U)$, are absolutely continuous with respect to $|E^cu|$ and $\hd\mres J_u$, respectively.
\end{remark}
In the rest of the paper, we will often work with functions defined by using auxiliary minimisation problems. For this reason, given a functional $F\in\mathfrak{F}$, a set $U\in\U_c(\Rd)$ with Lipschitz boundary $\partial U$, and a function $v\in {\rm BD}(U)$, we set 
\begin{equation}\notag 
    \m^F(v,U):=\inf\{F(u,U): u\in {\rm BD}(U) \text{ and } u=v \text{ on }\partial U \},
\end{equation}
where the equality $u=v$ is understood in the sense of traces. As $v$ is a competitor for the minimisation problem $\m^F(v,U)$, by the upper bounds in Definition \ref{abstract functionals} it follows that
\begin{equation}\notag 
      \m^F(v,U)\leq F(v,U)\leq c_3|Ev|(U)+c_4\Ld(U)<+\infty.
\end{equation}

Given  $F\in\mathfrak{F}$,  for every $x\in\Rd,$ $\zeta\in\Rd$,  $A\in\Rdd_{\rm sym}$, and $\nu\in\Sd$ we set
\begin{gather}\label{def f}
    f(x,A):=\limsup_{\rho \to 0^+}\frac{\m^F(\ell_A,Q(x,\rho))}{\rho^d},\\\label{def g}
    g(x,\zeta,\nu):=\limsup_{\rho \to 0^+}\frac{\m^F(u_{x,\zeta,\nu},Q_\nu(x,\rho))}{\rho^{d-1}},
\end{gather}
where $Q(x,\rho)$ and $Q_\nu(x,\rho)$ are the cubes defined in (c) and (e) of Section \ref{sec: Notation}, while the functions $\ell_A$ and $u_{x,\zeta,\nu}$ are those introduced in (f) of the same section.

We now state the partial integral representation result, whose proof is given in \cite[Lemma 5.1]{CaroccFocardiVan}, following the lines of \cite{bouchitte1998global} and of \cite{EbobisseToader}.
\begin{lemma}\label{representation proposition}
Let $F\in\mathfrak{F}_{\rm sc}$, $U\in\U_c(\Rd)$, and  $u\in {\rm BD}(U)$. 
Let $f$ and $g$ be the two functions defined by \eqref{def f} and \eqref{def g}. Then $f\in\mathcal{F}$, $g\in\mathcal{G}$, and 
\begin{gather}\label{integral representation bulk}
    F^a(u,B)=\int_Bf(x,\E u)\, {\rm d}x,\\\label{integral representation jump}
    F^j(u,B)=\int_{J_u\cap B}g(x,[u],\nu_u)\,{\rm d}\hd,
\end{gather}
for every set $B\in\B(U)$.
\end{lemma}

\begin{proof}
   The proof of equalities \eqref{integral representation bulk} and \eqref{integral representation jump} can be found in \cite[Lemma 5.1]{CaroccFocardiVan}. The inclusions $f\in\mathcal{F}$ and $g\in\mathcal{G}$ are proved in Lemmas \ref{inclusion bulk} and \ref{surface inclusion} below.  
\end{proof}

\begin{remark}
    We remark that properties (e) and (f) of Definition \ref{abstract functionals} are not used to obtain \eqref{integral representation bulk} and \eqref{integral representation jump}. Properties (e) and (f) are used only to prove that $f$ and $g$ satisfy (f3) and (g4).
\end{remark}

To prove that $f$ is Borel measurable we will need the following lemma, which will be used also in Section \ref{section: homo} for different purposes.
\begin{lemma}\label{lemma poco use}
    Let $F\in\mathfrak{F}$, $U',U\in\U_c(\Rd)$,  with Lipschitz boundary and  $U'\subset \subset U$, and let $w\in {\rm BD}(U)$. Then
    \begin{equation}\label{claim lemma useless}
        \m^F(w,U)\leq \m^F(w,U')+c_3|Ew|(U\setminus U')+c_4\Ld(U\setminus U').
    \end{equation}
\end{lemma}
\begin{proof}
    Let $\eta>0$ and $u\in {\rm BD}(U')$ with $u=w$ on $\partial U'$ be such that 
    \begin{equation}\label{ineq eta}
         F(u,U')\leq \m^F(w,U')+\eta.
    \end{equation}
    We can then extend $u$ to $U$ by setting $u=w$ on $U\setminus U'$, so that $u=w$ on $\partial U$, which implies 
    \begin{equation}\notag 
        \m^F(w,U)\leq F(u,U). 
    \end{equation}
    Using properties (a) and (c) of Definition \ref{abstract functionals} we get 
    \begin{equation}\label{intermediate something}
     F(u,U)=F(u,U')+F(u,U\setminus U')\leq F(u,U')+c_3|Eu|(U\setminus \overline{U}{'})+c_3|Eu|(\partial U')+c_4\Ld(U\setminus U').
    \end{equation}
    Since $u=w$ on the open set $U\setminus \overline{U}{'}$, we have $Eu=Ew$ as measures defined on $\B(U\setminus \overline{U}{'})$, hence $|Eu|(U\setminus \overline{U}{'})=|Ew|(U\setminus \overline{U}{'})$. Moreover, since $|Eu|\mres \partial U'=|(u^+-u^-)\odot\nu_{\partial U'}|\hd\mres (J_u\cap \partial U')$ and    $|Ew|\mres \partial U'=|(w^+-w^-)\odot\nu_{\partial U'}|\hd\mres  (J_w\cap \partial U')$, the equalities $u^+=w^+$ and $u^-=w^-$ on $\partial U'$ imply  that $|Eu|(\partial U')=|Ew|(\partial U')$.
    Combining this equality with \eqref{ineq eta}-\eqref{intermediate something} we get 
    \begin{equation*}
        \m^F(w,U)\leq \m^F(w,U')+c_3|Ew|(U\setminus \overline{U}{'})+c_3|Ew|(\partial U')+c_4\Ld(U\setminus U')+\eta,
    \end{equation*}
    which, by arbitrariness of $\eta$, implies \eqref{claim lemma useless}.
\end{proof}

We are ready to prove that $f\in\mathcal{F}$. 
\begin{lemma}\label{inclusion bulk}
     Let $F\in\mathfrak{F}$. Then the function $f$ defined by \eqref{def f} belongs to $\mathcal{F}$.
\end{lemma}
\begin{proof}
The proof of property (f1) can be obtained following the lines of the proof of \cite[Lemma 5.3]{DalConvex}, replacing the Lemma 4.11 of that paper by our  Lemma \ref{lemma poco use}.  

Let us fix $x\in\Rd$ and $A\in\Rdsym$.  We want to show that 
    \begin{equation}\label{upper bound f2}
        f(x,A)\leq c_3|A|+ c_4.
    \end{equation}
  To this aim, we note that for every $\rho>0$  from the upper bound in (c) of Definition \ref{abstract functionals} and the fact that $\ell_A$ is a competitor for the minimisation problem $\m^F(\ell_A,Q(x,\rho))$ it follows that 
 \begin{equation*}
     \m^F(\ell_A,Q(x,\rho))\leq F(\ell_A,Q(x,\rho))\leq (c_3|A|+c_4)\rho^d.
 \end{equation*}
 We now divide by $\rho^d$, take the limsup as $\rho\to 0^+$, and use \eqref{def f} to obtain \eqref{upper bound f2}.

 To conclude the proof of (f2) we  show that 
 \begin{equation}\label{lowerbound}
    c_1|A|-c_2\leq f(x,A).
 \end{equation}
 To prove this, let us fix $\rho>0$ and observe that the lower bound in (c) of Definition \ref{abstract functionals} we have 
\begin{equation}\notag 
   c_1|Eu|(Q(x,\rho))-c_2\rho^d\leq F(u,Q(x,\rho))
\end{equation}
 for every $u\in {\rm BD}(Q(x,\rho))$, so that 
 \begin{equation}\notag 
     c_1\min\{|Eu|(Q(x,\rho)):u\in {\rm BD}(Q(x,\rho)) \text{ with }u=\ell_A \text{ on }\partial Q(x,\rho)\}-c_2\rho^d\leq \m^F(\ell_A,Q(x,\rho)).
 \end{equation}
 By Jensen's inequality the function $\ell_A$ minimises the problem on the left-hand side of the previous inequality. Hence, 
 \begin{equation*}
       c_1|A|\rho^d-c_2\rho^d\leq \m^F(\ell_A,Q(x,\rho)).
 \end{equation*}
 Dividing by $\rho^d$, taking the limsup 
 as $\rho\to 0^+$, and recalling \eqref{def f}, we obtain \eqref{lowerbound}.

We now prove (f3). To this aim, let us fix $x\in\Rd$, $A_1,A_2\in\Rdsym$, $\rho>0$, and $\eta>0$.
Consider a function $u_2\in {\rm BD}(Q(x,\rho))$ with $u_2=\ell_{A_2}$ on $\partial Q(x,\rho)$ and such that 
\begin{equation*}
   F(u_2,Q(x,\rho))\leq \m^{F}(\ell_{A_2},Q(x,\rho))+\eta\rho^{d}.
\end{equation*}
Then, we set $u_1:=u_2-\ell_{A_2}+\ell_{A_1}$ and observe that $u_1=\ell_{A_1}$ on $\partial Q(x,\rho)$. Using property (e) of Definition \ref{abstract functionals} and the previous inequality, it follows that 
\begin{equation}\label{estimate for lipschiztianity}
\begin{gathered}
    \m^F(\ell_{A_1},Q(x,\rho))\leq F(u_1,Q(x,\rho))\leq F(u_2,Q(x,\rho))+c_5|A_1-A_2|\rho^d\\\leq \m^F(\ell_{A_2},Q(x,\rho))+c_5|A_1-A_2|\rho^d+\eta\rho^{d}.
\end{gathered}
\end{equation}
Dividing by $\rho^d$, taking first the limsup as $\rho\to 0^+$ and then as $\eta\to 0^+$, by \eqref{def f} we obtain 
\begin{equation*}
    f(x,A_1)\leq  f(x,A_2)+c_5|A_1-A_2|.
\end{equation*}
Exchanging the roles of $A_1$ and $A_2$, we obtain (f3).
\end{proof}

To prove that $g$ is Borel measurable we will need the following lemma, which will be used also in Section \ref{section: homo} for different purposes.
\begin{lemma}\label{lemma poco used}
    Let $F\in\mathfrak{F}$. Then there exists a continuous function $\omega\colon [0,+\infty)\times [0,+\infty)\to [0,+\infty)$, increasing with respect to each variable and with $\omega(0,0)=0$, such that for every $x_1,x_2\in\Rd$, $\zeta\in\Rd$, $\nu_1,\nu_2\in\Sd$, and $0<\rho_1<\rho_2$, the inclusion $Q_{\nu_1}(x_1,\rho_1)\subset \subset Q_{\nu_2}(x_2,\rho_2)$ implies \begin{equation}\notag
    \begin{gathered}
        \m^F(u_{x_2,\zeta,\nu_2},Q_{\nu_2}(x_2,\rho_2))\leq   \m^F(u_{x_1,\zeta,\nu_1},Q_{\nu_1}(x_1,\rho_1))+c_3|\zeta|(\rho^{d-1}_2-\rho^{d-1}_1)\\
        {}+c_4(\rho^d_2-\rho^d_1)+c_3|\zeta|\omega(\tfrac{|x_2-x_1|}{\rho_1},|\nu_1-\nu_2|)\rho^{d-1}_1.
    \end{gathered}
     \end{equation}
\end{lemma}
\begin{proof}
    The proof is obtained by the same argument used in the proof of \cite[Lemma 5.9]{DalConvex}, replacing the use of the distributional gradient $D$ by that of the symmetric distributional gradient $E$.
\end{proof}

We are now in position to prove that $g\in\mathcal{G}$.

\begin{lemma}\label{surface inclusion}
    Let $F\in\mathfrak{F}$. Then the function $g$ defined by \eqref{def g} belongs to $\mathcal{G}$.
\end{lemma}
\begin{proof}
    The proof of property (g1) can be obtained arguing  as in \cite[Lemma 5.12]{DalConvex}, replacing Lemma 5.9 of that paper by our Lemma \ref{lemma poco used}.

   Property (g2) follows immediately from (d) of Definition \ref{abstract functionals} and \eqref{def g} once we observe that $u_{x,-\zeta,-\nu}=u_{x,\zeta,\nu}-\zeta$ and that $Q_\nu(x,\rho)=Q_{-\nu}(x,\rho)$ by (d) and (e) of Section \ref{sec: Notation}.

   We prove property (g3). Let us fix $x\in\Rd$, $\zeta\in\Rd$, and $\nu\in\Sd$. As for every $\rho>0$ the function $u_{x,\zeta,\nu}$ is a competitor for the minimisation problem $\m^F(u_{x,\zeta,\nu},Q_\nu(x,\rho))$, by the upper bound in (c) of Definition \ref{abstract functionals} and (k) of Section \ref{sec: Notation} we obtain that 
   \begin{equation*}
       \m^F(u_{x,\zeta,\nu},Q_\nu(x,\rho))\leq F(u_{x,\zeta,\nu},Q_\nu(x,\rho))\leq  c_3|\zeta\odot \nu|\rho^{d-1}+c_4\rho^d.
   \end{equation*}
   Dividing by $\rho^{d-1}$, letting $\rho\to 0^+$, and  recalling \eqref{def g} we obtain the upper bound in (g3).

   To obtain the lower bound in (g3) we argue as follows. For every $\rho>0$ let  $u\in {\rm BD}(Q_\nu(x,\rho))$ be a function with $u=u_{x,\zeta,\nu}$ on $\partial Q_\nu(x,\rho)$ and observe that by the lower bound in (c) of Definition \ref{abstract functionals} we have
   \begin{equation}\notag
       c_1|Eu|(Q_\nu(x,\rho))-c_2\rho^d\leq F(u,Q_\nu(x,\rho)),
   \end{equation}
   so that 
   \begin{equation}\label{lower bound for g}
       c_1\inf\{|Eu|(Q_\nu(x,\rho)): u=u_{x,\zeta,\nu} \text{ on }\partial Q_\nu(x,\rho)\}-c_2\rho^d\leq \m^F(u_{x,\zeta,\nu},Q_\nu(x,\rho)).
   \end{equation}
  Consider the function $\Tilde{g}\colon\Rd\times \Sd\to [0,+\infty)$ defined  for every $\zeta\in\Rd$ and $\nu\in\Sd$ by
   \begin{equation}\label{def gtilde}
       \widetilde{g}(\zeta,\nu):=\limsup_{\rho\to 0^+}\frac{\inf\{|Eu|(Q_\nu(x,\rho)): u=u_{x,\zeta,\nu} \text{ on }\partial Q_\nu(x,\rho)\}}{\rho^{d-1}},
   \end{equation}
   where $x$ is an arbitrary point of $\Rd$.   By Lemma \ref{representation proposition} applied to the functional $\widetilde{F}(u,U):=|Eu|(U)$, we have
    \begin{equation}\label{widetilde g}
        \widetilde{F}^j(u,U)=\int_{J_u\cap U}\widetilde{g}([u],\nu_u)\,{\rm d}\hd
    \end{equation}
    for every $U\in\U_c(\Rd)$ and $u\in {\rm BD}(U)$. In particular, for $U=Q_\nu(x,1)$ and $u=u_{x,\zeta,\nu}$ by (k) of Section \ref{sec: Notation}  equality \eqref{widetilde g}   leads to
    $|\zeta\odot \nu|=\widetilde{g}(\zeta,\nu)$. Dividing \eqref{lower bound for g} by $\rho^{d-1}$, letting $\rho\to 0^+$, and using \eqref{def g} and \eqref{def gtilde}, we obtain the lower bound in (g3).

We conclude by showing that $g$ satisfies property (g4). Let $x\in\Rd$,  $\zeta_1,\zeta_2\in\Rd$, $\nu\in\Sd$, $\rho>0$, and $\eta>0$. Consider a function $u_2\in {\rm BD}(Q_\nu(x,\rho))$ with $u_2=u_{x,\zeta_2,\nu}$ on $\partial Q_\nu(x,\rho)$ and such that 
\begin{equation*}
   F(u_2,Q_\nu(x,\rho))\leq \m^{F}(u_{x,\zeta_2,\nu},Q_\nu(x,\rho))+\eta\rho^{d-1}.
\end{equation*}
Then, we set $u_1:=u_2-u_{x,\zeta_2,\nu}+u_{x,\zeta_1,\nu}$ and observe that $u_1=u_{x,\zeta_1,\nu}$ on $\partial Q_\nu(x,\rho)$. Using property (f) of Definition \ref{abstract functionals} and the previous inequality, it follows that 
\begin{equation}\label{sigma continuita g}
\begin{gathered}
    \m^F(u_{x,\zeta_1,\nu},Q_\nu(x,\rho))\leq F(u_{x,\zeta_1,\nu},Q_\nu(x,\rho))\leq F(u_{x,\zeta_2,\nu},Q_\nu(x,\rho))+\sigma(|\zeta_1-\zeta_2|)\rho^{d-1}\\\leq \m^F(u_{x,\zeta_2,\nu},Q_\nu(x,\rho))+\sigma(|\zeta_1-\zeta_2|)\rho^{d-1}+\eta\rho^{d-1}.
\end{gathered}
\end{equation}
Dividing this inequality by $\rho^{d-1}$, taking first the limsup for $\rho\to 0^+$ and then the limit for $\eta\to 0^+$, by \eqref{def g} we obtain
\begin{equation*}
    g(x,\zeta_1,\nu)\leq g(x,\zeta_2,\nu)+\sigma(|\zeta_1-\zeta_2|).
\end{equation*}
Exchanging the roles of $\zeta_1$ and $\zeta_2$, we obtain (g4).
\end{proof}

\section{A smaller collection of  functionals}
 In this section we introduce two subcollections of $\mathcal{F}$ and $\mathcal{G}$, denoted  by $\mathcal{F}^\alpha$ and $\mathcal{G}^\infty$, by prescribing suitable conditions on the integrands, related to their behaviour when $|A|\to +\infty$ and $|\zeta|\to +\infty$ (see Remarks \ref{remark recession infinity} and \ref{remark ginfty} below). 
 We then define a subcollection of $\mathfrak{F}$, denoted by  $\mathfrak{F}^{\alpha,\infty}$, which is stable by $\Gamma$-convergence (see Proposition \ref{alfa closure}) and includes the functionals $F^{f,g}$ with $f\in\mathcal{F}^\alpha$ and $g\in\mathcal{G}^\infty$. We shall see in Section \ref{sec:full integral} that all lower semicontinuos functionals $F\in\mathfrak{F}^{\alpha,\infty}$ admit a complete integral representation, including their Cantor part $F^c$, provided they satisfy an additional condition (see \eqref{x independence} below).

In the rest of the paper we fix three constants $0<\alpha<1$ and $c_6,c_7\geq 0$. We are ready to introduce the collections $\mathcal{F}^\alpha$ and $\mathcal{G}^\infty$.

\begin{definition}\label{defizione integrande rappresentabili}
    Let $\mathcal{F}^\alpha$ be the collection of all functions $f\in\mathcal{F}$ such that
for every $x\in\Rd$,  $A\in\R^{d\times d}_{\rm sym}$, and $s,t>0$ we have
     \begin{equation}\label{f4}\leqnomode\tag{f4}
     \big|\frac{f(x,sA)}{s}-\frac{f(x,tA)}{t}\big|\leq \frac{c_6}{s}f(x,sA)^{1-\alpha}+\frac{c_6}{s}+\frac{c_6}{t}f(x,tA)^{1-\alpha}+\frac{c_6}{t}.
     \end{equation}
    Let $\mathcal{G}^\infty$ be the collection of all functions $g\in\mathcal{G}$ such that for every $x\in\Rd$, $\zeta\in\Rd$, $\nu\in\Sd$, and $s,t>0$ we have
\begin{equation}\label{g5}\leqnomode\tag{g5}
   \big|\frac{g(x,s\zeta,\nu)}{s}- \frac{g(x,t\zeta,\nu)}{t}\big|\leq c_7\Big(\frac{g(x,s\zeta,\nu)}{s}+\frac{g(x,t\zeta,\nu)}{t}\Big)\Big(\frac{1}{s}+\frac{1}{t}\Big).
\end{equation}
\end{definition}
\begin{remark}\label{remark recession infinity}
     Property \eqref{f4} can be interpreted as a condition specifying the rate at which $f(x,tA)/t$ approaches its recession function $f^\infty(x,A)$ as $t\to+\infty$. Indeed (see \cite[Remark 3.4]{CagnettiAnnals}), the upper bound in (f2) and \eqref{f4}, are equivalent to the two conditions 
    \begin{gather}
    \label{existence finffty}f^\infty(x,A)=\lim_{t\to+\infty}\frac{f(x,tA)}{t}\quad \text{ for every $x\in\Rd$  and }A\in\Rdsym,\\\label{quantified infinity}
     \big|\frac{f(x,tA)}{t}-{f^\infty(x,A)}\big|\leq \frac{c_6}{t}+\frac{c_6}{t}f(x,tA)^{1-\alpha}\quad \text{ for every $x\in\Rd$, $t>0$, and  }A\in\Rdsym.
     \end{gather}
     
    Observe that by the growth conditions in (f2), inequality \eqref{quantified infinity} implies that
    \begin{equation}\label{quantified rate}
        \big|\frac{f(x,tA)}{t}-{f^\infty(x,A)}\big|\leq \frac{C_A}{t^\alpha}\quad \text{ for every $x\in\Rd$, $t\geq 1$, and $A\in\Rdsym$,}
    \end{equation}
    where  $C_A:=c_6+c_4^{1-\alpha}+c_6c_3^{1-\alpha}|A|^{1-\alpha}$.     Conditions similar to  \eqref{quantified infinity} and to \eqref{quantified rate} have already been considered in the literature (see, for instance, \cite[Property (H4)]{bouchitte1998global} and \cite{CagnettiAnnals,DalToaHomo,DalMasoDonati2025,Larsen}).
\end{remark}

\begin{remark}\label{remark ginfty}
    If $g\in\mathcal{G}^\infty$, then for every $x\in\Rd$, $\zeta\in\Rd$, and $\nu\in\Sd$ there exists the limit
    \begin{equation}\label{limit ginfty}
g^\infty(x,\zeta,\nu):=\lim_{t\to+\infty}\frac{g(x,t\zeta,\nu)}{t}.
    \end{equation}
    Indeed, (g3) and \eqref{g5} imply the Cauchy condition for the function $t\mapsto g(x,t\zeta,\nu)/t$. Passing to the limit in (g3) we obtain 
\begin{equation}\label{grwoth g infity}
        c_1|\zeta\odot\nu|\leq g^\infty(x,\zeta,\nu)\leq c_3|\zeta\odot\nu|
    \end{equation}
    for every $x\in\Rd$, $\zeta\in\Rd$, $\nu\in\Sd$, 
    while passing to the limit in (g4) and using \eqref{slope at infty for sigma} we get
    \begin{equation}\leqnomode\label{lipschitz for iinfity}\tag{g4$'$}
        |g^\infty(x,\zeta_1,\nu)-g^\infty(x,\zeta_2,\nu)|\leq \sigma^\infty|\zeta_1-\zeta_2|
    \end{equation}
    for every $x\in\Rd,\,\zeta_1,\zeta_2\in\Rd$, and $\nu\in\Sd$, where $\sigma^\infty$ is the non-negative constant introduced in \eqref{slope at infty for sigma}.

    Letting $s\to +\infty$ in (g5), we also obtain 
    \begin{equation}\label{ginfty quantified}
        \big|\frac{g(x,t\zeta,\nu)}{t}-g^\infty(x,\zeta,\nu)\big|\leq c_7\Big(\frac{g(x,t\zeta,\nu)}{t}+g^\infty(x,\zeta,\nu)\Big)\frac{1}{t}
    \end{equation}
    for every $x\in\Rd$, $\zeta\in\Rd$, $\nu\in\Sd$, and $t>0$.
 Thanks to the bounds in (g3) and \eqref{grwoth g infity}, this inequality implies that 
 \begin{equation}\label{difference with ginfty}
     \big|\frac{g(x,t\zeta,\nu)}{t}-g^\infty(x,\zeta,\nu)\big|\leq \frac{2c_3c_7|\zeta\odot \nu|}{t},
 \end{equation}
 while  the inequality
 \begin{equation*}
     \big|\frac{g(x,t\zeta,\nu)}{t}-g^\infty(x,\zeta,\nu)\big|\leq \frac{2c_1c_7|\zeta\odot \nu|}{t} \end{equation*}
     implies \eqref{ginfty quantified}.

Conversely, assume that $g$ is a function in $\mathcal{G}$ such that the limit in \eqref{limit ginfty} exists and  the inequality \begin{equation}\label{monotone g}
g^\infty(x,\zeta,\nu)\leq \frac{g(x,t\zeta,\nu)}{t}
    \end{equation}
   is satisfied for every $x\in\Rd$, $\zeta\in\Rd$, $\nu\in\Sd$, and $t>0$. If $g$ satisfies  also \eqref{ginfty quantified}, then $g\in\mathcal{G}^\infty$. Indeed,   by \eqref{ginfty quantified} and \eqref{monotone g} we have
   \begin{equation}\label{ginfty quantified 2}
        \big|\frac{g(x,t\zeta,\nu)}{t}-g^\infty(x,\zeta,\nu)\big|\leq 
        c_7\Big(\frac{g(x,t\zeta,\nu)}{t}+\frac{g(x,s\zeta,\nu)}{s}\Big) \frac{1}{t}
    \end{equation}
    for every $x\in\Rd$, $\zeta\in\Rd$, $\nu\in\Sd$, and $s,t>0$. Exchanging the roles of $s$ and $t$, we obtain 
    \begin{equation}\label{ginfty quantified 3}
        \big|\frac{g(x,s\zeta,\nu)}{s}-g^\infty(x,\zeta,\nu)\big|\leq 
        c_7\Big(\frac{g(x,s\zeta,\nu)}{s}+\frac{g(x,t\zeta,\nu)}{t}\Big) \frac{1}{s}.
    \end{equation}
   Using the triangle inequality, from \eqref{ginfty quantified 2} and \eqref{ginfty quantified 3} we obtain  (g5).

   Note that the existence of the limit in \eqref{limit ginfty} and inequality \eqref{monotone g} is always satisfied when $t\mapsto g(x,t\zeta,\nu)$ is concave.

\end{remark}
We now introduce a  collection of  functionals closely related to conditions \eqref{f4} and \eqref{g5}.
\begin{definition}\label{def Falpha}
    Let $\mathfrak{F}^{\alpha,\infty}$ be the collection of all functionals $F\in\mathfrak{F}$ such that for every $U\in\U_c(\Rd)$, $u\in {\rm BD}(U)$, and  $s,t>0$ we have
    
    \begin{equation}\label{(g)}\leqnomode\tag{g}
    \begin{gathered}
        \big|\frac{F(su,U)}{s}-\frac{F(tu,U)}{t}\big|\leq \frac{c_6}{s}\Ld(U)^\alpha F(su,U)^{1-\alpha}+\frac{c_6}{s}\Ld(U)\\
        +
        \frac{c_6}{t}\Ld(U)^\alpha F(tu,U)^{1-\alpha}+\frac{c_6}{t}\Ld(U)+c_7\Big(\frac{F(su,U)}{s}+\frac{F(tu,U)}{t}\Big)\Big(\frac{1}{s}+\frac{1}{t}\Big).
    \end{gathered}
    \end{equation}
The collection of functionals $F\in\mathfrak{F}^{\alpha,\infty}$ such that 
for every $U\in \U(\Rd)$ the functional $F(\cdot,U)$ is $L^1_{\rm loc}(\Rd;\Rd)$-lower semicontinuous is denoted by $\mathfrak{F}^{\alpha,\infty}_{\rm sc}$.
\end{definition}

We now show that integral functionals of the form $F^{f,g}$ belong to $\mathfrak{F}^{\alpha,\infty}$ whenever $f\in\mathcal{F}^\alpha$ and $g\in\mathcal{G}^\infty$.

\begin{proposition}\label{inclusion in falfainfty}
    Let $f\in\mathcal{F}^\alpha$ and $g\in\mathcal{G}^\infty$. Then the functional $F^{f,g}$ defined by \eqref{def:integral functionals} belongs to $\mathfrak{F}^{\alpha,\infty}$.
\end{proposition}
\begin{proof}
    Thanks to Proposition \ref{prop integrals are abstract}, we only need to prove that $F^{f,g}$ satisfies property \eqref{(g)}. To this aim, let us fix $U\in \U_c(\Rd)$, $u\in {\rm BD}(U)$, and $s,t>0$. Using \eqref{f4} and H\"older's inequality we see that 
    \begin{gather}\notag 
        \big|\int_U\frac{f(x,s\E u)}{s}\,{\rm d}x-\int_U\frac{f(x,t\E u)}{t}\,{\rm d}x\big|\leq \int_U\Big(\frac{c_6}{s}f(x,s\E u)^{1-\alpha}+\frac{c_6}{s}+\frac{c_6}{t}f(x,t\E u)^{1-\alpha}+\frac{c_6}{t}\Big){\rm d}x\\\label{bulk h}
        \leq \frac{c_6}{s}\Ld(U)^\alpha F(su,U)^{1-\alpha}+\frac{c_6}{s}\Ld(U)+\frac{c_6}{t}\Ld(U)^\alpha F(tu,U)^{1-\alpha}+\frac{c_6}{t}\Ld(U).
    \end{gather}
    Since $f^\infty$ is positively $1$-homogeneous, we immediately see that  
    \begin{equation}
        \frac{1}{s}\int_U{}f^\infty\Big(x,s\frac{{\rm d} E^cu}{{\rm d}|E^cu|}\Big)\,{\rm d}|E^cu|=   \frac{1}{t}\int_U{}f^\infty\Big(x,t\frac{{\rm d} E^cu}{{\rm d}|E^cu|}\Big)\,{\rm d}|E^cu|.
    \end{equation}
    Condition (g5) implies that 
    \begin{equation}\label{h jump}
    \begin{gathered}
\hspace{-3 cm}\big|\int_{J_u\cap U}\Big(\frac{g(x,s[u],\nu_u)}{s}-\frac{g(x,t[u],\nu_u)}{t}\Big){\rm d}\hd\big|\\\leq c_7\Big(\int_{J_u\cap U}\frac{g(x,s[u],\nu_u)}{s}\,\hd +\int_{J_u\cap U}\frac{g(x,t[u],\nu_u)}{t}\hd\Big)\Big(\frac{1}{s}+\frac{1}{t}\Big)\\
\leq c_7\Big(\frac{F(su,U)}{s}+\frac{F(tu,U)}{t}\Big)\Big(\frac{1}{s}+\frac{1}{t}\Big).
    \end{gathered}
    \end{equation}
    Finally, combining \eqref{bulk h}-\eqref{h jump} we obtain that $F^{f,g}$ satisfies property \eqref{(g)}. 
\end{proof}

We now prove that the closure of the class $\mathfrak{F}^{\alpha,\infty}$ with respect to $\Gamma$-convergence is $\mathfrak{F}^{\alpha,\infty}_{\rm sc}$.
\begin{proposition}\label{alfa closure}
    Let $\{F_n\}_{n\in\N}\subset \mathfrak{F}^{\alpha,\infty}$.\ Assume that for every $U\in\U_c(\Rd)$ the sequence $\{F_n(\cdot,U)\}_n$ $\Gamma$-converges to $F(\cdot,U)$ with respect to the topology of $L^1_{\rm loc}(\Rd;\Rd)$. Then $F\in\mathfrak{F}^{\alpha,\infty}_{\rm sc}$.
\end{proposition}
\begin{proof}
    Thanks to Theorem \ref{compactness}  we have $F\in\mathfrak{F}_{\rm sc}$. Hence, to conclude we only have to show that $F$ satisfies property \eqref{(g)}. The argument we use is a variant  of the one presented in \cite[Proposition 6.11]{DalMasoDonati2025} and \cite[Theorem 4.10]{DalToaHomo}.

Let us fix $U\in \U_c(\Rd)$, $u\in {\rm BD}(U)$, and $s,t>0$. To prove that \eqref{(g)} holds it is enough to show  
\begin{equation}\label{claim closure}
 \begin{gathered}
        \frac{F(su,U)}{s}-
        \frac{c_6}{s}\Ld(U)^\alpha F(su,U)^{1-\alpha}-\frac{c_6}{s}\Ld(U)-c_7\frac{F(su,U)}{s}\Big(\frac{1}{s}+\frac{1}{t}\Big)\\\quad \leq\frac{F(tu,U)}{t} +\frac{c_6}{t}\Ld(U)^\alpha F(tu,U)^{1-\alpha}+\frac{c_6}{t}\Ld(U)+c_7\frac{F(tu,U)}{t}\Big(\frac{1}{s}+\frac{1}{t}\Big).
    \end{gathered}
    \end{equation}
     Indeed, exchanging the role of $s$ and $t$ we obtain  \eqref{(g)}. We note that, if the left-hand side of \eqref{claim closure} is less than or equal to zero, then the inequality is trivial because the right-hand side is non-negative. Hence, we may assume that the left-hand side is  positive.
    
    Let $\{u_n\}_{n\in\N}\subset {\rm BD}(U)$ be a sequence converging to $u$  in $L^1_{\rm loc}(\Rd;\Rd)$ and such that $F_n(tu_n,U)$ converges to $ F(tu,U)$ as $n\to+\infty.$ It follows from property \eqref{(g)} applied to $F_n$ that for every $n\in\N$ we have
     \begin{equation} \label{qualcosa}
    \begin{gathered}
         \frac{F_n(su_n,U)}{s}-
        \frac{c_6}{s}\Ld(U)^\alpha F_n(su_n,U)^{1-\alpha}-\frac{c_6}{s}\Ld(U)-c_7\frac{F_n(su_n,U)}{s}\Big(\frac{1}{s}+\frac{1}{t}\Big)\\\quad \leq\frac{F_n(tu_n,U)}{t} +\frac{c_6}{t}\Ld(U)^\alpha F_n(tu_n,U)^{1-\alpha}+\frac{c_6}{t}\Ld(U)+c_7\frac{F_n(tu_n,U)}{t}\Big(\frac{1}{s}+\frac{1}{t}\Big).
    \end{gathered}
\end{equation}
    Thanks to our choice of $\{u_n\}_{n\in\N}$, we have 
\begin{gather*}
\lim_{n\to+\infty}\Big(\frac{F_n(tu_n,U)}{t}+\frac{c_6}{t}\Ld(U)^\alpha F_n(tu_n,U)^{1-\alpha}+\frac{c_6}{t}\Ld(U)+c_7\frac{F_n(tu_n,U)}{t}\Big(\frac{1}{s}+\frac{1}{t}\Big)\Big)\\= \frac{F(tu,U)}{t} +\frac{c_6}{t}\Ld(U)^\alpha F(tu,U)^{1-\alpha}+\frac{c_6}{t}\Ld(U)+c_7\frac{F(tu,U)}{t}\Big(\frac{1}{s}+\frac{1}{t}\Big).
\end{gather*}
Hence, by \eqref{qualcosa} to conclude we only need to show that 
\begin{equation}\label{last claim}
\begin{gathered} 
   \frac{F(su,U)}{s}-
        \frac{c_6}{s}\Ld(U)^\alpha F(su,U)^{1-\alpha}-\frac{c_6}{s}\Ld(U)-c_7\frac{F(su,U)}{s}\Big(\frac{1}{s}+\frac{1}{t}\Big)\\ \hspace{-0.3 cm}\leq  \liminf_{n\to+\infty}\Big( \frac{F_n(su_n,U)}{s}-
        \frac{c_6}{s}\Ld(U)^\alpha F_n(su_n,U)^{1-\alpha}-\frac{c_6}{s}\Ld(U)-c_7\frac{F_n(su_n,U)}{s}\Big(\frac{1}{s}+\frac{1}{t}\Big)\Big). 
\end{gathered}
 \end{equation}

We consider the function $\Phi\colon [0,+\infty)\to\R$ defined for every $z\in [0,+\infty)$ by
\begin{equation}\notag
    \Phi(z):=\frac{z}{s}-\frac{c_6}{s}\Ld(U)^\alpha z^{1-\alpha}-\frac{c_6}{s}\Ld(U)-c_7\frac{z}{s}\Big(\frac{1}{s}+\frac{1}{t}\Big). 
\end{equation}
We observe that $\Phi(F(su,U))$ coincides with the left-hand side of \eqref{claim closure} and \eqref{last claim}, while the  right-hand side of \eqref{last claim} coincides with $\liminf _{n}\Phi(F_n(su_n,U))$. If $1-c_7(1/s+1/t)\leq 0$ then $\Phi(z)\leq 0$ for every $z\in[0,+\infty)$,  so that \eqref{claim closure} is satisfied. If $1-c_7(1/s+1/t)>0$, then we set $z_0:=\big(c_6(1-\alpha)\Ld(U)^\alpha\big)^{1/\alpha}\big(1-c_7\big(1/s+1/t)\big)^{-1/\alpha}.$ By direct computation of $\Phi'$ we see that $\Phi$ is strictly decreasing in $[0, z_0]$ and strictly increasing in $[z_0,+\infty)$. Since $\Phi(0)\leq 0$ we deduce that $\Phi(z)\leq 0$ for $z\in [0,z_0]$. This implies that if $\Phi(z)>0$ then $z>  z_0$. Since $\Phi(F(su,U))$ coincides with the left-hand side of \eqref{last claim}, which we assumed to be positive, we have that $F(su,U)>z_0$. Moreover, by the $\Gamma$-liminf inequality  we have $F(su,U)\leq \liminf_n F_n(su_n,U)$, so that, using that $\Phi$ is increasing on $[z_0,+\infty)$, we deduce that
\begin{equation*}
   \Phi( F(su,U))\leq \liminf_{n\to+\infty} \Phi(F_n(su_n,U)).
\end{equation*}
This proves \eqref{last claim} and concludes the proof of the proposition.
\end{proof}

 We now present an inequality that allows to estimate the difference between the minimum values of some minimisation problems with Dirichlet boundary conditions involving functionals in $\mathfrak{F}^{\alpha,\infty}$. The aim of this lemma is twofold. On the one hand, it is useful to establish that the functions $f$ and $g$ defined by \eqref{def f} and \eqref{def g}  satisfy \eqref{f4} and \eqref{g5}. On the other hand, this lemma will be crucial in the proof of the integral representation of the Cantor part of functionals in $\mathfrak{F}^{\alpha,\infty}_{\rm sc}$. 
 The proof of this lemma closely resembles that of Proposition \ref{alfa closure}. It follows the lines of the proofs of \cite[Lemmas 6.12 and 6.13]{DalMasoDonati2025}, removing any truncation argument, which is not available in our context.
\begin{lemma}\label{minimum estimate}
Let $F\in\mathfrak{F}^{\alpha,\infty}$, let  $U\in\U_c(\Rd)$  with Lipschitz boundary, and let $w\in {\rm BD}(U)$. Then  
\begin{gather*}
    \big|\frac{\mathfrak{m}^F(sw,U)}{s}- \frac{\mathfrak m^F(tw,U)}{t}\big|\leq \frac{c_6}{s}\Lb^d(U)^\alpha\mathfrak{m}^F(sw,U)^{1-\alpha}+\frac{c_6}{s}\Lb^d(U)\\
    \frac{c_6}{t}\Lb^d(U)^\alpha\mathfrak{m}^F(tw,U)^{1-\alpha}+\frac{c_6}{t}\Lb^d(U)+c_7\Big(\frac{\m^F(sw,U)}{s}+\frac{\m^F(tw,U)}{t}\Big)\Big(\frac{1}{s}+\frac{1}{t}\Big).
\end{gather*}
\end{lemma}
\begin{proof}
     It  is enough to prove that
     \begin{equation}\label{claim lemma minimi}
    \begin{gathered}
        \frac{\mathfrak{m}^F(sw,U)}{s}- \frac{c_6}{s}\Lb^d(U)^\alpha\mathfrak{m}^F(sw,U)^{1-\alpha}-\frac{c_6}{s}\Lb^d(U)-c_7\frac{\m^F(sw,U)}{s}\Big(\frac{1}{s}+\frac{1}{t}\Big)\\
        \leq \frac{\mathfrak m^F(tw,U)}{t}+\frac{c_6}{t}\Ld(U)^\alpha\mathfrak{m}^F(tw,U)^{1-\alpha}+\frac{c_6}{t}\Lb(U)+c_7\frac{\m^F(tw,U)}{t}\Big(\frac{1}{s}+\frac{1}{t}\Big)
    \end{gathered}
    \end{equation}
    whenever the left-hand side is positive.

    Let $\eta>0$ and consider a function $u\in {\rm BD}(U)$ with $u=w$ on $\partial U$ and such that 
    \begin{equation}\notag
      F(tu,U)\leq \mathfrak m(tw,U)+\eta.
    \end{equation}
    Since $F$ satisfies property \eqref{(g)} of Definition \ref{def Falpha}, from this inequality we deduce that 
\begin{equation*}
 \begin{gathered}
        \frac{F(su,U)}{s}-
        \frac{c_6}{s}\Lb^d(U)^\alpha F(su,U)^{1-\alpha}-\frac{c_6}{s}\Ld(U)-c_7\frac{F(su,U)}{s}\Big(\frac{1}{s}+\frac{1}{t}\Big)\\\quad \leq\frac{F(tu,U)}{t} +\frac{c_6}{t} \Lb^d(U)^\alpha F(tu,U)^{1-\alpha}+\frac{c_6}{t}\Ld(U)+c_7\frac{F(su,U)}{s}\Big(\frac{1}{s}+\frac{1}{t}\Big)\\\notag 
        \leq \frac{\m^F(tw,U)}{t}+\frac{\eta}{t}+\frac{c_6}{t} \Ld(U)^\alpha\m^F(tw,U)^{1-\alpha}+\frac{c_6}{t} \Ld(U)^\alpha\eta^{1-\alpha}\\+\frac{c_6}{t}\Ld(U)+c_7\frac{\m^F(tw,U)}{t}\Big(\frac{1}{s}+\frac{1}{t}\Big)+c_7\frac{\eta}{s}\Big(\frac{1}{s}+\frac{1}{t}\Big).
    \end{gathered}
    \end{equation*}
   Exploiting the fact that $\m^F(sw,U)\leq F(su,U)$ and using the same argument employed at the end of the proof of Proposition \ref{alfa closure} we conclude that
   \begin{gather*} \frac{\mathfrak{m}^F(sw,U)}{s}- \frac{c_6}{s}\Lb^d(U)^\alpha\mathfrak{m}^F(sw,U)^{1-\alpha}-\frac{c_6}{s}\Lb^d(U)-c_7\frac{\m^F(sw,U)}{s}\Big(\frac{1}{s}+\frac{1}{t}\Big)
   \\   \leq \frac{\m^F(tw,U)}{t}+\frac{\eta}{t}+\frac{c_6}{t} \Ld(U)^\alpha\m^F(tw,U)^{1-\alpha}+\frac{c_6}{t} \Ld(U)^\alpha\eta^{1-\alpha}\\
   +\frac{c_6}{t}\Ld(U)+c_7\frac{\m^F(tw,U)}{t}\Big(\frac{1}{s}+\frac{1}{t}\Big)+c_7\frac{\eta}{s}\Big(\frac{1}{s}+\frac{1}{t}\Big).
   \end{gather*}
As $\eta$ is arbitrary, we obtain \eqref{claim lemma minimi}.
\end{proof}

We conclude the section by proving  a further property of the functionals $F\in\mathfrak{F}^{\alpha,\infty}$.
\begin{lemma}
    Let $F\in\mathfrak{F}^{\alpha,\infty}$. Then the functions $f$ and $g$  defined by \eqref{def f} and \eqref{def g}  belong to $\mathcal{F}^\alpha$ and $\mathcal{G}^\infty$, respectively.
\end{lemma}
\begin{proof}
   Thanks to Lemmas \ref{inclusion bulk} and \ref{surface inclusion}, to conclude it is enough to show that $f$ and $g$ satisfy properties \eqref{f4} and \eqref{g5}.
   
   We first prove that $f$ satisfies  (f4). Let us fix $x\in\Rd$, $A\in\Rdsym$, and  $s,t>0$. Thanks to Lemma \ref{minimum estimate} applied with $U=Q(x,\rho)$ and $w=\ell_A$, we have that 
\begin{gather*}
    \frac{\mathfrak{m}^F(s\ell_A,Q(x,\rho))}{s}\leq  \frac{\mathfrak m^F(t\ell_A,Q(x,\rho))}{t}+ \frac{c_6}{s}\rho^{\alpha d}\mathfrak{m}^F(s\ell_A,Q(x,\rho))^{1-\alpha}+\frac{c_6}{s}\rho^d\\
    \frac{c_6}{t}\rho^{\alpha d}\mathfrak{m}^F(t\ell_A,Q(x,\rho))^{1-\alpha}+\frac{c_6}{t}\rho^d+c_7\Big(\frac{\m^F(s\ell_A,Q(x,\rho))}{s}+\frac{\m^F(t\ell_A,Q(x,\rho))}{t}\Big)\Big(\frac{1}{s}+\frac{1}{t}\Big).
\end{gather*}
Dividing this inequality by $\rho^d$, letting $\rho\to 0^+$, and recalling \eqref{def f}, we obtain
\begin{equation}\notag 
\begin{gathered}
      \frac{f(x,sA)}{s}\leq \frac{f(x,tA)}{t}+  \frac{c_6}{s}f(x,sA)^{1-\alpha}+\frac{c_6}{s}+
    \frac{c_6}{t}f(x,tA)^{1-\alpha}\\+\frac{c_6}{t}+c_7\Big(\frac{f(x,sA)}{s}+\frac{f(x,tA)}{t}\Big)\Big(\frac{1}{s}+\frac{1}{t}\Big).
\end{gathered}
\end{equation}
Exchanging the roles of $s$ and $t$ we obtain
\begin{equation}\notag
\begin{gathered}
      \frac{f(x,tA)}{t}\leq \frac{f(x,sA)}{s}+  \frac{c_6}{s}f(x,sA)^{1-\alpha}+\frac{c_6}{s}+
    \frac{c_6}{t}f(x,tA)^{1-\alpha}\\+\frac{c_6}{t}+c_7\Big(\frac{f(x,sA)}{s}+\frac{f(x,tA)}{t}\Big)\Big(\frac{1}{s}+\frac{1}{t}\Big).
\end{gathered}
\end{equation}

As these inequalities hold for every $s$ and $t$, we may let $s\to +\infty$  and obtain for every $t>0$ that 
\begin{equation}\notag  
     \big|f^\infty(x,A)-\frac{f(x,tA)}{t}\big|\leq
    \frac{c_6}{t}f(x,tA)^{1-\alpha}\\+\frac{c_6}{t}+\frac{c_7}{t}\Big({f^\infty(x,A)}+\frac{f(x,tA)}{t}\Big).
\end{equation}
Let us fix $\lambda>0$ and apply this inequality with $A$ replaced by $\lambda A$ to get 
\begin{gather}\notag 
     \big|\lambda f^\infty(x, A)-\frac{f(x,t\lambda A)}{t}\big|\leq
    \frac{c_6}{t}f(x,t\lambda A)^{1-\alpha}+\frac{c_6}{t}+\frac{c_7}{t}\Big({\lambda f^\infty(x, A)}+\frac{f(x,t\lambda A)}{t}\Big)
\end{gather}
for every $t>0$, hence 
\begin{gather}\notag 
     \big|f^\infty(x, A)-\frac{f(x,t\lambda A)}{\lambda t}\big|\leq
    \frac{c_6}{\lambda t}f(x,t\lambda A)^{1-\alpha}+\frac{c_6}{\lambda t}+\frac{c_7}{t}\Big({ f^\infty(x, A)}+\frac{f(x,t\lambda A)}{\lambda t}\Big)
\end{gather}
for every $t>0$ and $\lambda>0$. Setting $\tau=t\lambda$
\begin{gather}\notag
     \big|f^\infty(x, A)-\frac{f(x,\tau A)}{\tau }\big|\leq
    \frac{c_6}{\tau }f(x,\tau  A)^{1-\alpha}+\frac{c_6}{\tau }+\frac{c_7}{t}\Big({ f^\infty(x, A)}+\frac{f(x,\tau A)}{\tau}\Big)
\end{gather}
for every $t>0$ and $\tau>0$, so that, letting $t\to +\infty$ we obtain 
\begin{gather}\label{last inequalirt}
     \big|f^\infty(x, A)-\frac{f(x,\tau A)}{\tau}\big|\leq
    \frac{c_6}{\tau }f(x,\tau  A)^{1-\alpha}+\frac{c_6}{\tau }
\end{gather}
for every $\tau>0$.  Recalling that $f$ satisfies (f2), we may use Remark \ref{remark recession infinity} to deduce from \eqref{last inequalirt} that $f$ satisfies (f4), concluding the proof of the inclusion $f\in\mathcal{F}^\alpha$.

We now prove that $g$ satisfies property (g5). Let $x\in\Rd$, $\zeta\in\Rd$, and $\nu\in\Sd$. Lemma \ref{minimum estimate} applied with $U=Q_\nu(x,\rho)$ and $w=u_{x,\zeta,\nu}$ implies that for every $s,t>0$ we have 
\begin{gather*}
    \frac{\mathfrak{m}^F(su_{x,\zeta,\nu},Q_\nu(x,\rho))}{s}\leq \frac{\mathfrak m^F(tu_{x,\zeta,\nu},Q_\nu(x,\rho))}{t}+ \frac{c_6}{s}\rho^{\alpha d}\mathfrak{m}^F(su_{x,\zeta,\nu},Q_\nu(x,\rho))^{1-\alpha}+\frac{c_6}{s}\rho^d\\
    \frac{c_6}{t}\rho^{\alpha d}\mathfrak{m}^F(tu_{x,\zeta,\nu},Q_\nu(x,\rho))^{1-\alpha}+\frac{c_6}{t}\rho^d\\+c_7\Big(\frac{\m^F(su_{x,\zeta,\nu},Q_\nu(x,\rho))}{s}+\frac{\m^F(tu_{x,\zeta,\nu},Q_\nu(x,\rho))}{t}\Big)\Big(\frac{1}{s}+\frac{1}{t}\Big).
\end{gather*}
Dividing this inequality by $\rho^{d-1}$, letting $\rho\to 0^+$, and recalling \eqref{def g} we deduce
\begin{equation}\notag 
   \frac{g(x,s\zeta,\nu)}{s}\leq \frac{g(x,t\zeta,\nu)}{t}+c_7\big(\frac{g(x,s\zeta,\nu)}{s}+\frac{g(x,t\zeta,\nu)}{t}\Big)\Big(\frac{1}{s}+\frac{1}{t}\Big).
\end{equation}
Exchanging the roles of $s$ and $t$ we obtain (g5).
\end{proof}

\section{Full integral representation}\label{sec:full integral}
In this section we prove a complete integral representation result for functionals $F$ in $\mathfrak{F}^{\alpha,\infty}_{\rm sc}$, including the representation of the Cantor part $F^c$, assuming that  the  function  $f$, defined in \eqref{def f}, does not depend on $x$. 

A full integral representation result for functionals $F\in\mathfrak{F}_{\rm sc}$ has recently been proved in \cite{CaroccFocardiVan}, under the additional assumption of uniform continuity of $F$ with respect to translations of the independent variables. This condition was originally considered  in \cite[Lemma 3.11]{bouchitte1998global} for the corresponding integral representation problem in ${\rm BV}$.

In the case of periodic homogenisation one can easily cheque that the $\Gamma$-limit functional $F_{\rm hom}$ (see Theorem \ref{homogeneous homogenization} below) is invariant under translations  (see, for instance,  \cite[Theorem 6.14]{CaroccFocardiVan}). However, in general, this invariance (nor the continuity with respect to translations) cannot be checked directly in the case of non-periodic homogenisation. For this reason, we will prove Theorem \ref{theorem cantor} below under the sole additional assumption \eqref{x independence}, which is much weaker than invariance under translations (see Example \ref{simple example}).
We shall see in Section \ref{section: homo} that condition \eqref{x independence} is satisfied almost surely under the standard hypotheses of stochastic homogenisation.

The following theorem is the main result of this section.
\begin{theorem}\label{theorem cantor}Let $F\in\mathfrak{F}^{\alpha,\infty}_{\rm sc}$ and let $g$ be the function defined by \eqref{def g}. Assume that there exist a function $f\colon\Rdsym\to[0,+\infty)$ and a Borel set $N\subset \Rd$, $\sigma$-finite with respect to $\hd$,  such that
\begin{equation}\label{x independence}
   f(A)=\lim_{\rho \to 0^+}\frac{\m^F(\ell_A,Q(x,\rho))}{\rho^d} \quad \text{for every }x\in\Rd\setminus N\text{ and }A\in\Rdsym.
\end{equation}
Then for every $U\in\U_c(\Rd)$ and  $u\in {\rm BD}(U)$ we have
\begin{equation}\notag 
    F(u,B)=\int_Bf(\E u)\,{\rm d}x+\int_Bf^\infty\Big(\frac{{\rm d}E^cu}{{\rm d}|E^cu|}\Big)\,{\rm d}|E^cu|+\int_{J_u\cap B}g(x,[u],\nu_u)\,{\rm d}\hd
\end{equation}
for every $B\in \B(U)$.
\end{theorem}

We first state a useful lemma, obtained by Caroccia, Focardi, and Van Goethem in \cite[Lemma 5.3]{CaroccFocardiVan} (see \cite[Lemma 3.7]{bouchitte1998global} for an analogous result in the ${\rm BV}$-setting), which characterises the Radon-Nikod\'ym derivative ${\rm d}F^c(u,\cdot)/{\rm d}|E^cu|$  by means of suitable minimum values of minimisation problems on small parallelograms. To state this result, we set some further notation.

We recall that De Philippis and Rindler proved in \cite[Theorem 1.18]{DePhilRindler} (see also the survey \cite{DePhilRindlersurvey} and the book \cite{Rindler}) the following remarkable theorem: given an open set $U\in \U(\Rd)$ and  $u\in {\rm BD}(U)$ there exist two Borel maps $a,b\colon U\to \Rd$  such that 
\begin{equation}\label{rank 2}
    \frac{{\rm d}E^cu}{{\rm d}|E^cu|}=a\odot b\quad\text{ and }  \quad |a\odot b|=1 \quad  \text{ $|E^cu|$-a.e.\ in }U. 
\end{equation}
 For every $\lambda>0$ and every pair $(a,b)\in\Rd\times \Rd$ with $a \neq \pm b$ we set
\begin{equation*}
    P^{a,b}_\lambda:=\big\{z\in U:|z\cdot b|<\tfrac\lambda2,\,\,|z\cdot a|<\tfrac12,\,|z\cdot \theta_i|<\tfrac12\, \text{ for }i\in\{1,...,n-2\}\big\},
\end{equation*}
where $\{\theta_i\}_{i=1}^{n-2}\subset \Sd$ is such that $\{a,b,\theta_1,...,\theta_{n-2}\}$ is a basis of $\Rd$, while if $a=\pm b$ we set
\begin{equation*}
    P^{a,b}_\lambda:=\big\{z\in U:|z\cdot b|<\tfrac\lambda2,\,\,|z\cdot \theta_i|<\tfrac12\, \text{ for }i\in\{1,...,n-1\}\big\},
\end{equation*}
where $\{\theta_i\}_{i=1}^{n-1}\subset \Sd$ is such that $\{b,\theta_1,...,\theta_{n-1}\}$ is a basis of $\Rd$. The choice of the vectors $\{\theta_i\}_i$ in the previous definitions is irrelevant for the arguments that follow. Given $x\in\Rd$ and $\rho>0$, we also set $P^{a,b}_\lambda(x,\rho):=x+\rho P_\lambda^{a,b}$.  Moreover, for every point $x\in U$ such that \eqref{rank 2} holds, we set $P^x_\lambda(x,\rho):=P^{a(x),b(x)}_\lambda(x,\rho)$.

The following result characterises ${\rm d}F^c(u,\cdot)/{\rm d}|E^cu|$ in terms of the double limit of infima of problems related to parallelograms of the form $P^x_\lambda$.

\begin{lemma}[{\!\!\cite[Lemma 5.3]{CaroccFocardiVan}}]\label{lemma focardi}Let $F\in\mathfrak{F}_{\rm sc}$, $U\in \U(\Rd)$, and $u\in {\rm BD}(U)$. Then there exists $C(u)\in\B(U)$, with $|E^cu|(U\setminus C(u))=0$, such that for every $x\in C(u)$ equality \eqref{rank 2} holds and  there exist a positive sequence $\{\lambda_{j}\}_{j\in\N}$, converging to $0$ as $j\to +\infty$,  and for every $j\in\N$ a positive sequence $\{\rho_{i,j}\}_{i\in\N}$, converging to $0$ as $i\to+\infty$, such that, setting \begin{equation}\label{def A e sij}
    A=A(x):=a(x)\odot b(x)\quad \text{and}\quad s_{i,j}:=\frac{|Eu|(P^x_{\lambda_j}(x,\rho_{i,j}))}{\Ld(P^x_j(x,\rho_{i,j}))},
\end{equation}
we have
\begin{gather} \label{eq Foc div}
 \text{ for every $j\in\N$ the sequence }\{s_{i,j}\}_{i\in\N} \text{ tends to $+\infty$ as }  i\to+\infty,
\\
\label{eq Foc}
    \frac{{\rm d} F^c(u,\cdot)}{{\rm d}|E^cu|}(x)=\lim_{j\to+\infty}\limsup_{i\to+\infty}\frac{\m^F(s_{i,j}\ell_{A},P^x_{\lambda_j}(x,\rho_{i,j}))}{s_{i,j}\Ld(P^x_{\lambda_j}(x,\rho_{i,j}))}.
\end{gather}
\end{lemma}
\begin{proof}
    This result is proved in \cite[Lemma 5.3]{CaroccFocardiVan} under slightly different hypotheses on the functional $F$. In particular, they assume a condition, which they denote by (H4), that prescribes a joint continuity of the functional $F$ with respect to  translations both of the dependent and of the independent variables.  Not all functionals in $\mathfrak{F}^{\alpha,\infty}$ satisfy this property. However, examining the proof of \cite[Lemma 5.3]{CaroccFocardiVan} one can see that the use of property (H4) can be replaced by the property (e) of Definition \ref{abstract functionals}. This invariance property also allows us to consider the  boundary conditions appearing in \eqref{eq Foc}  instead of those  used in \cite[(5.6)]{CaroccFocardiVan}, which differ from ours by a rigid motion.
\end{proof}

The following result shows that, under suitable assumptions on $F$, in definition \eqref{def f} we may replace the cube $Q(x,\rho)$ by  parallelograms of the form $P^{a,b}_{\lambda}(x,\rho)$. 
\begin{lemma}\label{change of parallelogram}
      Let $F\in\mathfrak{F}$, $A\in \Rdsym$, $a,b\in\Rd$, and  $\lambda>0$. Assume that  there exists $\mu\geq 0$  such that 
    \begin{equation*}
        \m^F(\ell_A,Q(y,\rho))\leq \mu \rho^d
    \end{equation*}
    for every $y\in\Rd$ and $\rho>0$.
    Then
    \begin{equation*}
\m^F(\ell_A,P^{a,b}_\lambda(y,\rho))\leq \mu \Lb^d(P^{a,b}_\lambda(y,\rho))
    \end{equation*}
    for every $y\in\Rd$ and $\rho>0$.  If, in addition, for some $x\in\Rd$ we have
    \begin{equation*}
        \lim_{\rho \to 0^+}\frac{\m^F(\ell_{A}, Q(x,\rho))}{\rho^d}=\mu,
    \end{equation*}
     then 
     \begin{equation*}
         \lim_{\rho \to 0^+} \frac{\m^F(\ell_{A}, P^{a,b}_\lambda(x,\rho))}{\Ld(P^{a,b}_\lambda(x,\rho))}=\mu.
     \end{equation*}
\end{lemma}
\begin{proof}
    The proof can be obtained arguing  as in \cite[Lemma 5.3]{DalToaHomo}, observing that in the last part of that lemma limsup can be replaced by lim. We also remark that, in contrast with \cite[Lemma 5.3]{DalToaHomo}, we deal with minimisation problems without constraints on the oscillation of the competitors, which simplifies the proof.
\end{proof}

We are now ready to prove Theorem \ref{theorem cantor}.

\medskip

\noindent{\it Proof of Theorem \ref{theorem cantor}.}  By \eqref{x independence} we have 
\begin{equation}\label{equality ookk}
    f(A)=\limsup_{\rho\to 0^+}\frac{\m^F(\ell_A,Q(x,\rho))}{\rho^d}\quad \text{ for $\Ld$-a.e.\ $x\in\Rd$ and every $A\in\Rdsym$.}
\end{equation}
Let us fix $U\in\U_c(\Rd)$ and $u\in {\rm BD}(U)$. Thanks to \eqref{somma facj}, \eqref{equality ookk}, and Proposition \ref{representation proposition}, it is enough to show that 
    \begin{equation*}
F^c(u,B)=\int_{B}f^\infty\Big(\frac{{\rm d}E^cu}{{\rm d}|E^cu|}\Big){\rm d}|E^cu|
    \end{equation*}
    for every $B\in \B(U)$. To prove this we will show that
    \begin{equation}\label{radon cantor}
        \frac{{\rm d}F^c(u,\cdot)}{{\rm d}|E^cu|}(x)=f^\infty\Big(\frac{{\rm d}E^cu}{{\rm d}|E^cu|}(x)\Big)
    \end{equation}
    for $|E^cu|$-a.e.\ $x\in U$. Let $C(u)\subset U$ be the set of Lemma \ref{lemma focardi} and let us fix $x\in C(u)\setminus N$, where $N$ is as in the statement. 
    
 Thanks to Lemma \ref{lemma focardi} there exist a positive sequence $\{\lambda_{j}\}_{j\in\N}$ converging to $0$ as $j\to +\infty$, and for every $j\in\N$ a positive sequence $\{\rho_{i,j}\}_{i\in\N}$, converging to $0$ as $i\to+\infty$, such that \eqref{eq Foc div} and \eqref{eq Foc} hold. For simplicity of notation, we set $P_j^x(x,\rho_{i,j}):=P^{x}_{\lambda_j}(x,\rho_{i,j})$.  To obtain \eqref{radon cantor} is enough to show that 
    \begin{equation}\label{claim radon cantor}
    f^\infty\Big(\frac{{\rm d}E^cu}{{\rm d}|E^cu|}(x)\Big)=\lim_{j\to+\infty}\limsup_{i\to+\infty}\frac{\m^F(s_{i,j}\ell_{A},P^x_j(x,\rho_{i,j}))}{s_{i,j}\Ld(P^x_j(x,\rho_{i,j}))},
    \end{equation}
    where $A=A(x)$ and $s_{i,j}$ are defined by \eqref{def A e sij}.

     By \eqref{def f}, \eqref{integral representation bulk}, and \eqref{x independence} we have $$\m^F(s_{i,j}\ell_{A},P^x_j(x,\rho_{i,j})\leq F(s_{i,j}\ell_A,P^x_j(x,\rho_{i,j}))=f(s_{i,j}A)\Ld(P^x_j(x,\rho_{i,j})) ,$$ so that  by \eqref{eq Foc div}   we have
    \begin{equation}\label{upper bound cantor}
\limsup_{i\to+\infty}\frac{\m^F(s_{i,j}\ell_{A},P^x_j(x,\rho_{i,j}))}{s_{i,j}\Ld(P^x_j(x,\rho_{i,j}))}\leq \limsup_{i\to +\infty} \frac{f(s_{i,j}A)}{s_{i,j}}=f^\infty(A),
    \end{equation}
    for every $j\in\N$.

Then, we observe that  for every $y\in\Rd$, $\rho>0$,  and $t>0$  by \eqref{def f}, \eqref{integral representation bulk}, and \eqref{x independence}  we have 
    \begin{gather}\notag \m^{F}(t\ell_A,Q(y,\rho))\leq F(t\ell_A,Q(y,\rho))=f(tA)\rho^d 
      \quad \text{ and }\quad \lim_{\rho\to 0^+}\frac{\m^F(t\ell_A,Q(x,\rho))}{\rho^d}=f(tA).
    \end{gather}
  Hence, we can  use Lemma \ref{change of parallelogram} to  deduce  that
\begin{equation}\label{limitazzo} 
          \lim_{\rho\to 0^+}\frac{\m^F(t\ell_A,P^x_\lambda(x,\rho))}{\Lb^d(P^x_\lambda(x,\rho))}=f(tA)
    \end{equation}
    for every $\lambda,t>0$. Recalling \eqref{existence finffty}, from the previous equality it follows that  
    \begin{equation}\label{recession parallelogram}
           \lim_{t\to +\infty}\lim_{\rho\to 0^+}\frac{\m^F(t\ell_A,P^x_\lambda(x,\rho))}{t\Lb^d(P^x_\lambda(x,\rho))}=f^\infty(A).
    \end{equation}

 Then, we apply Lemma \ref{minimum estimate} with $w=\ell_A$ and $U=P^x_j(x,\rho_{i,j})$ to get for every $t>0$  
 \begin{equation}\label{minimum recession}
    \begin{gathered}
\big|\frac{\mathfrak{m}^F(s_{i,j}\ell_A,P^x_j(x,\rho_{i,j}))}{s_{i,j}\Ld(P^x_j(x,\rho_{i,j}))}- \frac{\mathfrak m^F(t\ell_A,P^x_j(x,\rho_{i,j}))}{t\Ld(P^x_j(x,\rho_{i,j}))}\big|\leq \frac{c_6}{s_{i,j}}\Big(\frac{\m^F(s_{i,j}\ell_A,P^x_j(x,\rho_{i,j}))}{\Ld(P^x_j(x,\rho_{i,j}))}\Big)^{1-\alpha}\\+\frac{c_6}{s_{i,j}}+
    \frac{c_6}{t}\Big(\frac{\mathfrak{m}^F(t\ell_A,P^x_j(x,\rho_{i,j}))}{\Ld(P^x_j(x,\rho_{i,j}))}\Big)^{1-\alpha}+\frac{c_6}{t} \\
+c_7\Big(\frac{\m^F(s_{i,j}\ell_A,P^x_j(x,\rho_{i,j}))}{s_{i,j}\Ld(P^x_j(x,\rho_{i,j}))}+\frac{\m^F(t\ell_A,P^x_j(x,\rho_{i,j}))}{t\Ld(P^x_j(x,\rho_{i,j}))}\Big)\Big(\frac{1}{s_{i,j}}+\frac{1}{t}\Big).
\end{gathered}
\end{equation}
We observe that for every $\tau>0$ it follows from the upper bound in (c) of Definition \ref{abstract functionals} that 
\begin{equation}\label{estimate from above}
    \m^F(\tau\ell_A,P^x_j(x,\rho_{i,j}))\leq F(\tau\ell_A,,P^x_j(x,\rho_{i,j}))\leq (c_3\tau |A|+c_4)\Ld(P^x_j(x,\rho_{i,j})).
\end{equation}

Let us fix $\e>0$. Recalling  \eqref{existence finffty}, we may find $t>0$ such that 
\begin{equation}\label{epsilon} 
    |\frac{f(tA)}{t}-f^\infty(A)|<\e\quad \text{and }\quad \frac{c_6}{t}+ \frac{c_6}{t^\alpha}(c_3|A|+\frac{c_4}{t})^{1-\alpha}+c_7\big(2c_3|A|+\frac{c_4}{t}\big)\frac{1}{t}<\e.
    \end{equation}
 Thanks to \eqref{limitazzo}, from the first inequality in \eqref{epsilon} we deduce that  for every $j\in\N$ we have
 \begin{equation*}
       f^\infty(A)-\e< \frac{f(tA)}{t}=\lim_{i\to +\infty}\frac{\mathfrak{m}^F(t\ell_A,P^x_j(x,\rho_{i,j}))}{t\Ld(P^x_j(x,\rho_{i,j}))}.
 \end{equation*}

\noindent Combining this inequality  with    \eqref{minimum recession}, we get  
 \begin{gather}\notag
     f^\infty(A)-\e\leq \lim_{j\to+\infty}\lim_{i\to +\infty}\frac{\mathfrak{m}^F(t\ell_A,P^x_j(x,\rho_{i,j}))}{t\Ld(P^x_j(x,\rho_{i,j}))}\leq \limsup_{j\to+\infty}\limsup_{i\to+\infty}\Big(\frac{\mathfrak{m}^F(s_{i,j}\ell_A,P^x_j(x,\rho_{i,j}))}{s_{i,j}\Ld(P^x_j(x,\rho_{i,j}))}\\ \notag+ \frac{c_6}{s_{i,j}}\Big(\frac{\m^F(s_{i,j}\ell_A,P^x_j(x,\rho_{i,j}))}{\Ld(P^x_j(x,\rho_{i,j}))}\Big)^{1-\alpha}+\frac{c_6}{s_{i,j}}+
    \frac{c_6}{t}\Big(\frac{\mathfrak{m}^F(t\ell_A,P^x_j(x,\rho_{i,j}))}{\Ld(P^x_j(x,\rho_{i,j}))}\Big)^{1-\alpha}+\frac{c_6}{t}\\\label{stima lunghetta}
    +c_7\Big(\frac{\m^F(s_{i,j}\ell_A,P^x_j(x,\rho_{i,j}))}{s_{i,j}\Ld(P^x_j(x,\rho_{i,j}))}+\frac{\m^F(t\ell_A,P^x_j(x,\rho_{i,j}))}{t\Ld(P^x_j(x,\rho_{i,j}))}\Big)\Big(\frac{1}{s_{i,j}}+\frac{1}{t}\Big)\Big).
 \end{gather}

Using \eqref{estimate from above}, we see that the right-hand side of the previous chain of inequalities can be bounded from above by 
 \begin{gather}
\limsup_{j\to+\infty}\limsup_{i\to+\infty}\Big(\frac{\mathfrak{m}^F(s_{i,j}\ell_A,P^x_j(x,\rho_{i,j}))}{s_{i,j}\Ld(P^x_j(x,\rho_{i,j}))} \notag+ \frac{c_6}{s_{i,j}^\alpha}\Big(c_3|A|{+}\frac{c_4}{s_{i,j}}\Big)^{1-\alpha}\!\!\!+\frac{c_6}{s_{i,j}}+
    \frac{c_6}{t^\alpha}(c_3|A|+\frac{c_4}{t})^{1-\alpha}\!\!+\frac{c_6}{t}\\\label{stima lunghettta 2}
    +c_7\Big(2c_3|A|+\frac{c_4}{s_{i,j}}+\frac{c_4}{t}\Big)\Big(\frac{1}{s_{i,j}}+\frac{1}{t}\Big)\Big).
 \end{gather}
Recalling \eqref{upper bound cantor} and \eqref{eq Foc div}, from  \eqref{epsilon}-\eqref{stima lunghettta 2} we infer
    \begin{equation*}
        f^\infty(A)-\e<\lim_{j\to+\infty}\limsup_{i\to+\infty}\frac{\mathfrak{m}^F(s_{i,j}\ell_A,P^x_j(x,\rho_{i,j}))}{s_{i,j}\Ld(P^x_j(x,\rho_{i,j}))}+\e \leq  f^\infty(A)+\e.
    \end{equation*}
As $\e>0$ is arbitrary, we obtain \eqref{claim radon cantor}. Since $x\in C(u)\setminus N$, $|E^cu|(U\setminus C(u))=0$, and   $|E^cu|(N)=0$ (see part (ii) of (k) of Section \ref{sec: Notation}), this concludes the proof.\qed

\medskip

 It is easy to produce examples of functionals $F\in\mathfrak{F}^{\alpha,\infty}_{\rm sc}$ which are not invariant under translation of the independent variable, but satisfy \eqref{x independence}. 
\begin{example}\label{simple example}
 Assume that $0<c_1<c_3$ and let $f$ and $g$ be the functions defined for every  $A\in\Rdsym$, $x\in\Rd$, $\zeta\in\Rd$, and $\nu\in\Sd$ by 
  \begin{gather}\notag
      f(A):=c_3|A| \quad \text{and}\quad g(x,\zeta,\nu):=\begin{cases}
          c_3|\zeta\odot\nu| &\text{if }x\in \Rd\setminus \Pi^{e_d},\\
           c_1|\zeta\odot\nu|&\text{if }x\in\Pi^{e_d},
      \end{cases}
  \end{gather}
  where $e_d=(0,\dots,1)$ and $\Pi^{e_d}:=\{y\in \Rd: y\cdot e_d=0\}$. Clearly $f\in\mathcal{F}$ and $g\in\mathcal{G}$, and since they are positively homogeneous of degree one we  also have $f\in\mathcal{F}^\alpha$ and $g\in\mathcal{G}^\infty$. Hence, 
 the functional $F:=F^{f,g}$ belongs to $\mathfrak{F}^{\alpha,\infty}$ by Proposition \ref{inclusion in falfainfty}. 
   Let $\psi\colon\Rd\to \R$ be the function defined by \begin{equation*}
       \psi(x):=\begin{cases}
           c_3 &\text{if }x\in \Rd\setminus \Pi^{e_d},\\
           c_1&\text{if }x\in\Pi^{e_d}.
       \end{cases}
   \end{equation*}
 We observe that for every $U\in\U_c(\Rd)$ and $u\in {\rm BD}(U)$ we have
   \begin{equation*}
       F(u,U)=\int_U\psi\,{\rm d}|Eu|.
   \end{equation*}
   Since $\psi$ is lower semicontinuous in $\Rd$, this equality shows that $F(\cdot,U)$ is $L^1_{\rm loc}(\Rd;\Rd)$-lower semicontinuous, hence  $F\in\mathfrak{F}^{\alpha,\infty}_{\rm sc}$.
 
\end{example}

\medskip
\section{Characterisation of \texorpdfstring{$\Gamma$}{Γ}-convergence using minima on small cubes}\label{sec:gamma convergence}
In this section we determine the integrands of the $\Gamma$-limit  of a sequence  $\{F_n\}_{n\in\N}\subset \mathfrak{F}^{\alpha,\infty}$ by means of limits of the minimum values of suitable minimisation problems for $F_n$ on small cubes. 

\medskip

 The next lemma shows that the $\Gamma$-convergence of a sequence $\{F_{n}\}_{n\in\N}\subset \mathfrak{F}$ to a limit functional $F$ allows us to compare $\m^F$ with $\{\m^{F_n}\}_{n\in\N}$.

\begin{lemma}\label{Lemma: existence of limit in rho}
      Let $\{F_n\}_{n\in\N}\subset \mathfrak{F}$.  Assume that there exists $F\in\mathfrak{F}_{\rm sc}$ such that for every $U\in\U_c(\Rd)$ the sequence $\{F_n(\cdot,U)\}_{n\in\N}$ $\Gamma$-converges to $F(\cdot,U)$ with respect to the topology of  $L^1_{\rm loc}(\Rd;\Rd)$. Let $U,W\in \U_c(\Rd)$ with Lipschitz boundary and  $W\subset\subset U$, and let $w\in {\rm BD}(U)$.
      Then   
      \begin{gather}
   \label{liming geq point}
\m^F(w,U)\leq\liminf_{n\to+\infty}\m^{F_n}(w,W)+c_3|Ew|(U\setminus W)+c_4\Ld(U\setminus W),\\\label{limsup leq point}
 \limsup_{n\to+\infty}\m^{F_n}(w,U)\leq \m^F(w,U).
      \end{gather}
\end{lemma}
\begin{proof}
    We consider a subsequence, not relabelled, such that the liminf in the right-hand side of \eqref{liming geq point} is actually a limit. To prove \eqref{liming geq point} we fix $\delta>0$ and for every $n\in\N$   we select $z
_n\in {\rm BD}(W)$  such that $u_n=w$ on $\partial W$ and 
\begin{equation}\label{quasiminimiser linign}
    F_n(z_n,W)<\m^{F_n}(w,W)+\delta.
\end{equation}
We observe that by the upper bound in (c) of Definition \ref{abstract functionals}, we have 
\begin{equation*}
    F_n(z_n,W)<\m^{F_n}(w,W)+\eta\leq F_n(w,W)+\delta\leq  c_3|Ew|(W)+c_4\Ld(W)+\delta,
\end{equation*}
so that by the lower bound in (c) we also get that  $
   |Ez_n|(W)$ is bounded uniformly with respect to $n$.
   We extend $z_n$ to $U$ by setting $z_n=w$ on $U\setminus W$ and observe that $
   |Ez_n|(U) $ is uniformly bounded in $n$ as well. Hence, there exist a subsequence, not relabelled, and a function $z\in {\rm BD}(U)$ such that $z_n\to z$ in $L^1(U;\Rd)$ and  $z=w$ on $U\setminus W$, so that 
   \begin{gather}\label{dise}
    \m^F(w,U)\leq F(z,U)\leq \liminf_{n\to+\infty}F_n(z_n,U),
   \end{gather}
   where the last inequality is due to $\Gamma$-convergence. Since the inner trace of $z_n$ on $\partial W$ equals  the inner trace of $w$, we have $|Ez_n| (\partial W)=|Ew|(\partial W)$.  Hence, 
   the upper bound in (c) of Definition \ref{abstract functionals} gives
   \begin{gather*}
    F_n(z_n,U\setminus W)\leq c_3|Ez_n|(U\setminus W)+c_4\Ld(U\setminus W)= c_3|Ew|(U\setminus W)+c_4\Ld(U\setminus W),
   \end{gather*}
   where the last equality is due to the fact that $z_n=w$ in the open set $U\setminus \overline{W}$ and $|Ez_n| (\partial W)=|Ew|(\partial W)$.
  Combining this inequality with \eqref{quasiminimiser linign} and  \eqref{dise}, and letting $\delta\to 0^+$ we obtain \eqref{liming geq point}.
  
 To prove \eqref{limsup leq point}, let us fix $\delta>0$ and  let  $u\in {\rm BD}(U)$ be such that $u=w$ on $\partial U$ and $
        F(u,U)< \m^{F}(w,U)+\delta.$
    By $\Gamma$-convergence there exists a sequence $\{u_n\}_{n\in\N}\subset {\rm BD}(U)$ converging   to $u$ strongly  in $L^1(U;\Rd)$ and such that \begin{equation}
    \label{recovery for affine}
\lim_{n\to+\infty}F_n(u_n,U)= F(u,U)<\m^{F}(w,U)+\delta.
    \end{equation} 
    
    We now fix a compact set $K\subset U$ such that
    \begin{equation}\label{smallness on u}
        c_3|Eu|(U\setminus K)+c_4\Ld(U\setminus K)<\delta.
    \end{equation}
    We also consider two additional open sets $U''$ and $U'$ with the property that $K\subset U''\subset\subset  U'\subset \subset U$. We now argue as in the part of the proof of Lemma \ref{fundamental estimate}  that starts from \eqref{convex combination}, with $V=U\setminus U''$,  $v_n$ replaced by $u$, and for every $n\in\N$ we construct a function $z_n\in {\rm BD}(U)$ such that $z_n=u$ in a neighbourhood of $\partial U$ and 
    \begin{gather*}
        F_n(z_n,U)\leq  \,\,F_n(u_n,U)+F_n(u,U\setminus U'')+\frac{M}{m}+\frac{2c_3}{\eta}\|u_n-u\|_{L^1(U;\Rd)},
    \end{gather*}
    where $M\geq 0$ is the constant introduced in \eqref{bound M} and  $\eta:={\rm dist}(U,\partial U')$. Since $z_n=u=w$ on $\partial U$, we have  $\m^{F_n}(w,U)\leq F_n(z_n,U)$. Hence, using the upper bounds in property (c) of Definition \ref{abstract functionals} for $F_n$, \eqref{smallness on u}, and the inclusion $K\subset U''$, from the previous displayed formula we obtain 
    \begin{equation*}
        \m^{F_n}(w,U)\leq F_n(u_n,U)+\delta+\frac{M}{m}+\frac{2c_3}{\eta}\|u_n-u\|_{L^1(U;\Rd)}.
    \end{equation*}
   We now pass to the limsup as $n\to+\infty$ and, recalling that $u_n\to u$ in $L^1(U;\Rd)$ as $n\to+\infty$, from \eqref{recovery for affine} we get
    \begin{equation}\notag
    \limsup_{n\to+\infty}\m^{F_n}(w,U)\leq \m^F(w,U)+2\delta+\frac{M}{m}.
    \end{equation}
    Taking the limit as $m\to+\infty$ and $\delta\to0^+$ we obtain \eqref{limsup leq point}.
\end{proof}

The next result shows that for a functional $F\in\mathfrak{F}_{\rm sc}$ arising as $\Gamma$-limit of a sequence of functionals in $\mathfrak{F}$ the integrands $f$ and $g$ of the bulk and surface parts can be obtained using  limits of the minimum values of the minimisation problems $\m^{F_n}(\ell_A,Q(x,\rho))$  and $\m^{F_n}(u_{x,\zeta,\nu},Q_\nu(x,\rho))$ for the functionals $F_n$.

\begin{lemma}\label{asymptotic minimum}
  Let $\{F_n\}_{n\in\N}\subset \mathfrak{F}$.  Assume that there exists $F\in\mathfrak{F}_{\rm sc}$ such that for every $U\in\U_c(\Rd)$ the sequence $\{F_n(\cdot,U)\}_{n\in\N}$ $\Gamma$-converges to $F(\cdot,U)$ with respect to the topology of $L^1_{\rm loc}(\Rd;\Rd)$, and let $f$ and $g$ be the functions associated to $F$ by \eqref{def f} and \eqref{def g}. Then for every $x\in\Rd$, $A\in\Rdsym$, $\zeta\in\Rd$, and $\nu\in\Sd$ we have
  
 \begin{gather*}
f(x,A)=\limsup_{\rho\to 0^+}\limsup_{n\to+\infty}\frac{\m^{F_n}(\ell_A,Q(x,\rho))}{\rho^d}=\limsup_{\rho\to 0^+}\liminf_{n\to+\infty}\frac{\m^{F_n}(\ell_A,Q(x,\rho))}{\rho^d},\\
g(x,\zeta,\nu)=\limsup_{\rho\to 0^+}\limsup_{n\to+\infty}\frac{\m^{F_n}(u_{x,\zeta,\nu},Q_\nu(x,\rho))}{\rho^{d-1}}=\limsup_{\rho\to 0^+}\liminf_{n\to+\infty}\frac{\m^{F_n}(u_{x,\zeta,\nu},Q_\nu(x,\rho))}{\rho^{d-1}}.
    \end{gather*}
    \end{lemma}
    \begin{proof}
       The proof can be obtained arguing exactly as in \cite[Lemma 3.3]{DalToaHomo},  replacing $D$ with $E$, $\nabla$ with $\E$, and Propositions 3.1 and 3.2 of that paper with our Lemma \ref{Lemma: existence of limit in rho}.
    \end{proof}
  
We conclude this section by proving a sufficient condition for the $\Gamma$-convergence of a sequence $F_n$, based on limits of minimum values of minimisation  problems for $F_n$ on small cubes. Note that we require that the limits corresponding to the volume integrand do not depend on $x$.
\begin{theorem}\label{theorem sufficient}
    Let $\{F_n\}_{n\in\N}\subset\mathfrak{F^{\alpha,\infty}} $. Assume that there exist  $\hat{f}\colon\Rdsym\to[0,+\infty)$ and $\hat{g}
    \colon\Rd\times \Rd\times \Sd\to [0,+\infty)$ such that 
    \begin{gather}\label{1}
\hat{f}(A)=\lim_{\rho\to 0^+}\limsup_{n\to+\infty}\frac{\m^{F_n}(\ell_A,Q(x,\rho))}{\rho^d}=\lim_{\rho\to 0^+}\liminf_{n\to+\infty}\frac{\m^{F_n}(\ell_A,Q(x,\rho))}{\rho^d},\\\label{2}
\hat{g}(x,\zeta,\nu)=\limsup_{\rho\to 0^+}\limsup_{n\to+\infty}\frac{\m^{F_n}(u_{x,\zeta,\nu},Q_\nu(x,\rho))}{\rho^{d-1}}=\limsup_{\rho\to 0^+}\liminf_{n\to+\infty}\frac{\m^{F_n}(u_{x,\zeta,\nu},Q_\nu(x,\rho))}{\rho^{d-1}},
    \end{gather}
    for every $x\in\Rd$, $A\in\Rdsym$, $\zeta\in\Rd$, and $\nu\in\Sd$. Then $\hat{f}\in\mathcal{F}^\alpha$, $\hat{g}\in\mathcal{G}^\infty$, and for every $U\in\mathcal{U}_c(\Rd)$ we have that $\{F_n(\cdot,U)\}_{n\in\N}$ $\Gamma$-converges to $F^{\hat{f},\hat{g}}(\cdot,U)$ with respect to the topology of $L^1_{\rm loc}(\Rd;\Rd)$, where $F^{\hat{f},\hat{g}}$ is the functional introduced in Definition \ref{def:integral functionals}.
\end{theorem}
\begin{proof}
    By Theorem \ref{compactness} there exists a subsequence, not relabelled, and a functional $F\in\mathfrak{F}_{\rm sc}$ such that $\{F_n(\cdot,U)\}_{n\in\N}$ $\Gamma$-converges ot $F(\cdot,U)$ for every $U\in\U_c(\Rd)$.  
Let $f$ and $g$ be the functions defined by \eqref{def f} and \eqref{def g} corresponding to $F$. Using Lemmas \ref{Lemma: existence of limit in rho} and  \ref{asymptotic minimum}, together with \eqref{1} and \eqref{2}, we see that $f(x,A)=\hat{f}(A)$ and $g(x,\zeta,\nu)=\hat{g}(x,\zeta,\nu)$ for every $x\in\Rd$, $\zeta\in\Rd$, $A\in\Rdsym$, and $\nu\in\Sd$. 
Moreover, 
thanks to Lemma \ref{Lemma: existence of limit in rho} we obtain that 
    \begin{gather*}
        \hat{f}(A)=\lim_{\rho\to 0^+}\limsup_{n \to +\infty}\frac{\m^{F_n}(\ell_A,Q(x,\rho))}{\rho^d}\leq \liminf_{\rho\to 0^+}\frac{\m^F(\ell_A,Q(x,\rho))}{\rho^d},\\
          \hat{f}(A)=\lim_{\rho\to 0^+}\liminf_{n \to +\infty}\frac{\m^{F_n}(\ell_A,Q(x,\rho-\rho^2))}{(\rho-\rho^2)^d}\geq \limsup_{\rho\to 0^+}\frac{\m^F(\ell_A,Q(x,\rho))}{\rho^d},
    \end{gather*}
    hence
    \begin{equation}\notag
        \hat{f}(A)=\lim_{\rho\to 0^+}\frac{\m^F(\ell_A,Q(x,\rho))}{\rho^d}\quad \text{ for every $x\in\Rd$ and $A\in\Rdsym$}.
    \end{equation}
     Thus, the functional $F$ satisfies the hypotheses of Theorem \ref{theorem cantor}, so that $F=F^{\hat{f},\hat{g}}$.

     Moreover, by \eqref{1} and \eqref{2} the functions $\hat{f}$ and $\hat{g}$, and hence  the functional $F^{\hat{f},\hat{g}}$, do not depend on the  subsequence chosen at the beginning of the proof, so that by the Urysohn property of $\Gamma$-convergence (see \cite[Proposition 8.3]{DalBook})  for every $U\in \U_c(\Rd)$ the whole sequence $\{ F_n(\cdot,U)\}_{n\in\N}$ $\Gamma$-converges to $F^{\hat{f},\hat{g}}(\cdot,U)$.
\end{proof}

\section{Functionals obtained by rescaling}\label{section: homo}

In this section we use the theory developed in the previous sections to deal with the problem of $\Gamma$-convergence of oscillating functionals obtained by rescaling of a single functional. In particular, we prove a general theorem which provides sufficient conditions for the $\Gamma$-convergence of these functionals (see Theorem \ref{homogeneous homogenization}).

\medskip 
We now introduce the notation we will use in the rest of the work. Throughout this section we keep fixed  $f\in\mathcal{F}^\alpha$ and $g\in\mathcal{G}^{\rm \infty}$ and set
\begin{equation}\notag 
F:=F^{f,g}\in\mathfrak{F}^{\alpha,\infty},\end{equation}
where $F^{f,g}$ is the functional introduced in Definition \ref{def:integral functionals}.
We also assume that the modulus of continuity introduced in  \eqref{slope at infty for sigma} satisfies \begin{equation}\label{def sigma}
\sigma(\tau)=\sigma_1\tau\quad \text{for every }\tau\geq 0,\end{equation} for some constant $\sigma_1>0$. With this hypothesis, condition (g4) of Definition \ref{bulk and surface integrands} reads
\begin{equation}\notag
    |g(x,\zeta_1,\nu)-g(x,\zeta_2,\nu)|\leq \sigma_1|\zeta_1-\zeta_2|,
\end{equation}
while the surface continuity estimate (f) of Definition \ref{abstract functionals} becomes 
\begin{equation*}
  F(u+u_{x,\zeta,\nu},B)\leq F(u,B)+\sigma_1|\zeta|\Hd{\Pi^\nu_x\cap B}.
\end{equation*}

\begin{definition}\label{definition rescaled}
For every $\e>0$  we  consider the integrands $f_\e\in\mathcal{F}^\alpha$ and $g_\e\in\mathcal{G}^\infty$ defined by 
\begin{gather}\label{def fe}
    f_\e(x,A):=f\big(\tfrac{x}{\e},A\big) \quad \text{ for }x\in\Rd \text{ and }A\in\Rdsym,\\\label{def ge}
    g_\e(x,\zeta,\nu):=\e g\big(\tfrac{x}{\e},\tfrac{\zeta}{\e},\nu\big) \quad \text{ for }x\in\Rd, \zeta\in\Rd, \text{ and }\nu\in\Sd,
\end{gather}
and we set $F_\e:=F^{f_\e,g_\e}\in\mathfrak{F}^{\alpha,\infty}$.
\end{definition}

\begin{remark}
    If for every $x\in\Rd$, $\zeta\in\Rd$, and $\nu\in\Sd$ the function $t\mapsto g(x,t\zeta,\nu)$ is positively homogeneous of degree one, that is to say
     \begin{equation}\label{homogeneity}
        g(x,t\zeta,\nu)=tg(x,\zeta,\nu)\quad \text{for every } x\in\Rd, \,\zeta\in\Rd, \, \nu\in\Sd,\, \text{ and }t\geq 0, 
    \end{equation}
    then
    $$
    g_\e(x,\zeta,\nu)=g(\tfrac{x}{\e},\zeta,\nu) \quad \text{ for every }x\in\Rd, \, \zeta\in\Rd,\, \nu\in\Sd,\, \text{ and }\e>0.
    $$
    In this case, for $U\in \U_c(\Rd)$ and $u\in {\rm BD}(U)$  the functionals $F_\e(u,U)$ of Definition \ref{definition rescaled} become \begin{equation}\label{homogenizzato}
       \int_{U}f\Big(\frac{x}{\e},\E u\Big)\,{\rm d}x +\int_{U}f^\infty\Big(\frac{x}{\e},\frac{{\rm d}E^cu}{{\rm d}|E^cu|}\Big)\,{\rm d}|E^cu|+\int_{J_u\cap U}g\Big(\frac{x}{\e},[u],\nu_u\Big)\,{\rm d}\hd,    \end{equation}
       which are the functionals commonly considered in homogenisation of free discontinuity problems.
      
      Our choice in the definition of $g_\e$ given by \eqref{def ge} is justified by the fact that the corresponding functional $F_\e$ defined by $F^{f_\e,g_\e}$ satisfies good change of variables formulas (see Lemmas \ref{lemma change of bulk variables}  and \ref{lemma change of surface variables} below) even when $g$ does not satisfy \eqref{homogeneity}. This will allow us to prove a very general $\Gamma$-convergence result for $F^{f_\e,g_\e}$ (see Theorem \ref{homogeneous homogenization}), which implies the $\Gamma$-convergence of the functionals in \eqref{homogenizzato} when $g$ satisfies the additional condition \eqref{homogeneity}.
      Unfortunately, in the ${\rm BD}$ case we are not able to extend the truncation arguments that were crucial  to study the analogue of \eqref{homogenizzato} in ${\rm BV}$ (see \cite{CagnettiAnnals,DalToaHomo, DalMasoDonati2025}). 
\end{remark}

In the rest of the section we need the following technical result about a change of variables formula involving the Cantor part of a ${\rm BD}$ function.
\begin{lemma}\label{lemma change cantor}
    Let $\e,\rho\in(0,1)$, $x\in\Rd$, $\nu \in\Sd$, $u\in {\rm BD}(Q_\nu(x,\rho))$, and $v\in {\rm BD}(Q_\nu(\frac{x}{\e},\frac{\rho}{\e})) $. Assume that $v(z)=\frac{1}{\e}u(\e z)$ for every $z\in Q_\nu(\frac{x}{\e},\frac{\rho}{\e}).$ Then      \begin{gather}\label{ponte}
          \int_{Q_\nu(x,\rho)}f^\infty_\e\Big(y,\frac{{\rm d}E^cu}{{\rm d}|E^cu|}\Big)\,{\rm d}|E^cu|=\e^d\int_{Q_\nu(\frac{x}{\e},\frac{\rho}{\e})}f^\infty\Big(z,\frac{{\rm d}E^cv}{{\rm d}|E^cv|}\Big)\,{\rm d}|E^cv|.
    \end{gather}
\end{lemma}
\begin{proof}
     
    Let $\psi\in C^\infty_c(Q_\nu(\frac{x}{\e},\frac{\rho}{\e});\Rdsym)$ and let $\psi_\e:=\psi\big(\frac{\cdot}{\e}\big)\in C^\infty_c(Q_\nu(x,\rho);\Rdsym)$. By  change of variables and integration by parts we have 
     \begin{gather*}
         -\int_{Q_\nu(\frac{x}{\e},\frac{\rho}{\e})}\psi\,{\rm d}Ev=\int_{Q_\nu(\frac{x}{\e},\frac{\rho}{\e})}v\,{\rm div}\,\psi \,{\rm d}z=\int_{Q_\nu(\frac{x}{\e},\frac{\rho}{\e})}\tfrac{1}{\e}u(\e z)\,{\rm div}\,\psi \,{\rm d}z\\ =\frac
         {1}{\e^d}\int_{Q_\nu(x,\rho)}\frac{u(y)}{\e}{\rm div}\,\psi\big(\frac{y}{\e}\big)\,{\rm d}y=\frac{1}{\e^d}\int_{Q_\nu(x,\rho)}u \,{\rm div}\,\psi_\e\,{\rm d}y\\
         =-\frac{1}{\e^d}\int_{Q_\nu(x,\rho)}\psi_\e\,{\rm d}Eu
         =-\frac{1}{\e^d}\int_{Q_\nu(\frac{x}{\e},\frac{\rho}{\e})}\psi\,{\rm d}\Big(\big(\frac{\cdot}{\e}\big)_{\#}Eu\Big),
     \end{gather*}
     where $(\frac{\cdot}{\e}\big)_\#Eu$ denotes the push-forward of $Eu$ via the function $y\mapsto\frac{y}{\e}$ (see, for instance, \cite[Theorem 3.6.1]{Bogachev}). This implies that 
\begin{equation*}
        Ev=\frac{1}{\e^d}\Big(\frac{\cdot}{\e}\Big)_{\#}Eu \quad \text{ as Borel measures on }Q_\nu(\tfrac{x}{\e},\tfrac{\rho}{\e})
    \end{equation*}
   and, passing to their singular parts,
    \begin{equation*}
        E^sv=\frac{1}{\e^d}\Big(\frac{\cdot}{\e}\Big)_{\#}E^su \quad \text{ as Borel measures on }Q_\nu(\tfrac{x}{\e},\tfrac{\rho}{\e}),
    \end{equation*}
    where $E^sv$ and $E^su$ denote the singular part of $Ev$ and $Eu$ with respect to the Lebesgue measure. Restricting the previous equality to $Q_\nu(\tfrac{x}{\e},\tfrac{\rho}{\e})\setminus J_v=Q_\nu(\tfrac{x}{\e},\tfrac{\rho}{\e})\setminus (\frac{1}{\e}J_u)$, we deduce that
    \begin{equation*}
        E^cv=\frac{1}{\e^d}\Big(\frac{\cdot}{\e}\Big)_{\#}E^cu\quad \text{ as Borel measures on }Q_\nu(\tfrac{x}{\e},\tfrac{\rho}{\e}).
    \end{equation*}
    This implies that 
    \begin{equation*}
        \frac{{\rm d}E^cv}{{\rm d}|E^cv|}\big(\frac{y}{\e}\big)= \frac{{\rm d}E^cu}{{\rm d}|E^cu|}(y)\quad \text{ for } |E^cu|\text{-a.e.\ } y\in Q_\nu(x,\rho).
    \end{equation*}
    Hence, by the integration formula for the push-forward of measures (see, for instance, \cite[Theorem 3.6.1]{Bogachev}) we get \eqref{ponte}.
\end{proof}

We state two preliminary results that allow us to rewrite the minimisation problems on small cubes for the functionals $F_\e$ by means of minimisation problems on large cubes for $F$ and $F^{f^\infty,g^\infty}$, where $f^\infty$ and $g^\infty$ are the functions defined by \eqref{existence finffty} and \eqref{limit ginfty}, respectively.
\begin{lemma}\label{lemma change of bulk variables}
    Let $\e,\rho\in(0,1)$, $x\in\Rd$, and $A\in\Rdsym$. Then
    \begin{equation}\notag 
     \m^{F_\e}(\ell_A,Q(x,\rho))=\e^d\m^{F^{f,g}}(\ell_A,Q(\tfrac{x}{\e},\tfrac{\rho}{\e})).
\end{equation}
\end{lemma}
\begin{proof}
    Let $u\in {\rm BD}(Q(x,\rho))$, with $u=\ell_A$ on $\partial Q(x,\rho)$, and let $v(z):=\frac{1}{\e}u(\e z)$ for $z\in Q(\frac{x}{\e},\tfrac{\rho}{\e})$. It is easy to see that $v=\ell_A$ on $\partial Q(\frac{x}{\e},\frac{\rho}{\e})$. Thanks to \eqref{def fe} and \eqref{def ge}, a change of variables and Lemma \ref{lemma change cantor} show that
    \begin{gather*}
        \int_{Q(x,\rho)}f_\e(y,\E u)\,{\rm d}y=\e^d\int_{Q(\frac{x}{\e},\frac{\rho}{\e})}f(z,\E v)\,{\rm d}z,\\
          \int_{Q(x,\rho)}f^\infty_\e\Big(y,\frac{{\rm d}E^cu}{{\rm d}|E^cu|}\Big)\,{\rm d}|E^cu|=\e^d\int_{Q(\frac{x}{\e},\frac{\rho}{\e})}f^\infty\Big(z,\frac{{\rm d}E^cv}{{\rm d}|E^cv|}\Big)\,{\rm d}|E^cv|,\\
          \int_{J_u\cap Q(x,\rho)}g_\e(y,[u],\nu_u)\,{\rm d}\hd=\e^d\int_{J_v\cap Q(\frac{x}{\e},\frac{\rho}{\e})}g(z,[v],\nu_v)\,{\rm d}\hd.
    \end{gather*}
    Thus, we infer
\begin{equation}\notag
    \m^{F_\e}(\ell_A,Q(x,\rho))=\e^{d}\m^{F}(\ell_A,Q(\tfrac{x}{\e},\tfrac{\rho}{\e})),
\end{equation}
concluding the proof.
\end{proof}

\begin{definition}\label{def Finfty}
    We set 
    \begin{equation}\notag 
F^\infty:=F^{f^\infty,g^\infty}\in\mathfrak{F}^{\alpha,\infty},
    \end{equation}
    where $f^\infty$ and $g^\infty$ are  the functions defined by \eqref{existence finffty} and \eqref{limit ginfty}.
\end{definition}

\begin{lemma}\label{lemma change of surface variables}
    Let $\e\in (0,1/(2c_6))$, $x\in\Rd$, $\zeta\in\Rd$, $\nu\in\Sd$, and $\rho\in (0,1)$. Then 
    \begin{equation}\label{claim change of surface}
        \big|\m^{F_\e}(u_{x,\zeta,\nu},Q_\nu(x,\rho))-\e^{d-1}\m^{F^\infty}(u_{\frac{x}{\e},\zeta,\nu},Q_\nu(\tfrac{x}{\e},\tfrac{\rho}{\e}))\big|\leq C\rho^{d-1+\alpha}+C\e\rho^{d-1},
    \end{equation}
    where $C=C_\zeta>0$ is a constant depending only on $|\zeta|$ and on the structural constants $\alpha, c_3$, $c_6$, and $c_7$, but is independent of $\e$, $x$, $\nu$, and $\rho$.
\end{lemma}
\begin{proof}
Let $u\in {\rm BD}(Q_\nu(x,\rho))$ with $u=u_{x,\zeta,\nu}$ on $\partial Q_\nu(x,\rho)$ and let $v(z):=u(\e z)$ for $z\in Q_\nu(\frac{x}{\e},\tfrac{\rho}{\e})$, so that $v=u_{\frac{x}{\e},\zeta,\nu}$ on $\partial Q_\nu(\frac{x}{\e},\frac{\rho}{\e})$. Recalling \eqref{def fe} and \eqref{def ge}, a change of variables and Lemma \ref{lemma change cantor} imply that
    \begin{gather}\label{bulk}
        \int_{Q_\nu(x,\rho)}f_\e(y,\E u)\,{\rm d}y=\e^{d-1}\int_{Q_\nu(\frac{x}{\e},\frac{\rho}{\e})} \e f\Big(z,\frac{1}{\e}\E v\Big)\,{\rm d}z,\\ \label{cantor}
\int_{Q_\nu(x,\rho)}f^\infty_\e\Big(y,\frac{{\rm d}E^cu}{{\rm d}|E^cu|}\Big)\,{\rm d}|E^cu|=\e^{d-1}\int_{Q_\nu(\frac{x}{\e},\frac{\rho}{\e})}f^\infty\Big(z,\frac{{\rm d}E^cv}{{\rm d}|E^cv|}\Big)\,{\rm d}|E^cv|,\\\label{jump}
          \int_{J_u\cap Q_\nu(x,\rho)}g_\e(y,[u],\nu_u)\,{\rm d}\hd=\e^{d-1}\int_{J_v\cap Q_\nu(\frac{x}{\e},\frac{\rho}{\e})} \e g\Big(z,\frac{1}{\e}[v],\nu_v\Big)\,{\rm d}\hd.
    \end{gather}    

We can now exploit \eqref{quantified infinity} of  Remark \ref{remark recession infinity}  to obtain for $\Ld$-a.e.\ $z\in  Q_\nu(\tfrac{x}{\e},\tfrac{\rho}{\e})$
\begin{gather}\label{minus 1}
    |\e f(z,\tfrac{1}{\e}{\E v})-f^\infty(z,\E v)|\leq c_6\e+c_6\e f(z,\tfrac{1}{\e}{\E v})^{1-\alpha},
 \end{gather}
 while using \eqref{difference with ginfty} of Remark \ref{remark ginfty} and inequalities  (g3) and \eqref{grwoth g infity},   for $\hd$-a.e.\ $z\in J_v\cap Q_\nu(\tfrac{x}{\e},\tfrac{\rho}{\e})$ we get  
 \begin{gather}
    \label{zeroth}
    |\e g(z,\tfrac{1}{\e}{[v]},\nu_v)-g^\infty(z,[v],\nu_v)|\leq2c_3c_7\e|[v]\odot\nu_v|\leq  C_1\e^2 g(z,\tfrac{1}{\e}{[v]},\nu_v),\\
    \label{zeroth1}
    |\e g(z,\tfrac{1}{\e}{[v]},\nu_v)-g^\infty(z,[v],\nu_v)| \leq2c_3c_7\e|[v]\odot\nu_v|\leq C_1\e g^\infty(z,[v],\nu_v),
\end{gather}
where we have set $C_1:=(2c_3c_7)/c_1$. From \eqref{minus 1} we then deduce  for $\Ld$-a.e.\ $z\in  Q_\nu(\tfrac{x}{\e},\tfrac{\rho}{\e})$ that
\begin{gather}\label{first}
    f^\infty(z,\E v)\leq\e f(z,\tfrac{1}{\e}{\E v})+ c_6\e+c_6\e f(z,\tfrac{1}{\e}{\E v})^{1-\alpha},\\\notag
    f^\infty(z,\E v)\geq \e f(z,\tfrac{1}{\e}{\E v})-c_6\e-c_6\e f(z,\tfrac{1}{\e}{\E v})^{1-\alpha}\\\label{second}\geq \tfrac{\e}{2} f(z,\tfrac{1}\e{\E v})-c_6\e-c_6\e(2c_6(1-\alpha))^{\frac{1-\alpha}{\alpha}},
\end{gather}
where we have used the inequality  $\tau^{1-\alpha}\leq \frac{1}{2c_6}\tau+(2c_6(1-\alpha))^{\frac{1-\alpha}{\alpha}}$ for every $\tau\geq 0$ and $\e\in (0,1/(2c_6))$, while from \eqref{zeroth} and \eqref{zeroth1} we obtain for $\hd$-a.e.\ $z\in J_v\cap Q_\nu(\tfrac{x}{\e},\tfrac{\rho}{\e})$ 
\begin{gather}\label{third}
    g^\infty(z,[v],\nu_v)\leq \e g(z,\tfrac{1}{\e}{[v]},\nu_v)+ C_1\e^2 g(z,\tfrac{1}{\e}{[v]},\nu_v),\\\label{fourth}
     g^\infty(z,[v],\nu_v)\geq  \e g(z,\tfrac{1}{\e}{[v]},\nu_v)- C_1\e g^\infty(z,{[v]},\nu_v).
\end{gather}
Note that \eqref{second} implies that 
\begin{equation*}
    \e f(z,\tfrac{1}{\e}{\E v})\leq 2f^\infty(z,\E v)+C_2\e,
\end{equation*}
where $C_2:=2c_6+2c_6(2c_6(1-\alpha))^{\frac{1-\alpha}{\alpha}}$, so that from  \eqref{minus 1} we infer 
\begin{equation}\label{fifth}
    \e f(z,\tfrac{1}{\e}{\E v})\leq f^\infty(z,\E v)+c_6\e+ c_6\e^\alpha(2f^\infty(z,\E v)+C_2\e)^{1-\alpha}.
\end{equation}

Integrating \eqref{first} on $Q_\nu\big(\tfrac{x}{\e},\tfrac{\rho}{\e}\big)$, by H\"older's inequality we get 
\begin{equation}\label{sixth}
\begin{gathered}
   \e^{d-1} \int_{Q_\nu(\frac{x}{\e},\frac{\rho}{\e})} f^\infty(z,\E v)\,{\rm d}z\leq  \e^{d-1} \int_{Q_\nu(\frac{x}{\e},\frac{\rho}{\e})} \e f(z,\tfrac{1}{\e}{\E v})\,{\rm d}z+c_6\rho^d\\
   +c_6\Big(\e^{d-1}\int_{Q_\nu(\frac{x}{\e},\frac{\rho}{\e})} \e f(z,\tfrac{1}{\e}{\E v})\,{\rm d}z \Big)^{1-\alpha}\rho^{\alpha d},
\end{gathered}
\end{equation}\vspace{-0.1 cm}
while from \eqref{fifth} we get
\begin{equation}\label{seventh}
\begin{gathered}
   \e^{d-1} \int_{Q_\nu(\frac{x}{\e},\frac{\rho}{\e})} \e f(z,\tfrac{1}{\e}{\E v})\,{\rm d}z\leq  \e^{d-1} \int_{Q_\nu(\frac{x}{\e},\frac{\rho}{\e})} f^\infty(z,{\E v})\,{\rm d}z+c_6\rho^d\\
   +c_6\Big(2\e^{d-1}\int_{Q_\nu(\frac{x}{\e},\frac{\rho}{\e})} f^\infty(z,{\E v})\,{\rm d}z +C_2\rho^d \Big)^{1-\alpha}\rho^{\alpha d}.
\end{gathered}
\end{equation}
Integrating \eqref{third} and \eqref{fourth} on $J_v\cap Q_\nu\big(\tfrac{x}{\e},\tfrac{\rho}{\e}\big)$ we obtain 
\begin{equation}\label{eighth}
\begin{gathered}
     \e^{d-1}\int_{J_v\cap Q_\nu(\tfrac{x}{\e},\tfrac{\rho}{\e})}g^\infty(z,[v],\nu_v)\,{\rm d}\hd\leq \e^{d-1}\int_{J_v\cap Q_\nu(\tfrac{x}{\e},\tfrac{\rho}{\e})} \e g(z,\tfrac{1}{\e}{[v]},\nu_v)\,{\rm d}\hd\\+\e^{d}{C_1}\int_{J_v\cap Q_\nu(\tfrac{x}{\e},\tfrac{\rho}{\e})} \e 
     g(z,\tfrac{1}{\e}{[v]},\nu_v)\,{\rm d}\hd,
     \end{gathered}
     \end{equation}
     and 
     \begin{equation}
     \label{ninth}
     \begin{gathered}
     \e^{d-1}   \int_{J_v\cap Q_\nu(\tfrac{x}{\e},\tfrac{\rho}{\e})}g^\infty(z,[v],\nu_v)\,{\rm d}\hd\geq \e^{d-1}\int_{J_v\cap Q_\nu(\tfrac{x}{\e},\tfrac{\rho}{\e})} \e g(z,\tfrac{1}{\e}{[v]},\nu_v)\,{\rm d}\hd\\- \e^{d
     }C_1\int_{J_v\cap Q_\nu(\tfrac{x}{\e},\tfrac{\rho}{\e})}g^\infty(z,[v],\nu_v)\,{\rm d}\hd.
\end{gathered}
\end{equation}

Recalling \eqref{bulk}-\eqref{jump} and combining  \eqref{sixth} and \eqref{eighth},  we obtain 
\begin{gather*}
    \e^{d-1}F^{\infty}(v,Q_\nu(\tfrac{x}{\e},\tfrac{\rho}{\e}))\leq F_\e(u,Q_\nu(x,\rho))+c_6\rho^d\\
     +c_6F_\e(u,Q_\nu(x,\rho))^{1-\alpha}\rho^{\alpha d}
    +C_1\e F_\e(u,Q_\nu(x,\rho)),
\end{gather*}
while from \eqref{seventh} and \eqref{ninth} we get
\begin{gather*}
    F_\e(u,Q_\nu(x,\rho))\leq \e^{d-1}F^{\infty}(v,Q_\nu(\tfrac{x}{\e},\tfrac{\rho}{\e}))+c_6\rho^d\\
    +c_6\Big(2\e^{d-1}F^{\infty}(v,Q_\nu(\tfrac{x}{\e},\tfrac{\rho}{\e})) +C_2\rho^d \Big)^{1-\alpha}\rho^{\alpha d}
    +C_1\e^d F^{\infty}(v,Q_\nu(\tfrac{x}{\e},\tfrac{\rho}{\e})).
\end{gather*}
Since $u=u_{x,\zeta,\nu}$ on $\partial Q_\nu(x,\rho)$ if and only  $v=u_{\frac{x}{\e},\zeta,\nu}$ on $\partial Q_\nu(\frac{x}{\e},\frac{\rho}{\e})$, the last two inequalities imply 
\begin{equation}\label{tenth}
\begin{gathered}
    \e^{d-1}\m^{F^{\infty}}(u_{\frac{x}{\e},\zeta,\nu},Q_\nu(\tfrac{x}{\e},\tfrac{\rho}{\e}))\leq \m^{F_\e}(u_{x,\zeta,\nu},Q_\nu(x,\rho))+c_6\rho^d\\
     +c_6\m^{F_\e}(u_{x,\zeta,\nu},Q_\nu(x,\rho))^{1-\alpha}\rho^{\alpha d}
    +C_1\e \m^{F_\e}(u_{x,\zeta,\nu},Q_\nu(x,\rho)),
\end{gathered}
\end{equation}
and 
\begin{equation}\label{eleventh}
\begin{gathered}
  \m^{F_\e}(u_{x,\zeta,\nu},Q_\nu(x,\rho))\leq \e^{d-1}\m^{F^{\infty}}(u_{\tfrac{x}{\e},\zeta,\nu},Q_\nu(\tfrac{x}{\e},\tfrac{\rho}{\e}))+c_6\rho^d\\
\hspace{-0.3 cm}+c_6\Big(2\e^{d-1}\m^{F^{\infty}}(u_{\frac{x}{\e},\zeta,\nu},Q_\nu(\tfrac{x}{\e},\tfrac{\rho}{\e})) +C_2\rho^d \Big)^{1-\alpha}\rho^{\alpha d}
    +C_1\e^d \m^{F^{\infty}}(u_{\frac{x}{\e},\zeta,\nu},Q_\nu(\tfrac{x}{\e},\tfrac{\rho}{\e})).
\end{gathered}
\end{equation}
To conclude, as $u_{x,\zeta,\nu}$ and $u_{\frac{x}{\e},\zeta,\nu}$ are competitors for the  problems $ \m^{F_\e}(u_{x,\zeta,\nu},Q_\nu(x,\rho))$ and $m^{F^{\infty}}(u_{\frac{x}{\e},\zeta,\nu},Q_\nu(\tfrac{x}{\e},\tfrac{\rho}{\e}))$, respectively,  using the upper bounds in (g3) and \eqref{grwoth g infity}  we obtain
\begin{gather*} 
 \m^{F_\e}(u_{x,\zeta,\nu},Q_\nu(x,\rho))\leq c_3|\zeta\odot\nu|\rho^{d-1},\\  \m^{F^{\infty}}(u_{\frac{x}{\e},\zeta,\nu},Q_\nu(\tfrac{x}{\e},\tfrac{\rho}{\e}))\leq c_3|\zeta\odot\nu|\frac{\rho^{d-1}}{\e^{d-1}}.
 \end{gather*}
 Finally, combining these two inequalities with \eqref{tenth} and \eqref{eleventh} we get
 \begin{gather*}
      \e^{d-1}\m^{F^{\infty}}(u_{\frac{x}{\e},\zeta,\nu},Q_\nu(\tfrac{x}{\e},\tfrac{\rho}{\e}))\leq \m^{F_\e}(u_{x,\zeta,\nu},Q_\nu(x,\rho))+c_6\rho^d\\
     +c_6c_3|\zeta\odot\nu|\rho^{d-1+\alpha}
    +C_1 c_3|\zeta\odot\nu|\e\rho^{d-1},
 \end{gather*}
 and 
 \begin{gather*}
     \m^{F_\e}(u_{x,\zeta,\nu},Q_\nu(x,\rho))\leq \e^{d-1}\m^{F^{\infty}}(u_{\tfrac{x}{\e},\zeta,\nu},Q_\nu(\tfrac{x}{\e},\tfrac{\rho}{\e}))+c_6\rho^d\\
    +c_6\big(2c_3|\zeta\odot\nu| +C_2 \big)^{1-\alpha}\rho^{d-1+\alpha}
    +C_1c_3|\zeta\odot\nu|\e\rho^{d-1},
 \end{gather*}
 which imply \eqref{claim change of surface} for $C:=\max\{c_6,c_6c_3|\zeta|, C_1c_3|\zeta|,c_6\big(2c_3|\zeta| +C_2 \big)^{1-\alpha}\}$. 
\end{proof}

The following theorem constitutes the main result of this section. We shall see in the next section that its hypotheses are satisfied under the standard hypotheses of periodic or stochastic homogenisation.
\begin{theorem}\label{homogeneous homogenization}
    Let $f\in\mathcal{F}^\alpha$, $g\in\mathcal{G}^{\rm \infty}$, and let $F_\e:=F^{f_\e,g_\e}$ be the functionals introduced in Definition \ref{definition rescaled}. Assume that there exist functions $f_{\rm lim}\colon \Rdsym\to [0,+\infty)$ and $g_{\rm lim}\colon \Rd\times \Sd\to[0,+\infty)$ such that 
    \begin{gather}\label{def f hom 2}
        f_{\rm lim}(A)=\lim_{r\to+\infty}\frac{\m^{F^{f,g}}(\ell_A,Q(r x,r))}{r^d}\quad \text{for all $x\in\Rd$ and $A\in \Rdsym$,}\\\label{def g hom 2}\hspace{-0.2 cm}
        g_{\rm lim}(\zeta,\nu)=\lim_{r\to+\infty}\frac{\m^{F^{\infty}}(u_{rx,\zeta,\nu},Q_\nu(rx,r))}{r^{d-1}}\quad \text{for all $x\in\Rd$, $\zeta\in \Rd$, and $\nu\in \Sd$.}
    \end{gather}
   Then $f_{\rm lim}\in\mathcal{F}^\alpha$, $g_{\rm lim}\in\mathcal{G}^{\rm \infty}$, $g_{\rm lim}$ satisfies \eqref{homogeneity}, and  the following property holds:  for every positive sequence $\{\e_n\}_{n\in\N}$ converging to $0$ as $n\to+\infty$ and for every $U\in\U_c(\Rd)$  the sequence $\{F_{\e_n}(\cdot,U)\}_{n\in\N}$ $\Gamma$-converges to $F^{f_{\rm lim},g_{\rm lim}}(\cdot,U)$ with respect to the topology of $L^1_{\rm loc}(\Rd;\Rd)$. Moreover, we have the equality 
   \begin{equation}\label{eq omogenea}
   g_{\rm lim}(\zeta,\nu)=f^\infty_{\rm lim}(\zeta\odot\nu)\quad \text{ for all $\zeta\in\Rd$ and $\nu\in\Sd$,}
   \end{equation} which implies that for every $U\in\U(\Rd)$, $u\in {\rm BD}(U)$, and $B\in\B(U)$ we have 
   \begin{equation}\label{singular defin}
       F^{f_{\rm lim},g_{\rm lim}}(u,B)=\int_Bf_{\rm lim}(\E u)\,{\rm d}x+\int_Bf_{\rm lim}^\infty\Big(\frac{{\rm d}E^su}{{\rm d}|E^su|}\Big)\,{\rm d}|E^su|,
   \end{equation}
   where $E^su$ denotes the singular part of $Eu$ with respect to the Lebesgue measure. 
   \end{theorem}
   \begin{proof}
    Let us fix a sequence $\{\e_n\}_{n\in\N}\subset (0,1)$ converging to $0$ as $n\to +\infty$. Since $\{F_{\e_n}\}_{n\in\N}\subset \mathfrak{F}^{\alpha,\infty}$, to prove the $\Gamma$-convergence part of the statement it is enough to check that the hypotheses of Theorem \ref{theorem sufficient} are satisfied. Using Lemma \ref{lemma change of bulk variables} we see that for every $\rho\in (0,1)$, $n\in\N$, $x\in\Rd$, and $A\in \Rdsym$ we have 
       \begin{equation*}
             \frac{1}{\rho^d}\m^{F_{\e_n}}(\ell_A,Q(x,\rho))=\frac{\e_n^d}{\rho^d}\m^{F}(\ell_A,Q(\tfrac{x}{\e_n},\tfrac{\rho}{\e_n})),
       \end{equation*}
       so that, setting $r_n:=\rho/\e_n$, we may rewrite the previous equality as 
       \begin{equation*}
            \frac{1}{\rho^d}\m^{F_{\e_n}}(\ell_A,Q(x,\rho))=\frac{1}{r_n^d}\m^{F}(\ell_A,Q(r_n\tfrac{x}{\rho},r_n)).
       \end{equation*}
        By \eqref{def f hom 2} applied with $x$ replaced by $x/\rho$, we then obtain for every $A\in \Rdsym$ and $\rho\in (0,1)$ that 
       \begin{equation*}
           \lim_{n\to+\infty}
            \frac{1}{\rho^d}\m^{F_{\e_n}}(\ell_A,Q(x,\rho))=f_{\rm lim} (A),
       \end{equation*}
       which proves \eqref{1} for every $ x\in\Rd$ and  $A\in \Rdsym$. 

       Let us fix $x\in\Rd$, $\zeta\in \Rd$, and $\nu\in \Sd$.  Using Lemma \ref{lemma change of surface variables} we see that for every $\rho\in (0,1)$ and $n$ large enough we have
       \begin{gather*}
           \frac{\e_n^{d-1}}{\rho^{d-1}}\m^{F^{f^\infty,g}}(u_{\frac{x}{\e_n},\zeta,\nu},Q_\nu(\tfrac{x}{\e_n},\tfrac{\rho}{\e_n}))-C(\rho^\alpha +\e_n)\leq  \frac{1}{\rho^{d-1}}\m^{F_{\e_n}}(u_{x,\zeta,\nu},Q_\nu(x,\rho))\\\leq  \frac{\e_n^{d-1}}{\rho^{d-1}}\m^{F^{f^\infty,g}}(u_{\frac{x}{\e_n},\zeta,\nu},Q_\nu(\tfrac{x}{\e_n},\tfrac{\rho}{\e_n}))+C(\rho^\alpha+\e_n),
       \end{gather*}
       so that, setting again $r_n:=\rho/\e_n$, by \eqref{def g hom 2} we may pass to the limit first as $n\to +\infty$  and then as $\rho\to 0^+$ to obtain \begin{equation*}
           \lim_{\rho \to 0^+}\lim_{n\to+\infty}
            \frac{1}{\rho^{d-1}}\m^{F_{\e_n}}(u_{x,\zeta,\nu},Q_\nu(x,\rho))=g_{\rm lim} (\zeta,\nu),
       \end{equation*}
       proving \eqref{2} for every $x\in\Rd$, $\zeta\in \Rd$, and $\nu\in \Sd$.

       The fact that $f_{\rm lim}\in\mathcal{F}^\alpha$ and $g_{\rm lim}\in\mathcal{G}^\infty$ follows again by Theorem \ref{theorem sufficient}, while property \eqref{homogeneity} for $g_{\rm lim}$ follows from the observation that, being $f^\infty$ and $g^\infty$ both positively homogeneous of degree one (in the variables $A$ and $\zeta$, respectively), so is the function $t\mapsto \m^{F^{\infty}}(u_{x,t\zeta,\nu},Q_\nu(x,\rho))$. By \eqref{def g hom 2} this leads to the positive homogeneity of $g_{\rm lim}$.

       Finally, since for every $U\in\U_c(\Rd)$ the functional $F^{f_{\rm lim},g_{\rm lim}}(\cdot,U)$ is $L^1_{\rm loc}(\Rd;\Rd)$ lower semicontinuous and the functions $f_{\rm lim}$ and $g_{\rm lim}$ are independent of the variable $x$, equality \eqref{eq omogenea}
    
    \noindent follows from Lemma \ref{lower semicontinuity implies} below. Equality \eqref{singular defin} follows immediately from Definition \ref{def:integral functionals}.
   \end{proof}

For the application to stochastic homogenisation it is useful to obtain the conclusion of the previous theorem assuming only that the limits in \eqref{def f hom 2} and \eqref{def g hom 2} hold  on  countable dense collections of $A$, $\zeta$, and $\nu$. This is made possible by the following two lemmas.
\begin{lemma}\label{lemma: extension bulk}
    Let $f\in\mathcal{F}$, $g\in\mathcal{G}$, and $\mathbb{D}\subset \Rdsym$ be a dense subset. Assume that there exists a function $f_{\rm lim}\colon \mathbb{D}\to [0,+\infty)$ such that 
    \begin{equation}\label{net}
         f_{\rm lim}(A)=\lim_{r\to+\infty}\frac{\m^{F^{f,g}}(\ell_A,Q(r x,r))}{r^d}\quad \text{for every $x\in\Rd$ and $A\in \mathbb{D}$}.
    \end{equation}
    Then there exists a unique continuous  extension $f_{\rm lim}\colon \Rdsym\to [0,+\infty)$ and this extension is Lipschitz continuous with Lipschitz constant $c_5$ and  satisfies
    \begin{equation}\label{nest}
         f_{\rm lim}(A)=\lim_{r\to+\infty}\frac{\m^{F^{f,g}}(\ell_A,Q(r x,r))}{r^d}\quad \text{for every $x\in\Rd$ and $A\in \Rdsym$}.
    \end{equation}
\end{lemma}
\begin{proof}
  Arguing exactly as in the proof of \eqref{estimate for lipschiztianity}, we obtain
       \begin{equation*}
           |\m^{F^{f,g}}(\ell_{A_1},Q(rx,r))-  \m^{F^{f,g}}(\ell_{A_2},Q(rx,r))|\leq c_5|A_2-A_1|r^d \quad \text{ for every $A_1,A_2\in \Rdsym$}.
       \end{equation*}
        Combining the previous inequality with \eqref{net}, we obtain that there exists a unique continuous extension of $f_{\rm lim}$ and that this extension is $c_5$-Lipschitz continuous  and satisfies \eqref{nest}.
\end{proof}
\begin{lemma}\label{lemma: extension surface}
     Let $f\in\mathcal{F}$, $g\in\mathcal{G}$, and $\mathbb{D}_1\subset \Rd$ and $\mathbb{D}_2\subset \Sd$ be dense subsets. Assume that there exists a function $g_{\rm lim}\colon \mathbb{D}_1\times \mathbb{D}_2\to [0,+\infty)$ such that 
     \begin{equation}\label{hp theorem extension surf}
          g_{\rm lim}(\zeta,\nu)=\lim_{r\to+\infty}\frac{\m^{F^{\infty}}(u_{rx,\zeta,\nu},Q_\nu(rx,r))}{r^{d-1}}\quad \text{for all  $x\in\Rd$, $\zeta\in \mathbb{D}_1$, and $\nu\in \mathbb{D}_2$.}
     \end{equation}
     Then there exists a unique extension $g_{\rm lim}\colon \Rd\times \Sd\to [0,+\infty)$ such that 
     \begin{equation}\label{claim lemma extension}
        \hspace{-0.5 cm} g_{\rm lim}(\zeta,\nu)=\lim_{r\to+\infty}\frac{\m^{F^{\infty}}(u_{rx,\zeta,\nu},Q_\nu(rx,r))}{r^{d-1}}\quad \text{for all $x\in\Rd$, $\zeta\in \Rd$, and $\nu\in \Sd$.}
     \end{equation}
\end{lemma}
 \begin{proof}
       We begin by proving that there exists an extension 
       $g_{\rm lim}\colon\Rd\times \mathbb{D}_2\to[0,+\infty)$ such that
\begin{equation}\label{g hom 2 every}
g_{\rm lim}(\zeta,\nu)=\lim_{r\to+\infty}\frac{\m^{F^{\infty}}(u_{rx,\zeta,\nu},Q_\nu(rx,r))}{r^{d-1}}\quad \text{for all $x\in\Rd$, $\zeta\in \Rd$, and $\nu\in \mathbb{D}_2$.}
       \end{equation}
To this aim, we observe that by \eqref{lipschitz for iinfity} and \eqref{def sigma} we have
\begin{equation*}
           |\m^{F^{\infty}}(u_{x,\zeta_1,\nu},Q_\nu(rx,r))-  \m^{F^{\infty}}(u_{x,\zeta_2,\nu},Q_\nu(rx,r))|\leq \sigma_1|\zeta_1-\zeta_2|r^{d-1}
       \end{equation*}
for every $x\in\Rd$,  $\zeta_1,\zeta_2\in \Rd$, and $\nu\in \Sd$. This can be proved arguing as in the proof of inequality \eqref{sigma continuita g} in Lemma \ref{surface inclusion}.  Combining this inequality with \eqref{hp theorem extension surf}, we obtain that the limit in \eqref{g hom 2 every} exists for every $\zeta\in\Rd$ and $\nu\in \mathbb{D}_2$.

  We now prove that there exists and extension of $g_{\rm lim}$ satisfying \eqref{claim lemma extension}.
We begin by observing that, given $0<\eta<1$, by the continuity condition on the map $\nu\mapsto R_\nu$ introduced in (d) of Section \ref{sec: Notation} there exists $0<\delta_\eta<\eta$ 
such that 
\begin{equation}\label{inclusione cubi}Q_{\nu_1}(0,1)\subset\subset Q_{\nu_2}(0,1+\eta) \quad \text{for every $\nu_1,\nu_2\in \Sd_{\pm}$ with $|\nu_1-\nu_2|<\delta_\eta$.}\end{equation}

Let us fix $x\in\Rd$, $\zeta\in \Rd$,  $r>0$, and $\eta>0$. We observe that  \eqref{inclusione cubi} gives 
\begin{equation*}
    Q_{\nu_1}(rx,r)\subset\subset Q_{\nu_2}(rx,(1+\eta)r) \quad \text{for every $\nu_1,\nu_2\in \Sd_{\pm}$ with $|\nu_1-\nu_2|<\delta_\eta$.}
\end{equation*}
Hence, we can apply Lemma \ref{lemma poco used} with  $x_1=x_2=r x$, $\rho_1=r$,  $\rho_2=(1+\eta)r$, and $\nu_1,\nu_2\in\Sd_{\pm}$ with $|\nu_1-\nu_2|<\delta_\eta$. We observe that $F^{\infty}$ satisfies the upper estimate in (c) of Definition \ref{abstract functionals} with $c_4=0$ because of \eqref{bound finfty} and Proposition \ref{prop integrals are abstract}, so that we can omit the term containing $c_4$ in Lemma \ref{lemma poco used}.
Using  the inequality $(1+\eta)^{d-1}-1\le 2^{d-1}\eta$, for every $\nu_1,\nu_2\in \Sd_{\pm}$ with $|\nu_1-\nu_2|<\delta_\eta<\eta$ we then obtain
 \begin{gather*}
\m^{F^{\infty}}\!(u_{rx ,\zeta,\nu_2},Q_{\nu_2}(r x, (1+\eta)r))\leq   \m^{F^{\infty}}
    (u_{r x,\zeta,\nu_1},Q_{\nu_1}(r
    x,r))+c_3|\zeta|(2^{d-1}\eta+\omega(0,\eta)) r^{d-1}.
        \end{gather*}
We set $x_\eta:=\frac{x}{1+\eta}$ and $r_\eta:=(1+\eta)r$. Dividing the previous inequality by $r^{d-1}$, we obtain
        \begin{gather}\label{stima continuità nu}
\hspace{-0.38 cm}\frac{\m^{F^{\infty}}\!(u_{r_\eta x_\eta ,\zeta,\nu_2},Q_{\nu_2}(r_\eta x_\eta, r_\eta))}{r_\eta^{d-1}}\leq   \frac{\m^{F^{\infty}}
    (u_{r x,\zeta,\nu_1},Q_{\nu_1}(r
    x,r))}{r^{d-1}}+c_3|\zeta|(2^{d-1}\eta+\omega(0,\eta)).
        \end{gather}
Exchanging the roles of $\nu_1$ and $\nu_2$ we also get
  \begin{gather}\label{stima continuityà nu 2}
\hspace{-0.38 cm}\frac{\m^{F^{\infty}}\!(u_{r_\eta x_\eta ,\zeta,\nu_1},Q_{\nu_1}(r_\eta x_\eta, r_\eta))}{r_\eta^{d-1}}\leq   \frac{\m^{F^{\infty}}
    (u_{r x,\zeta,\nu_2},Q_{\nu_2}(r
    x,r))}{r^{d-1}}+c_3|\zeta|(2^{d-1}\eta+\omega(0,\eta)).
        \end{gather}
Moreover, applying  \eqref{stima continuità nu} with $x$ replaced by $x^{\eta}:=(1+\eta)x$ we obtain
\begin{gather}\label{stima continuità nu 3}
\hspace{-0.38 cm}\frac{\m^{F^{\infty}}\!(u_{r_\eta x ,\zeta,\nu_2},Q_{\nu_2}(r_\eta x, r_\eta))}{r_\eta^{d-1}}\leq   \frac{\m^{F^{\infty}}
    (u_{r x^\eta,\zeta,\nu_1},Q_{\nu_1}(r
    x^\eta,r))}{r^{d-1}}+c_3|\zeta|(2^{d-1}\eta+\omega(0,\eta)).
        \end{gather}

Let us fix $\nu\in \Sd_{\pm}$. For every $\eta>0$ we can find $\nu_\eta\in \Sd_{\pm}\cap \mathbb{D}_2$ with $|\nu_\eta-\nu|<\delta_\eta$, where $0<\delta<\eta$ is the constant given by \eqref{inclusione cubi}. Using \eqref{stima continuityà nu 2} with $\nu_1=\nu_\eta$ and $\nu_2=\nu$,  by \eqref{g hom 2 every}, which clearly holds with $r$ replaced by $r_\eta$, we obtain 
\begin{gather}\label{stima novella}
    g_{\rm lim}(\zeta,\nu_\eta)-c_3|\zeta|(2^{d-1}\eta+\omega(0,\eta))\leq \liminf_{r\to+\infty}\frac{\m^{F^{\infty}}
    (u_{r x,\zeta,\nu},Q_{\nu}(r
    x,r))}{r^{d-1}},
\end{gather}
while using \eqref{stima continuità nu 3} with $\nu_1=\nu_\eta$ and $\nu_2=\nu$  by \eqref{g hom 2 every} we get
\begin{equation}\label{stima continuit}
\limsup_{n\to+\infty}\frac{\m^{F^{\infty}}
    (u_{r x,\zeta,\nu},Q_{\nu}(r
    x,r))}{r^{d-1}}\leq g_{\rm lim}(\zeta,\nu_\eta)+c_3|\zeta|(2^{d-1}\eta+\omega(0,\eta)),
\end{equation}
where we have also used that by \eqref{g hom 2 every} the limit appearing the definition of $g_{\rm lim}(\zeta,\nu_\eta)$ is the same for $x$ and $x^\eta$. Combining \eqref{stima novella} and \eqref{stima continuit}, we obtain 
\begin{gather}\notag
    \limsup_{r\to+\infty}
    \frac{\m^{F^{\infty}}\!
    (u_{r x,\zeta,\nu},Q_{\nu}(r
    x,r))}{r^{d-1}}-\liminf_{r\to+\infty}\frac{\m^{F^{\infty}}\!
    (u_{r x,\zeta,\nu},Q_{\nu}(r
    x,r))}{r^{d-1}}\leq  
   2c_3|\zeta|(2^{d-1}\eta+\omega(0,\eta)).
\end{gather}
Letting $\eta\to 0^+$ we conclude the proof of the lemma.
     \end{proof}
We conclude this section by showing that, if  $F^{f,g}(\cdot,U)$ is $L^1(U;\Rd)$ lower semicontinuous for some bounded open set $U\subset \Rd$,  then $f^\infty(\zeta\odot\nu)=g(\zeta,\nu)$, provided that $f$ and $g$ are independent of $x$ and  $g$ is positively homogeneous of degree one in  $\zeta$. Although  variants of this result are well-known (see, for instance, \cite[Step 3 of Theorem 6.14]{CaroccFocardiVan}),  we present here a complete proof in order to conclude the proof of Theorem \ref{homogeneous homogenization}.
\begin{lemma}\label{lower semicontinuity implies}
   Let $f\colon \Rdsym\to [0,+\infty)$ and $g\colon\Rd\times \Sd\to [0,+\infty)$ be Borel functions satisfying {\rm (f2), (g2), (g3)}, and \eqref{homogeneity}. Assume that there exists a non-empty bounded open set $U\subset \Rd$ such that the functional $F^{f,g}(\cdot,U)$ is $L^1(U;\Rd)$-lower semicontinuous. Then we have 
   \begin{equation}\notag 
    g(\zeta,\nu)=f^\infty(\zeta\odot\nu)
   \end{equation}
   for every $\zeta\in\Rd$ and $\nu\in\Sd$.
   \end{lemma}
\begin{proof}
Let us fix $\zeta\in\Rd$,  $\nu\in\Sd$, and $x\in\Rd$ such that  
\begin{gather}\label{berc}
    0<\hd(\Pi^\nu_x\cap U),\\\label{Lebesgue point}
    0 \text{ is a Lebesgue point of the function }\R\ni t\mapsto \hd(\Pi^\nu_{x+t\nu}\cap U),
\end{gather}
where we recall that for $\nu_0\in\Sd$ and $x_0\in\Rd$ the symbol $\Pi^{\nu_0}_{x_0}$ denotes the hyperplane passing  through $x_0$ and orthogonal to $\nu_0$, given by $\{y\in\Rd:(y-x_0)\cdot\nu_0=0\}$.
Thanks to the Fubini Theorem and the Lebesgue Differentiation Theorem the set of points $x\in\Rd$ for which the two previous conditions are satisfied has positive $\Ld$ measure. 

We now show that 
   \begin{equation}\label{g leq fhom}
     g(\zeta,\nu)\leq {f}^\infty(\zeta\odot\nu).   \end{equation}
To prove this,  we approximate the function $u_{x,\zeta,\nu}$ by the sequence of piecewise affine functions $\{u_n\}_{n\in\N}\subset  {\rm BD}(U)$ defined for every $y\in U$ by
   \begin{equation}\notag u_n(y):=
       \begin{cases}
           0 &\text{ if }(y-x)\cdot\nu\leq -\tfrac{1}{2n},\\
          
           n((y-x)\cdot\nu+\frac{1}{2n})  \zeta &\text{ if }-\tfrac{1}{2n}\leq (y-x)\cdot\nu\leq \tfrac{1}{2n},\\
         
           \zeta &\text{ if }(y-x)\cdot\nu\geq \tfrac{1}{2n}.
       \end{cases}
   \end{equation}
   Clearly $\{u_n\}_{n\in\N}$ converges to $u_{x,\zeta,\nu}$ strongly in $L^1(U;\Rd)$ as $n\to+\infty$. 
   
   Setting $H_n:=\{y\in U:-\tfrac{1}{2n}\leq (y-x)\cdot\nu\leq \tfrac{1}{2n}\}$,  one checks immediately that,
   \begin{equation}\label{gradient of aggine}
       \E u_n= n\zeta\odot\nu\chi_{H_n}\quad \text{ $\Ld$-a.e.\ in } U,
   \end{equation}
   where $\chi_{H_n}$ is the characteristic function of $H_n$. Moreover, the Fubini Theorem  implies that 
   \begin{equation}\notag 
       \Ld(H_n)=\int_{-\frac{1}{2n}}^{\frac{1}{2n}}\hd(\Pi^\nu_{x+t\nu}\cap U)\,{\rm d}t,
   \end{equation}
   which together with \eqref{Lebesgue point} this gives 
   \begin{equation}\label{nuova sfubinata}
       \lim_{n\to+\infty}n\Ld(H_n)=\hd(\Pi^\nu_x\cap U).
   \end{equation}
   
   In light of the lower semicontinuity of $F^{f,g}(\cdot,U)$, a direct computation shows that
   \begin{gather*}
       f(0)\Ld(U)+g(\zeta,\nu)\hd(\Pi^\nu_x\cap U)=F^{f,g}(u_{x,\zeta,\nu},U) \leq 
       \liminf_{n\to+\infty}F^{f,g}(u_n,U)\\
       =\liminf_{n\to+\infty}\Big(\int_{H_n}f(\E u_n)\,{\rm d}x+f(0)\Ld(U\setminus H_n)\Big).
   \end{gather*}
   Since $\Ld(H_n)$ converges to $0$ as $n\to+\infty$, the previous inequality gives
   \begin{equation}\label{8}
       g(\zeta,\nu)\hd(\Pi^\nu_x\cap U)\leq \liminf_{n\to+\infty}\int_{H_n}f(\E u_n)\,{\rm d}x.
   \end{equation}
   Using \eqref{gradient of aggine} we see  that
   \begin{gather}\label{9}
    \int_{H_n}f(\E u_n)\,{\rm d}x=f(n\zeta\odot\nu)\Ld(H_n).
   \end{gather}
   Combining \eqref{nuova sfubinata}-\eqref{9}, recalling the definition of $f^\infty$ given by \eqref{def recession},   we have
   \begin{gather*}
g(\zeta,\nu)\hd(\Pi^\nu_x\cap U)\leq \liminf_{n\to+\infty}\frac{f(n\zeta\odot\nu)}{n}(n\Ld(H_n))\leq f^\infty(\zeta\odot\nu)\hd(\Pi^\nu_x\cap U),
   \end{gather*}
   which by \eqref{berc} implies  \eqref{g leq fhom}.

  We now prove
  \begin{equation}\label{other}
      f^\infty(\zeta\odot\nu)\leq g(\zeta,\nu).
  \end{equation}
  To this aim, we set  $T>{\rm diam}(U)$ and let $x\in U$. By the Fubini Theorem we have
  \begin{gather}
    \notag\Ld(U)=\int_{-T}^{T}\hd(\Pi^\nu_{x+t\nu}\cap U)\,{\rm d}t.
  \end{gather}
  By the well-know approximation properties of Lebesgue integrals by means of Riemann sums (see \cite{Hahn}, \cite[Page 63]{Doob}, or \cite[Lemma 4.12]{FrancfortToad})  we can select $x\in U$ in such a way that 
  \begin{equation}
   \label{manz}\Ld(U)=\lim_{n\to+\infty}\frac{T}{n}\sum_{i=-n+1}^n\hd\big(\Pi^\nu_{x+\frac{i-1}{n}T\nu}\cap U\big).
  \end{equation}

We set $A:=\zeta\otimes\nu$ and fix $t>0$. To prove \eqref{other}, we  construct a sequence of pure jump functions approximating $\ell_{tA}$ in $L^1(U;\Rd)$. Given $n\in\N$,  let $\sigma_n\colon(-T,T)\to \R$ be defined by 
   \begin{equation}\notag 
    \sigma_n(s):=\sum_{i=-n+1}^{n}\tfrac{iT}{n}\chi_{(\tfrac{i-1}{n}T,\tfrac{i}{n}T\,)}(s)\quad \text{ for every }s\in(-T,T).
   \end{equation}
   For every $n\in\N$, let $u_n\in {\rm BD}(U)$ be the function defined by
   \begin{equation}\notag
u_n(y):= 
    t\zeta  \sigma_n((y-x)\cdot\nu)\quad \text{for every }y\in U.
   \end{equation}
   It is easy to see that $\{u_n\}_{n\in\N}$ converges to $\ell_{tA}$ strongly in $L^1(U;\Rd)$ as $n\to+\infty$. Thus, by the $L^1$-lower  semicontinuity  of $F^{f,g}(\cdot, U)$ we obtain
   \begin{equation}\label{bound fhat}
       f(t A^{\rm sym})\Ld(U)=\int_{U}f(t A^{\rm sym})\,{\rm d}x=F^{f,g}(\ell_{tA},U)\leq \liminf_{n\to+\infty}F^{f,g}(u_n,U).
   \end{equation}
On the other hand, using the homogeneity of $g$ given by \eqref{homogeneity}  we immediately check that
   \begin{gather*}
       F^{f,g}(u_n,U)=f(0)\Ld(U)+\sum_{i=-n+1}^ng\Big(\frac{t\zeta}{n}T,\nu\Big)\hd\big(\Pi^\nu_{x+\frac{i-1}{n}T\nu}\cap U)\\
       = f(0)\Ld(U)+tg(\zeta,\nu)\frac{T}{n}\sum_{i=-n+1}^n\hd\big(\Pi^\nu_{x+\frac{i-1}{n}T\nu}\cap U).  
   \end{gather*}
   Taking the liminf as $n\to+\infty$ and using \eqref{manz} and \eqref{bound fhat} we get
\begin{equation}\notag 
   f(t A^{\rm sym})\Ld(U)\leq f(0)\Ld(U)+tg(\zeta,\nu)\Ld(U).
\end{equation}
Hence, we get  $\frac{1}{t}f(t A^{\rm sym})\leq \frac{1}{t}f(0)+g(\zeta,\nu)$, and letting $t\to+\infty$   we obtain \eqref{other}. 
\end{proof}

\section{Stochastic homogenisation}\label{subsection stochastic}
In this section we use the results of the previous section to study the problem of stochastic homogenisation of free discontinuity functionals defined on the space ${\rm BD}$. 
 In the following we still assume that the modulus of continuity introduced in \eqref{slope at infty for sigma} satisfies \eqref{def sigma}. We begin by introducing the probabilistic setting we will use to deal with this problem.

\medskip

Throughout this section  $(\Omega, \mathcal{T}, P )$  is a probability space endowed with  a group $(\tau_z )_{z\in\Z^d}$ of $P$-preserving transformations on
 $(\Omega,\mathcal{T},P)$, i.e., a family $(\tau_z )_{z\in\Z^d}$ of $\mathcal{T}$-measurable bijective maps $\tau_z\colon \Omega\to\Omega$ such that
 \begin{enumerate}
\item [(a)]    $\tau_{z_1}\circ\tau_{z_2}=\tau_{z_1+z_2}$ for every $z_1,z_2\in\Z^d$;
\item [(b)] $P (\tau_z^{-1} (E))= P(E)$ for every $E \in\mathcal{T}$ and $z\in\>\Z^d$. 
 \end{enumerate}
Note that from (a) and the bijectivity it follows that for $z=0$ the map $\tau_0$ is the identity on $\Omega$. We 
say that a group $(\tau_z)_{z\in\Z^d}$ of $P$-preserving transformations is ergodic if every set  $E\in\mathcal{T}$ with $\tau_z(E)=E$ for every $z\in\Z^d$ has probability either $0$ or $1$.
 In analogy with \cite{DalToaHomo,DalMasoDonati2025}, we introduce the following two classes of stochastic integrands.
 \begin{definition}\label{def:stochastic integrands}
      $\mathcal{F}^\alpha_{\rm stoc}$ denotes the set of all  $\mathcal{T}\otimes\mathcal{B}(\Rd\times\Rdsym)$-measurable functions $f\colon\Omega\times\Rd\times\Rdsym\to[0,+\infty)$ such that for every $\omega\in\Omega$ the function $f(\omega):=f(\omega,\cdot,\cdot)$ belongs to $\mathcal{F}^\alpha$  and the following stochastic periodicity condition holds: 
         \begin{equation*}
    f(\omega,x+z,A)=f(\tau_z(\omega),x,A)
       \end{equation*}
       for every $\omega\in\Omega$, $z\in\Z^d$, $x\in\Rd$, and $A\in\Rdsym$.
    $\mathcal{G}_{\rm stoc}^{\rm \infty}$ denotes 
    the set of all $\mathcal{T}\otimes\mathcal{B}(\Rd\times\Rd\times\mathbb{S}^{d-1})$-measurable functions $g\colon\Omega\times\Rd\times\Rd\times\mathbb{S}^{d-1}\to[0,+\infty)$ such that for every $\omega\in\Omega$ the function $g(\omega):=g(\omega,\cdot,\cdot,\cdot)$ belongs to $\mathcal{G}^{\rm \infty}$, and the following stochastic periodicity condition holds:
         \begin{equation*} g(\omega,x+z,\zeta,\nu)=g(\tau_z(\omega),x,\zeta,\nu)
       \end{equation*}
        for every $\omega\in\Omega$, $z\in\Z^d$, $x,\zeta\in\Rd$, and $\nu\in\mathbb{S}^{d-1}$.
 \end{definition}

We recall the definition of subadditive process.
Let $\mathcal{R}$ the collection of half-closed rectangles defined by
\begin{align*}
    \mathcal{R}:=\{R\subset\Rd\colon& R=[a_1,b_1)\times...\times [a_d,b_d) \text{ with } a_i<b_i \text{ for }i\in\{1,...,d\}\}.
\end{align*}
Given $R\in\mathcal{R}$ its interior is denoted by $R^\circ$.
 We also introduce the completion of $(\Omega,\mathcal{T},P)$, denoted by $(\Omega,\widehat{\mathcal{T}},\widehat{P})$. It is immediate to see that $(\tau_z)_{z}$ is a group of $P$-preserving transformation also on $(\Omega,\widehat{\mathcal{T}},\widehat{P})$. 
\begin{definition}\label{subadditive process}
A function $\mu\colon\Omega\times \mathcal{R}\to\R$ is  said to be a covariant subadditive process with respect to $(\tau_{z})_{z\in\Z^d}$ if the following properties are satisfied
 \begin{enumerate}
\item[(a)]  for every $R\in\mathcal{R}$ the function $\mu(\cdot, R)$  is $\widehat{\mathcal{T}}$-measurable;
\item[(b)]  for every $\omega\in\Omega$, $R \in\mathcal{R}$, and $z \in \Z^d$ we have  $\mu(\omega,R + z)=\mu(\tau_z (\omega), R)$;
\item[(c)] given $R\in\mathcal{R}$ and  a finite partition $(R_i)_{i=1}^n\subset\mathcal{R}$ of $R$, we have
\begin{equation*}
   \mu(\omega,R)\leq \sum_{i=1}^n\mu(\omega,R_i)
\end{equation*}
 for every $\omega\in\Omega$;
\item[(d)] there exists $C>0$ such that $0\leq \mu(\omega,R)\leq C\Ld(R)$ for every $\omega\in\Omega$
and $R\in\mathcal{R}$.
 \end{enumerate} 
\end{definition}

In the following we will make substantial use of the  Subadditive Ergodic Theorem of Akcoglu and Krengel {\cite[Theorem 2.7]{AkcogluKrengel}}. In particular, we will use the version of this theorem  stated in  \cite[Proposition 1]{DaLMasoModica}.
 \begin{theorem}\label{thm:subadditive ergodic theorem}
 Let $\mu$ be a subadditive process with respect to the group $(\tau_z)_{z\in\Z^d}$. Then there exist  $\Omega'\in \mathcal{T}$, with $P(\Omega')=1$, and  $\varphi\colon\Omega\to[0,+\infty)$ such that
 \begin{equation*}
     \lim_{r\to+\infty} \frac{\mu(\omega,\widetilde{Q}(rx,r))}{r^d}=\varphi(\omega)
\end{equation*}
for every $x\in \Rd$ and every $\omega \in \Omega' $, where $\widetilde{Q}(rx,r):=[rx_1-\frac{r}{2},rx_1+\frac{r}{2})\times\dots\times [rx_d-\frac{r}{2},rx_d+\frac{r}{2})$.

If the group $(\tau_z )_{z\in\Z^d}$ is also ergodic, then $\varphi$ is constant $P$-a.e.
 \end{theorem}

The following is the main result concerning the bulk part.

\begin{proposition}\label{stochastic bulk}
Let $f \in \mathcal{F}^\alpha_{\rm stoc}$ and  $g \in\mathcal{G}^{\rm \infty}_{\rm stoc}$. Then there exist  $\Omega'\in \mathcal{T}$, with $P(\Omega')=1$, and a  function $f_{\rm lim}\colon \Omega\times \Rdsym \to [0,+\infty)$, with $f_{\rm lim}(\omega,\cdot)$ continuous for every $\omega\in\Omega'$ and $f_{\rm lim}(\cdot,A)$ $\mathcal{T}$-measurable for every $A\in\Rdsym$, such that 
\begin{equation}\label{statement bulk ergodico}
f_{\rm lim }(\omega, A)=\lim_{r\to+\infty}\frac{\m^{E^{f(\omega),g(\omega)}}(\ell_A,Q(rx,r))}{r^d}
\end{equation}
for every $\omega\in\Omega'$, $x\in\Rd$, and $A\in\Rdsym$. 
\noindent If, in addition,  $(\tau_z)_{z\in\Z^d}$ is ergodic, by choosing $\Omega'$ appropriately we have that $f_{\rm lim}$ is
independent of $\omega$. 
\end{proposition}

\begin{proof}
Let us fix $A\in\Rdsym$ and consider the function $\Phi_{A}\colon\Omega\times\mathcal{R}\to[0,+\infty)$ defined by 
     \begin{equation}\notag 
         \Phi_{A}(\omega,R):=\m^{E^{f(\omega),g(\omega)}}(\ell_A,R^\circ).
     \end{equation}
 This function defines a covariant subadditive process.
Indeed,  it is easy to check that properties (b), (c), and (d) of Definition \ref{subadditive process} hold (see \cite[Theorem 9.1]{CagnettiAnnals} for the details), while  the proof of the measurability property (a) can be obtained by slightly modifying the Appendix of \cite{CagnettiAnnals}, replacing $Du$ and $\nabla u$ by $Eu$ and $\E u$.

Let $\mathbb{D}$ be a countable dense subset of $\Rdsym$. We may then apply Theorem \ref{thm:subadditive ergodic theorem} with $\mu=\Phi_A$ for every $A\in \mathbb{D}$ to obtain a set $\Omega'\in\mathcal{T}$, with $P(\Omega')=1$, and a function $\phi_A\colon\Omega\to[0,+\infty)$ such that 
\begin{equation*}
   \phi_A(\omega)= \lim_{r\to+\infty}\frac{\Phi_A(\omega,\widetilde{Q}(rx,r))}{r^d}=\lim_{r\to+\infty}\frac{\m^{F^{f(\omega),g(\omega)}}(\ell_A,Q(rx,r))}{r^d}
\end{equation*}
for every $\omega\in\Omega'$, $x\in\Rd$, and $A\in \mathbb{D}$. By Lemma \ref{lemma: extension bulk}
for every $\omega\in\Omega'$ there exists a unique continuous function $f_{\rm lim}(\omega,\cdot)\colon\Rdsym\to [0,+\infty)$ such that $f_{\rm lim}(\omega,A)=\phi_A(\omega)$ for every $A\in\Rdsym$. Moreover, $A\mapsto f_{\rm lim}(\omega,A)$ is $c_5$-Lipschitz continuous and satisfies \eqref{statement bulk ergodico}. 

Finally, the statement concerning the ergodic hypothesis follows from the last part of Theorem \ref{thm:subadditive ergodic theorem}.
\end{proof}

The following proposition collects the main results concerning the surface term.
\begin{proposition}\label{stochastic surface}
    Let $f \in \mathcal{F}^\alpha_{\rm stoc}$ and  $g \in\mathcal{G}^{\rm \infty}_{\rm stoc}$. Then there exist  $\Omega'\in \mathcal{T}$, with $P(\Omega')=1$, and a  function $g_{\rm lim}\colon \Omega\times \Rd\times \Sd \to [0,+\infty)$ such that 
    \begin{gather*}
        \zeta \mapsto g_{\rm lim}(\omega,\zeta,\nu) \text{ is Lipschitz continuous for every }\omega\in \Omega' \text{ and }\nu\in\Sd,\\
        \Sd_{\pm}\ni\nu\mapsto g_{\rm lim}(\omega,\zeta,\nu) \text{ are continuous for every }\omega\in \Omega' \text{ and }\zeta\in\Rd,\\
        \omega\mapsto g_{\rm lim}(\omega,\zeta,\nu) \text{ is $\mathcal{T}$-measurable for every }\zeta\in\Rd \text{ and }\nu\in\Sd,
    \end{gather*}
    and 
\begin{equation}\label{statement surface ergodico}
g_{\rm lim }(\omega, \zeta,\nu)=\lim_{r\to+\infty}\frac{\m^{{F^\infty(\omega)}}(u_{rx,\zeta,\nu},Q_\nu(rx,r))}{r^{d-1}}
\end{equation}
for every $\omega\in\Omega'$, $x\in\Rd$, $\zeta\in\Rd$, and $\nu\in\Sd$, where $F^\infty$ is the functional introduced in Definition \ref{def Finfty}. 
\noindent If, in addition,  $(\tau_z)_{z\in\Z^d}$ is ergodic, by choosing $\Omega'$ appropriately we have that $g_{\rm lim}$ is
independent of $\omega$. 
\end{proposition}

\begin{proof}
It is enough to repeat the arguments of \cite[Proposition 9.4 and 9.5]{CagnettiAnnals}, replacing Step 2 of their Proposition 9.4 and 9.5  by our Lemma \ref{lemma: extension surface}.
\end{proof}

Collecting the results of Propositions \ref{stochastic bulk} and \ref{stochastic surface} we obtain the following almost sure convergence result.
\begin{theorem}\label{stochastic homo}
 Let $f \in \mathcal{F}^\alpha_{\rm stoc}$,   $g \in\mathcal{G}^{\rm \infty}_{\rm stoc}$ and let $F_\e(\omega):=F^{f_\e(\omega),g_\e(\omega)}$ be the functionals introduced in Definition \ref{definition rescaled}. Then there exist  $\Omega'\in \mathcal{T}$, with $P(\Omega')=1$, and functions $f_{\rm lim}\in\mathcal{F}_{\rm stoc}^\alpha$ and $g_{\rm lim}\in \mathcal{G}^{\rm \infty}_{\rm stoc}$, with $g_{\rm lim}(\omega)$ satisfying \eqref{homogeneity} for every $\omega\in\Omega'$, and with the following property: for every sequence $\{\e_n\}_{n\in\N}\subset (0,1)$ converging to $0$ as $n\to+\infty$, $U\in \U_c(\Rd)$, and $\omega\in\Omega'$ we have that $\{F_{\e_n}(\omega)(\cdot,U)\}_{n\in\N}$ $\Gamma$-converges, with respect to the topology of $L^1_{\rm loc}(\Rd;\Rd)$, to $F^{f_{\rm lim}(\omega),g_{\rm lim}(\omega)}(\cdot,U)$ as $n\to+\infty$. Moreover, we have the equality 
 \begin{equation}\label{g=fstoc}
 g_{\rm lim}(\omega,\zeta,\nu)=f^\infty_{\rm lim}(\omega,\zeta\odot\nu)\quad \text{ for every $\zeta\in\Rd$ and $\nu\in\Sd$.}
 \end{equation}
If, in addition,  $(\tau_z)_{z\in\Z^d}$ is ergodic, by choosing $\Omega'$ appropriately we have that $f_{\rm lim}$ and $g_{\rm lim}$ are independent of $\omega$.
\end{theorem}
\begin{proof}
Thanks to Propositions \ref{stochastic bulk} and \ref{stochastic surface}, there exist $\Omega'\in\mathcal{T}$, with $P(\Omega')=1$, $f_{\rm lim}\colon\Omega\times \Rdsym\to [0,+\infty)$, and $g_{\rm lim}\colon\Omega\times \Rd\times \Sd\to [0,+\infty)$ such that \eqref{statement bulk ergodico} and \eqref{statement surface ergodico} hold. Thus, by Theorem \ref{homogeneous homogenization} we obtain the two inclusions $f_{\rm lim}\in\mathcal{F}^\alpha_{\rm stoc}$ and $g_{\rm  lim}\in\mathcal{G}^{\rm \infty}_{\rm stoc}$, together with the fact that $g_{\rm lim}(\omega)$ satisfies   \eqref{homogeneity} for every $\omega\in\Omega'$, as well as the $\Gamma$-convergence part of the statement and equality \eqref{g=fstoc}. 

The last part of the statement follows from the final parts of the statements of Propositions \ref{stochastic bulk} and \ref{stochastic surface}. 
\end{proof}

As a particular case of the previous theorem we deduce the following result in the periodic deterministic setting.
\begin{corollary}\label{periodic corollary}
    Let $f\in\mathcal{F}^\alpha$, $g\in\mathcal{G}^{\rm \infty}$, and let $F_\e:=F^{f_\e,g_\e}$ be the functional introduced in Definition \ref{definition rescaled}. Assume that both $f$ and $g$ are $1$-periodic in the $x$ variable. Then there exist functions $f_{\rm lim}\in \mathcal{F}^\alpha$ and $g_{\rm lim}\in \mathcal{G}^{\rm \infty}$, with $f_{\rm lim}$ satisfying \eqref{def f hom 2}, and with $g_{\rm lim}$ satisfying \eqref{homogeneity} and \eqref{def g hom 2},  and consequently the following property holds:  for every sequence $\{\e_n\}_{n\in\N}\subset (0,1)$ converging to $0$ as $n\to+\infty$ and for every $U\in\U_c(\Rd)$  the sequence $\{F_{\e_n}(\cdot,U)\}_{n\in\N}$ $\Gamma$-converges to $F^{f_{\rm lim},g_{\rm lim}}(\cdot,U)$ with respect to the topology of $L^1_{\rm loc}(\Rd;\Rd)$. Moreover, we have the equality $$g_{\rm lim}(\zeta,\nu)=f^\infty_{\rm lim}(\zeta\odot\nu)\quad \text{ for every $\zeta\in\Rd$ and $\nu\in\Sd$.}$$
\end{corollary}
\begin{proof}
    It is enough to apply Theorem \ref{stochastic homo} in the case where  $\Omega$ consists of a single point.
\end{proof}

\textbf{Acknowledgements.}
 This paper is based on work supported by the National Research Project PRIN 2022J4FYNJ “Variational methods for stationary and evolution problems
with singularities and interfaces” funded by the Italian Ministry of University and Research. 
Davide Donati is a member of GNAMPA of INdAM.

\bibliographystyle{siam}
\bibliography{Sources}
\end{document}